\documentclass{article}

\usepackage{amssymb,amsthm,amsmath,amsfonts}
\usepackage{mathrsfs}
\usepackage{tikzsymbols}
\usepackage{tikz}
\usepackage{ytableau}
\usepackage[all]{nowidow}
\usepackage{textcase}
\usepackage{enumitem}
\usepackage[longnamesfirst]{natbib}
\usepackage{etoolbox} % Use '\AfterEndEnvironment' to avoid indent after theorems

\usepackage[hidelinks]{hyperref}

\makeatletter
\def\bea#1\ena{\begin{align}#1\end{align}}
\def\beas#1\enas{\begin{align*}#1\end{align*}}
\def\beq#1\enq{\begin{equation}#1\end{equation}}
\newcommand{\Goodset}{\ifmmode\raisebox{-0.15ex}{\Smiley}\else\Smiley\fi}

\newcommand{\ignore}[1]{}

\newcommand{\qmq}[1]{\quad\text{#1}\quad}
\newcommand{\qm}[1]{\quad\mbox{#1}}

\newcommand{\Hyp}{\mathop{\mathrm{Hyp}}}

\newcommand{\esup}{\mathop{\mathrm{ess\,sup}}}
\newcommand{\ER}{Erd\H{o}s-R\'enyi}

\def\@noindentfalse{\global\let\if@noindent\iffalse}
\def\@noindenttrue {\global\let\if@noindent\iftrue}
\def\@aftertheorem{%
  \@noindenttrue
  \everypar{%
    \if@noindent%
      \@noindentfalse\clubpenalty\@M\setbox\z@\lastbox%
    \else%
      \clubpenalty \@clubpenalty\everypar{}%
    \fi}}

\theoremstyle{plain}
\newtheorem{theorem}{Theorem}[section]
\AfterEndEnvironment{theorem}{\@aftertheorem}

\AfterEndEnvironment{definition}{\@aftertheorem}
\newtheorem{lemma}[theorem]{Lemma}
\AfterEndEnvironment{lemma}{\@aftertheorem}

\AfterEndEnvironment{corollary}{\@aftertheorem}

\theoremstyle{definition}
\newtheorem{remark}[theorem]{Remark}
\AfterEndEnvironment{remark}{\@aftertheorem}

\AfterEndEnvironment{example}{\@aftertheorem}

\AfterEndEnvironment{proof}{\@aftertheorem}

\renewcommand{\cite}{\citet}

\numberwithin{equation}{section}

\renewcommand\section{\@startsection {section}{1}{\z@}%
	{-3.5ex \@plus -1ex \@minus -.2ex}%
	{1.3ex \@plus.2ex}%
	{\center\small\sc\mathversion{bold}\MakeTextUppercase}}

\def\subsection#1{\@startsection {subsection}{2}{0pt}%
	{-3.5ex \@plus -1ex \@minus -.2ex}%
	{1ex \@plus.2ex}%
	{\bf\mathversion{bold}}{#1}}

\def\subsubsection#1{\@startsection{subsubsection}{3}{0pt}%
	{\medskipamount}%
	{-10pt}%
	{\normalsize\itshape}{\kern-2.2ex. #1.}}

\def\note#1{\par\smallskip%
\noindent\kern-0.01\hsize%
{\setlength\fboxrule{0pt}\fbox{\setlength\fboxrule{0.5pt}\fbox{%
\llap{$\boldsymbol\Longrightarrow$ }%
\vtop{\hsize=0.98\hsize\parindent=0cm\small\rm #1}%
\rlap{$\enskip\,\boldsymbol\Longleftarrow$}
}}}%
}

\def\commandfactory#1#2#3{%
	\expandafter\def\csname #2#1\endcsname{{\csname #3\endcsname{#1}}}}

\let\original@left\left
\let\original@right\right
\renewcommand{\left}{\mathopen{}\mathclose\bgroup\original@left}
\renewcommand{\right}{\aftergroup\egroup\original@right}

\newcommand{\bigo}{\mathop{{}\mathrm{O}}\mathopen{}}
\newcommand{\lito}{\mathop{{}\mathrm{o}}\mathopen{}}

\newcounter{ctr}
\loop
\stepcounter{ctr}
\edef\X{\@Alph\c@ctr}%
\expandafter\commandfactory \X c {mathcal}
\expandafter\commandfactory \X I {mathbb}
\expandafter\commandfactory \X s {mathscr}
\expandafter\commandfactory \X b {boldsymbol}
\ifnum\thectr<26
\repeat

\newcommand{\Var}{\mathop{\mathrm{Var}}\nolimits}

\let\@IE\IE\let\IE\undefined
\newcommand{\IE}{\mathop{{}\@IE}\mathopen{}}

\def\be#1{\begin{equation*}#1\end{equation*}}
\def\ben#1{\begin{equation}#1\end{equation}}
\def\bes#1{\begin{equation*}\begin{split}#1\end{split}\end{equation*}}
\def\besn#1{\begin{equation}\begin{split}#1\end{split}\end{equation}}
\def\bg#1{\begin{gather*}#1\end{gather*}}
\def\bgn#1{\begin{gather}#1\end{gather}}
\def\bm#1{\begin{multline*}#1\end{multline*}}

\def\ba#1{\begin{align*}#1\end{align*}}
\def\ban#1{\begin{align}#1\end{align}}

\usepackage{delimset}
\def\given{\mskip 0.5mu plus 0.25mu\vert\mskip 0.5mu plus 0.15mu}
\newcounter{bracketlevel}% 
\def\@bracketfactory#1#2#3#4#5#6{%
\expandafter\def\csname#1\endcsname##1{%
\global\advance\c@bracketlevel 1\relax%
\global\expandafter\let\csname @middummy\alph{bracketlevel}\endcsname\given%
\global\def\given{\mskip#5\csname#4\endcsname\vert\mskip#6}\csname#4l\endcsname#2##1\csname#4r\endcsname#3%
\global\expandafter\let\expandafter\given\csname @middummy\alph{bracketlevel}\endcsname%
\global\advance\c@bracketlevel -1\relax%
}%
}
\def\bracketfactory#1#2#3{%
\@bracketfactory{#1}{#2}{#3}{relax}{0.5mu plus 0.25mu}{0.5mu plus 0.15mu}
\@bracketfactory{b#1}{#2}{#3}{big}{1mu plus 0.25mu minus 0.25mu}{0.6mu plus 0.15mu minus 0.15mu}
\@bracketfactory{bb#1}{#2}{#3}{Big}{2.4mu plus 0.8mu minus 0.8mu}{1.8mu plus 0.6mu minus 0.6mu}
\@bracketfactory{bbb#1}{#2}{#3}{bigg}{3.2mu plus 1mu minus 1mu}{2.4mu plus 0.75mu minus 0.75mu}
\@bracketfactory{bbbb#1}{#2}{#3}{Bigg}{4mu plus 1mu minus 1mu}{3mu plus 0.75mu minus 0.75mu}
}
\bracketfactory{clc}{\lbrace}{\rbrace}
\bracketfactory{clr}{(}{)}
\bracketfactory{cls}{[}{]}
\bracketfactory{abs}{\lvert}{\rvert}
\bracketfactory{norm}{\Vert}{\Vert}
\bracketfactory{floor}{\lfloor}{\rfloor}
\bracketfactory{ceil}{\lceil}{\rceil}
\bracketfactory{angle}{\langle}{\rangle}

\def\clc#1{\{#1\}}
\def\abs#1{\vert#1\vert}

\def\law{\sL}
\newcommand{\eqlaw}{=_d}
\def\eps{\varepsilon}
\def\epsilon{\varepsilon}
\def\I{\mathop{{}\mathrm{I}}}
\def\~#1{\ifmmode {\mathaccent"707E #1} \else {\accent"7E #1} \fi}
\def\~#1{\widetilde{#1}}

\def\ERRG{\mathop{\mathrm{ER}}}
\newcommand{\Vee}{{\mathcal V}}
\makeatother
\def\phi{\varphi}
\def\Leq{\enskip\le\enskip}

\begin{document}

\title{\sc\bf\large\MakeUppercase{Stein's method via induction
}}
\date{\today}
\author{\sc Louis H.~Y.~Chen$^*$, Larry Goldstein$^\ddagger$\\[0.5ex] \sc and Adrian R\"ollin$^*$}
\date{\it National University of Singapore$^{*}$\\[0.1\baselineskip] and University of Southern California$^{\ddagger}$}
\maketitle

\begin{abstract}
\noindent Applying an inductive technique for Stein and zero bias couplings yields Berry-Esseen theorems for
normal approximation for two new examples. The conditions of the main results do not require that the couplings be bounded. Our two applications, one to the \ER\,  random graph with a fixed number of edges, and one to Jack measure on tableaux, 
demonstrate that the method can handle non-bounded variables with non-trivial global dependence, and can produce bounds in the Kolmogorov metric with the optimal rate.
\end{abstract}

\makeatletter
\def\blfootnote{\xdef\@thefnmark{}\@footnotetext}
\makeatother

\section{Introduction}

We present new Berry-Esseen theorems for sums~$Y$ of possibly dependent variables by combining both the Stein and zero bias couplings of Stein's method with the inductive technique of Bolthausen \citeyearpar{Bolthausen84} originally developed for the combinatorial central limit theorem. We apply these results to obtain normal approximations in the Kolmogorov metric for two new examples.

Stein's method (\cite{Stein72}, \cite{Stein86}) typically proceeds by coupling a random variable~$Y$ of interest to a related variable~$Y'$; for an overview see \cite{Chen10} and \cite{Ross11}. Here we develop results that can be applied to the Stein couplings of \cite{ChRo10} and to the zero bias couplings of \cite{GoRe97}, thus encompassing most of the known couplings that have appeared in the literature, including settings not typically framed in terms of couplings, such as local dependence. The innovation here is the widened scope of the couplings that can be handled that permit applications when the difference~$|Y-Y'|$ between~$Y$ and the coupled~$Y'$ is not almost surely bounded by a constant, or where the bound on this difference increases in the problem size. This work is a broad extension and continuation of \cite{Gh09}, applying induction and the zero bias coupling for the combinatorial central limit theorem where the random permutations are involutions, and of \cite{Go13} using the size bias coupling to study degree counts in the Erd\H{o}s-R\'enyi random graph; the inductive method considered here is inspired by \cite{Bolthausen84}, but goes ultimately back to \cite{Bergstrom1944}.

At the center of Stein's method is the characterization that~$Z$ is a standard normal random variable if and only if
\be{
  \IE\clc{Zf(Z)}=\IE\clc{f'(Z)}
}
for all locally absolutely continuous functions~$f$ for which the above expectations exist. Given a standardized variable~$W$ whose distribution is to be compared to~$Z$, and a test function~$h$
on which to evaluate the difference~$\IE h(W)-\IE h(Z)$, one solves the Stein
equation
\bea \label{1}
  f'(w)-wf(w)=h(w)-\IE h(Z)
\ena
for~$f$. The difference~$\IE h(W)-\IE h(Z)$ may then be evaluated by substituting~$W$ for~$w$ and taking expectation on
the left hand side of \eqref{1}, rather than the right. One explanation of why the expectation of the left hand side may simpler to
compute, or bound, than that of the right is that it depends only on the distribution of~$W$, whereas the right
also depends on that of~$Z$. In particular, on the left hand side one may apply couplings of~$W$ to auxiliary random
variables having properties that allow for convenient manipulations.

In Theorem \ref{thm1} we present results for situations in which one can form a Stein coupling as defined by 
\cite{ChRo10}.  Following the treatment there, we say that the triple~$(W,W',G)$ of random variables is a Stein coupling when
\bea \label{2}
\IE\clc{Gf(W')-Gf(W)}=\IE\clc{Wf(W)}
\ena
for all functions~$f$ for which the expectations above exist. It is not difficult to see that the canonical
exchangeable pair coupling of \cite{Stein86}, and the size bias coupling of \cite{Goldstein96} are both special
cases of Stein couplings. Indeed, recall that for~$\lambda \in (0,1]$ we say~$(W,W')$ is a~\emph{$\lambda$-Stein pair} if~$(W,W')$ is exchangeable and
\bea \label{3}
\IE\{W'\given W\} = (1-\lambda) W.
\ena
In this case, it is easily verified that \eqref{2} is satisfied with
\beas
G=\frac{1}{2\lambda}(W'-W).
\enas
Likewise, for a non-negative random variable~$Y$ with finite mean~$\mu$, we say that~$(Y,Y')$ is a size bias coupling of~$Y$
when~$Y'$ has the~$Y$-size bias distribution, that is, when
\beas
\IE\clc{Yf(Y)}=\mu \IE\clc{f(Y')}
\enas
for all functions~$f$ for which these expectations exist. Again, it is easy to verify that for such couplings \eqref{2} is satisfied with
\beas
W=Y-\mu, \quad W'=Y'-\mu \qmq{and} G= \mu.
\enas
In particular, Theorem \ref{thm1} extend results in \citet{Go13} for the size bias coupling.

Theorem \ref{thm2} provides a parallel result for the zero bias coupling~$(W,W^*)$ of \cite{GoRe97}. Recall that for a non-trivial mean zero, variance~$\sigma^2$ random variable~$W$, we say that~$W^*$ has the~$W$-zero biased distribution if
\bea \label{4}
\IE\{Wf(W)\}=\sigma^2 \IE \{f'(W^*)\}
\ena
for all functions~$f$ for which the quantities above exist.

In Stein's method in general, simplification occurs when one can achieve couplings of~$W$ to an appropriate~$W'$ such that the difference is almost surely bounded, or bounded uniformly in the size of the problem. However, in many situations appropriately bounded couplings may be difficult to construct, whereas unbounded couplings seem to appear naturally. Hence Theorems \ref{thm1} and \ref{thm2}, which do not impose restrictive boundedness conditions, may be applied to produce new results in a variety of examples. 

\begin{description}[leftmargin=0cm]
\item[General Framework.] Let~$(\Theta,\cT)$ and~$(\Omega,\cF)$  be two measurable spaces, the \emph{parameter space} and the \emph{sample space}, respectively. All  random variables are understood to be real valued measurable functions from the product space~$(\Theta\times\Omega,\cT\otimes\cF)$. The distribution of a random variable~$X$ is determined by a parameter~$\theta \in \Theta$ through a given transition kernel~$\IP_\theta$ from~$\Theta$ to~$\Omega$. That is, for each~$\theta\in\Theta$,~$\IP_\theta[\cdot]$ is a probability measure on~$(\Omega,\cF)$, and for each~$A\in\cF$, the map~$\IP_\cdot[A]$ is~$\cT$-measurable. Depending on context and emphasis, we may also write~$X$ as~$X(\theta,\omega)$ or~$X_\theta(\omega)$, so that, for instance,~$\IE_\theta X =\int_\Omega X(\theta,\omega)\IP_\theta[d\omega]$.
	
\end{description}

\noindent These measurability conditions are needed to assure the measurability of mappings that appear later, such as of the mean~$\mu_\theta$,  the variance~$\sigma_\theta^2$ of~$Y$, and of~$Y_{\Psi(\theta,\omega)}(\omega)$, which represents the value of~$Y$ at the parameter used in the inductive step. These conditions
will not always be invoked explicitly below;  we illustrate their use by showing in the Appendix, Section \ref{sec5}, that this latter variable in particular is measurable. 

Our goal is to obtain bounds on the Kolmogorov distance between the standardized version~$W$ of a random variable~$Y$ and the normal distribution in 
terms of the parameter~$\theta$. Theorems \ref{thm1} and \ref{thm2} below yield a bound of the form~$C/r_\theta$ for~$r_\theta$ a positive `rate' function of~$\theta$ and~$C$ a constant not depending on~$\theta$. 

As noted, one main step our method requires is to couple~$W$ to a random variable~$W'$, which satisfies either the Stein coupling relation \eqref{2} or the zero bias coupling relation \eqref{4}. In order to apply induction, we identify a subset~$\Goodset \subset \Theta$ in Condition \eqref{6}, consisting of the `nicely behaved' parameters; its complement plays the role of the base case, on which the bound~$C/r_\theta$ may be trivial. For our bound to be informative, it is necessary that the rate function~$r_\theta$ be unbounded on~\Goodset.

For the induction step, we also introduce a sub~$\sigma$-algebra~$\cF_\theta$ that, roughly speaking, captures the information about the changes that were necessary to construct $W'$ from $W$ (or equivalently, $Y'$ from $Y$); the coarser $\cF_\theta$ is, the better the normal approximation will be. A certain tension is created here, as $\cF_\theta$ must be large enough to contain the variables describing the changes from $Y$ to $Y'$, but small enough so that the conditional distribution of~$Y$ on~$\cF_\theta$, is sufficiently close to its original one.

Conditional on  $\cF_\theta$, the variable $Y$ may no longer have its original distribution, but induction is viable when one can identify within $Y$ another variable $V$ that has a distribution similar to the original $Y$; when the parameter space $\Theta$ is ordered, $V$ typically has a smaller parameter. For a successful induction, the parameter of the smaller problem should not stray too far from that of $Y$. There is some leeway here, as it suffices to have control over an event $F_{\theta,1}$, as specified in Condition \eqref{11}. Intuitively, the event $F_{\theta,1}$ should contain the bulk of the support of the variables that generate~$\cF_\theta$, and not their extremes. For instance, for the \ER\ graph problem considered, $\cF_\theta$ contains the label and degree of a chosen vertex on which the coupling is based, and  $F_{\theta,1}$ is an even on which its degree is `not too large'.

Relaxing the condition that the difference~$D=W'-W$ be bounded, we control the magnitude of this difference by its moments. Moreover, we upper bound~$D$ by~$\overline{D}$, and in the case of a Stein coupling, also~$G$ by~$\overline{G}$, where these majorizing variables are required to be~${\cal F}_\theta$ measurable; we are able to handle exceptional or boundary cases as these upper bounds are only required to hold on~$F_{\theta,1}$. We will also require the existence of a random variable~$B$ that bounds the absolute
difference~$|Y-V|$, and which is not `too large.'  See Conditions \eqref{9}, \eqref{11} and 
\eqref{16} for the case of Stein couplings.

There is also some leeway in that the distribution of~$V$, conditionally on~$\cF_\theta$, only needs to be close to that of~$Y$ on an event~$F_{\theta,2} \in \cF_\theta$. 
Precisely, for the Stein coupling case, with similar remarks also applying to zero bias couplings, we impose  in Condition \eqref{13} that
\bea \label{5}
\law _\theta(V\given\cF _\theta)=\law _{\Psi_\theta}(Y) \qmq{on~$F_{\theta,2}$,}
\ena
where $\Psi_\theta$ is the (typically random) parameter capturing the conditional distribution of the embedded variable $V$.
For clarification, by \eqref{5} we mean
\be{ 
	\IP_{\theta}[V\in \cdot\, \given\cF_\theta](\omega) = \IP_{\Psi_\theta(\omega)}[Y \in \cdot\,] \qquad\text{for all~$\omega \in F_{\theta,2}$}.
}
With the help of~$V$, a recursive inequality for a bound on the
distance between~$W$ and the normal can be produced.

Before attempting to apply the methods presented in this article, it is advisable that a user first `test the waters' by constructing a Stein or zero-bias coupling and proving a normal approximation for a smooth metric such as the Wasserstein distance; see \cite{ChRo10}, or \cite{Go17}, respectively. Once this goal has been achieved, the sigma-algebra $\cF_\theta$ will typically arise naturally from the coupling construction, and one may then proceed to identify a suitable variable $V$ whose conditional distribution given $\cF_\theta$ is within the same class of distributions determined by $\Theta$ and close to that of $Y$. For instance, in occupancy problems, a Stein coupling or zero-bias coupling typically involves moving around a small number of balls among a small number of urns, and $V$ will typically again represent an occupancy problem, but on fewer balls and fewer urns.

\subsection{Abstract approximation theorems}\label{subsec1}

We now state the conditions required for our main results.
The inverse rate function~$r_\theta$ is assumed to be a positive function, measurable in~$\theta$, a condition satisfied for all natural examples, including the ones considered
here. The mean~$\mu_\theta=\IE_\theta Y$ and variance~$\sigma_\theta^2=\Var_\theta (Y)$ are measurable by the conditions in our General Framework. To avoid repetition, the distribution of random variables indicated after~$\theta \in \Theta$ has been
fixed is with respect to~$\law _\theta(\cdot)$. The random variable~$Z$ will always denote the standard normal.

The variable~$Y$ denotes the unstandardized random variable of interest. Theorem \ref{thm1} shows that the following set of conditions are sufficient for the Kolmogorov distance between the standardized version~$W$ of~$Y$ and the normal to be bounded by~$C/r_\theta$ for some universal constant~$C$.

\begin{enumerate}
[label=\bfseries(G\arabic*),ref=G\arabic*]
\item\label{6} Let~$r_\theta$ be a positive measurable function, let~$\overline{r}$ be a positive number, and let
\ben{\label{7}
	\Goodset= \{\theta\in\Theta\,:\,r_\theta>\overline{r}\}.
}
Assume that~$\overline{r}$ is chosen such that~$\Var_\theta Y > 0$ for all~$\theta\in\Goodset$.
	
\item\label{8} For all~$\theta \in \Theta$, let~$\mu_\theta = \IE_\theta Y$ and~$\sigma^2_\theta = \Var_
\theta Y$, and define
\be{
	W = \frac{Y-\mu_\theta}{\sigma_\theta}
}
whenever~$\sigma_\theta>0$, and set~$W=0$ otherwise. Let~$W'$ and~$G$ be two random variables such that, for each~$\theta\in \Goodset$,
$(W,W',G)$ is a Stein coupling, in the sense of \eqref{2}, with respect to~$\IP_\theta$.

\item \label{9}  With~$D=W'-W$ assume that
\ben{\label{10}
	\sup_{\theta\in\Goodset} r_\theta \,\IE_\theta\babs{\IE_\theta\clr{1-GD\given W}}< \infty
	\qmq{and}
	\sup_{\theta\in\Goodset} r_\theta \,\IE_\theta\bclc{(1+\abs{W})\abs{G}D^2}<\infty.
}
\item\label{11} For each~$\theta\in \Goodset$, let~$\cF_\theta\subset \cF$ be a sub-$\sigma$-algebra. Let~$\overline G$ and~$\overline D$ be random variables such that, for each~$\theta\in\Goodset$, the mappings~$\overline{G}(\theta,\cdot)$ and~$\overline{D} (\theta,\cdot)$ are~$\cF_\theta$-measurable and such that, on some event~$F_{\theta,1}$ which need not be in~${\cal F}_\theta$, we have~$\abs{G}\leq \overline{G}$,~$\abs{D}\leq \overline{D}$, and
\ben{\label{12}
	\sup_{\theta\in\Goodset} r^2_\theta\,\IE_\theta\bclc{\abs{G}D^2(1-I_{F_{\theta,1}})} < \infty \qmq{and}   \sup_{\theta\in\Goodset} r_\theta \,\IE_\theta\bclc{\overline{G}\,\overline{D}^2 }<\infty.
}

\item\label{13} Let~$\Psi$ be a~$\Theta$-valued random element such that, for each~$\theta\in \Goodset$,~$\Psi(\theta,\cdot)$ is~$\cF_\theta$-measurable.  Let~$V$ be a random variable, and for each~$\theta\in\Goodset$, let~$F_{\theta,2}\in{\cal F}_\theta$ be such that
\ben{\label{14}
	\law_\theta(V\given\cF_\theta) = \law_{\Psi}(Y) \qquad\text{on~$F_{\theta,2}$,}
}
and
\ben{\label{15}
	\sup_{\theta\in \Goodset} r^2_\theta\,\IE_\theta\bclc{\abs{G}D^2(1-I_{F_{\theta,2}})} < \infty.
}
\item\label{16} Let~$\overline B$ be a random variable such that, for each~$\theta\in\Goodset$,~$\overline B(\theta,\cdot)$ is~$\cF_
\theta$-measurable,
\ben{\label{17}
	\sigma_\theta^{-1}\abs{Y-V} \leq \overline B \quad
	\text{on~$F_{\theta,1}$},\quad\text{and} \quad
	\sup_{\theta\in\Goodset} r^2_\theta \IE_\theta\bclc{\overline G\,\overline D^2\overline B  I_{F_{\theta,2}}} < \infty.
}
\item\label{18} Assume
\bgn{\label{19}
	\sup_{\theta\in\Goodset}\,\,\esup_{\omega\in F_{\theta,2}\cap \{\Psi\in\Goodset\}}\frac{\sigma^2_\theta}{\sigma^2_{\Psi(\theta,\omega)}} < \infty,\\
	\label{20}
	\sup_{\theta\in\Goodset}\,\,\esup_{\omega\in F_{\theta,2}}\frac{r_\theta}{r_{\Psi(\theta,\omega)}} < \infty,
	\qquad  
	\sup_{\theta\in\Goodset}\,\,\esup_{\omega\in F_{\theta,2}\cap \{\Psi\in\Goodset\}}\frac{r_{\Psi(\theta,\omega)}}{r_\theta} < \infty,
}
where the essential suprema are taken with respect to $\IP_\theta$.
\end{enumerate}

\begin{theorem} \label{thm1}
If Conditions \eqref{6}--\,\eqref{18} are satisfied, then there exists a constant $C$, independent of $\theta$, such that
\ben{\label{21}
	\sup_{z\in \IR}\babs{\IP_\theta[W\leq z] -
			\IP[Z\leq z]} \leq \frac{C}{r_\theta}\qquad \text{for all~$\theta\in\Theta$.}
}
\end{theorem}

Theorem \ref{thm1} extends Theorem 1.1 in \citet{Go13}, which produces a Kolmogorov bound
equivalent up to constants to the bound in \citet{ChRo10} for the Wasserstein distance to the normal for bounded
size bias couplings. In addition, the bound produced by \citet{BaGo13} by an application of Theorem 1.1 of
\citet{Go13} to counts in a multinomial occupancy model was shown there to be of optimal order by the lower bound
(1.6) of \citet{Englund81}, see also (1.7) of \citet{BaGo13}; the bound of Theorem 1.2 of \citet{Go13}, using also
Theorem 1.1 of that same work, for degree counts in the Erd\H{o}s-R\'enyi random graph can also be shown to be
optimal up to constant factors in the same manner.

When higher moments exist a number of the conditions of the theorem may be verified using simpler expressions, obtained via
standard inequalities. For instance, using~$f(w)=w$ and that~$\Var_\theta(W)=1$ in
\eqref{2} shows that~$\IE_\theta (G D)=1$, hence applying the Cauchy-Schwarz inequality to the first expression in \eqref{10} in Condition \eqref{9} above, followed by a consequence of the conditional variance formula, we obtain 
\bea \label{22}
\IE_\theta \left|\IE_\theta\clr{1-G D\given W} \right|  \le \sqrt{\Var_
	\theta \left(\IE_\theta\clr{G D\given W}\right)} \le \sqrt{\Var_
	\theta \left(\IE_\theta\clr{G D\given {\mathcal H}}\right)},
\ena
where~${\mathcal H}$ is any~$\sigma$-algebra with respect to which~$W$ is measurable. 

We now state a parallel result for zero bias couplings.

\begin{enumerate}[label=\bfseries(Z\arabic*),ref=Z\arabic*]
\item\label{23} Let~$r_\theta$ be a positive measurable function, let~$\overline{r}$ a positive number, and let
\be{
	\Goodset = \{\theta\in\Theta\,:\,r_\theta>\overline{r}\}.
}
Assume that~$\overline{r}$ is chosen such that~$\Var_\theta Y > 0$ for all~$\theta\in\Goodset$.

\item \label{24} Let~$\mu_\theta = \IE_\theta Y$ and~$\sigma^2_\theta = \Var_
\theta Y$, and define
\be{
	W = \frac{Y-\mu_\theta}{\sigma_\theta}
}
whenever~$\sigma_\theta>0$ and~$W=0$ otherwise. Let~$W^*$ be defined on~$\Omega$, such that for each~$\theta\in\Goodset$ the variable~$W^*$ has the~$W$-zero bias distribution as in \eqref{4} with respect to~$\IP_\theta$.

\item\label{25} For each $\theta \in \Goodset$ let $\cF _\theta$ be a sub-sigma algebra of $\cF$, let $D=W^*-W$, and let ${\overline D}$ be a random variable such that ${\overline D}(\theta,\cdot)$ is $\cF _\theta$-measurable, and let $F_{\theta,1}$ be an event, which need not be in ${\cal F}_\theta$, on which $|D| \le {\overline D}$ and such that
\ben{\label{26}
	\sup_{\theta\in\Goodset} r^2_\theta\,\IE_\theta\bclc{|D|(1-I_{F_{\theta,1}})} < \infty \qmq{and} \sup_{\theta \in \Goodset} r_\theta E_\theta \left\{|D W|+{\overline D}\right\} <\infty.
}

\item\label{27} Let $V$ be a random variable, and let $\Psi$ be a $\Theta$-valued random element such that, for each $\theta\in\Goodset$, $\Psi(\theta,\cdot)$ is $\cF_\theta$-measurable. For each $\theta\in\Goodset$, let~$F_{\theta,2}$ be an event in ${\cal F}_\theta$ such that
\ben{\label{28}
	\law_\theta(V\given\cF_\theta) = \law_{\Psi}(Y) \qquad\text{on~$F_{\theta,2}$,}
}
and
\ben{\label{29}
	\sup_{\theta\in\Goodset} r^2_\theta\,\IE_\theta\bclc{|D|(1-I_{F_{\theta,2}})} < \infty.
}
\item\label{30} Let $\overline B$ be a random variable such that, for each $\theta\in\Goodset$, $\overline B(\theta,\cdot)$ is $\cF_
\theta$-measurable, and
\ben{\label{31}
	\sigma_\theta^{-1}\abs{Y-V} \leq \overline B  \qmq{on $F_{\theta,1}$,\quad and}   \sup_{\theta\in\Goodset} r^2_\theta \IE_\theta\bclc{\overline D \left( \overline B  + \overline D \right) I_{F_{\theta,2}}} < \infty.
}
\end{enumerate}

\begin{theorem} \label{thm2}
If Conditions \eqref{23}--\,\eqref{30} and \eqref{18} are satisfied, then there exists a constant $C$, independent of $\theta$, such that
\be{ 
	\sup_{z\in \IR}\babs{\IP_\theta[W\leq z] -
			\IP[Z\leq z]} \leq \frac{C}{r_\theta},\qquad \text{for all~$\theta\in\Theta$.}
}
\end{theorem}

Many of the conditions of Theorem \ref{thm2}, as for Theorem \ref{thm1}, can be shown to be satisfied using inequalities on moments. The proofs of Theorems \ref{thm1} and~\ref{thm2} appear in Section \ref{sec4}.

\subsection{Applications}

We apply Theorems \ref{thm1} and \ref{thm2} to obtain new results in two examples; the proofs are deferred to  Sections \ref{sec1} and \ref{sec3}.

The first examples invokes Theorem \ref{thm1} for Stein couplings for the normal approximation of the number $Y$ of isolated vertices in the \ER\, graph~$\cG \sim \ERRG(n,m)$ on $n$ vertices, having exactly $m$ edges, distributed uniformly at random. This model is related to the one where edges between each pair of vertices are chosen independently with some fixed probability, but in the model we consider the indicators that vertices are isolated exhibit a non-trivial global dependence since the total number of edges is fixed. In fact, while in the model with independent edges these indicators are positively correlated, the effect of the global dependence in $\ERRG(n,m)$ is stronger, resulting in a negative correlation; see proof of Lemma~\ref{lem5}.  

Related work was done by \cite{Kord87} on the number of isolated vertices in the {\ER}  graph model, although his general framework is not applicable here.The boundedness of the second derivative of the solution to the Stein equation on page 132 is shown only for the points where the second derivative exists, whereas, in order to perform the Taylor expansion on page 135, it is needed to hold everywhere; we were thus not able to reproduce his final results. In addition, the fixed number of edges model does not appear to satisfy the condition on page 134 of his work. 
We also mention the work by \cite{Go13}, who considered vertex degrees in general, though it only addressed the independent edge model.

Theorem \ref{thm3} provides the following bound on the Kolmogorov distance between the standardized variable $Y$ and the normal.

\begin{theorem}\label{thm3} Let $Y$ count the number of isolated vertices in the \ER\, graph
$\cG \sim \ERRG(n,m)$ on $n$ vertices, having exactly $m$ edges, distributed uniformly at random. Then, with $\mu_{n,m}$ and $\sigma_{n,m}^2$ the mean and variance of $Y$, letting 
$W=(Y-\mu_{n,m})/\sigma_{n,m}$ when $\sigma_{n,m}>0$ and zero otherwise, with 
\bea \label{35}
\Theta=\left\{(n,m)\,:\,\text{$n\geq3$, $0< m< {n\choose2}$}\right\},
\ena
there exists a universal constant $C>0$ such that, for all $(n,m) \in \Theta$,
\be{
	\sup_{z\in\IR}\babs{\IP_{n,m}[W\leq z]-\Phi(z)}\leq\frac{C}{r_{n,m}}
}
where
\ben{\label{36}
	r_{n,m} = 
	\frac{\sigma_{n,m}^3}{\mu_{n,m}(1+\frac{m^2}{n^2})}.
}
\end{theorem}

\begin{remark} In order to better understand the bounds obtained in Theorem~\ref{thm3}, we now discuss in more detail the different regimes at which $m$ and $n$ can tend to infinity. To this end, denote by $a(n)\sim b(n)$ that $\lim a(n)/b(n) =1$, and by $a(n)\asymp b(n)$ that $\liminf a(n)/b(n)>0$ and $\limsup a(n)/b(n)<\infty$. By Lemma \ref{lem7}, if $n$ and $m$ tend to infinity so that $\max\{m/n^2,m^2/n^3\} \rightarrow 0$, then 
\begin{align*}
\mu_{n,m} \sim n e^{-2m/n} \qmq{and} \sigma_{n,m}^2 \sim  n\phi(2m/n) \qmq{for} \phi(x)=e^{-x}(1-e^{-x}(1+x)).
\end{align*}
Hence, we have
\begin{align*}
\frac{ \sigma_{n,m}^2}{\mu_{n,m}}\sim 1-e^{-2m/n}\left(1+\frac{2m}{n}\right) ,
\end{align*}
so that
\begin{align*}
  r_{n,m} \sim \sigma_{n,m}  \frac{1-e^{-2m/n}(1+\frac{2m}{n})}{1+(m/n)^2} .
\end{align*}
For $n \asymp m$, the central domain, it follows that  $r_{n,m} \asymp \sigma_{n,m}$, and moreover, in the special case where $m \sim cn$, 
		\begin{align*}
		\mu_{n,m} \sim ne^{-2c} \qmq{and} \sigma_{n,m}^2 \sim ne^{-2c}\left(1-  e^{-2c}(1+2c)\right). 
		\end{align*}
Regarding lower bounds, \cite[Section~6]{Englund81} shows that for the standardized number of occupied cells in a uniform occupancy model with $n$ balls and $m$ boxes, 
\begin{align*}
	\sup_{z \in \mathbb{R}}|\IP[W_{n,m} \le z]-\IP[Z\le z]| \ge 0.087/\max(3,\sigma_{n,m}).
\end{align*}
Englund’s argument holds without changes for any random variable with finite variance supported on the integers, and so also for the number of isolated vertices in our model. Hence, since in the central domain $r_{n,m} \asymp \sigma_{n,m}$, the rate function is of optimal order.

If $m\to\infty$ and $m/n\to0$, the left domain, say, then 
\be{
  r_{n,m}\asymp \frac{\sigma_{n,m}m^2}{n^2} \asymp \frac{m^3}{n^{5/2}}
}
since $1- e^{-x}(1+x) \sim x^2/2$ as $x\to 0$ for the first relation, and $\sigma_{n,m}^2\asymp m^2/n$ for the second. In this case,  Englund's lower bound is not achieved since $r_{n,m}=\lito(\sigma_{n,m})$. Nonetheless, the bound is informative as long as $r_{n,m}\to\infty$, which is the case as long as $m / n^{5/6}\to\infty$, such as when $m=cn^{\alpha}$ for $c>0$ and $5/6<\alpha<1$.

If $m/n\to\infty$, the right domain, using $\sigma^2_{n,m}\asymp n e^{-2m/n}$ for the second relation we have
\be{
  r_{n,m}\sim \frac{\sigma_{n,m}n^2}{m^2} \asymp \frac{e^{-m/n}n^{5/2}}{m^2}, 
}
so Englund's lower bound is not attained. However, $r_{n,m}$  goes to infinity when $m \le \alpha\, n\log n$ for $0<\alpha<1/2$.
\end{remark}

In the second example, we use the zero bias coupling constructed in \cite[Theorem 3.1]{FuGo11} in Theorem \ref{thm2} to give a bound on the normal approximation of the content $Y$ of a Young tableux under Jack$_{\alpha}$ measure over a range of large $\alpha$. In more detail, we recall that a partition of a positive integer~$n$ can be represented as a vector $\Lambda=(\lambda_1,\ldots,\lambda_p)$ of non-increasing, positive integers summing to $n$, where $p$ is the number of parts of the partition. For instance, $\Lambda=(4,2,1)$ corresponds to a partition of $n=7$ with $p=3$. In turn, the partition $\Lambda$ can be represented by a tableaux with $p$ rows of equal sized boxes, whose $j^{\mathrm{th}}$ row is of length $\lambda_j$, such as in \eqref{38}. 

The Jack$_{\alpha}$ measure on tableaux, defined for $\alpha>0$, recovers the Plancherel measure when specializing to the case $\alpha=1$. Under Jack$_{\alpha}$, see \cite{F1} for instance, the 
probability of a partition $\Lambda$ of $n$ is given by
\bea \label{37}
\mbox{Jack}_\alpha(\Lambda)=
\frac{\alpha^n n!}{\prod_{x \in \Lambda} (\alpha a(x) + l(x) +1)
	(\alpha a(x) + l(x) + \alpha)},
\ena
where the product is over all boxes $x$
in the partition, $a(x)$ denotes the number of boxes in the same
row of $x$ and to the right of $x$ (the ``arm'' of $x$), and $l(x)$
denotes the number of boxes in the same column of $x$ and below $x$
(the ``leg'' of $x$). For each tableaux representing a partition of $n$ we may define the $\alpha$-content of any individual box by 
\[ c_{\alpha}(x) = \alpha(\mbox{column number of $x - 1$}) - (\mbox{row number of $x - 1$}),\]
as depicted in the following tableaux for the partition $(4,2,1)$ of 7:
\ben{\label{38}
\ytableausetup{mathmode, boxsize=2em}
\begin{ytableau}0 & \alpha & 2\alpha & 3\alpha  \\  - 1 & \scriptstyle \alpha-1 \\-2 \end{ytableau}
}
Here we study the distribution of the standardized sum of the $\alpha$-contents over all boxes in the tableaux, that is, 
\bea \label{39}
W = \frac{Y}{\sqrt{\alpha {n \choose 2}}}, \qmq{where} Y = \sum_{x \in
	\Lambda_n} c_{\alpha}(x)
\ena 
and where the partition $\Lambda_n$ of $n$ is sampled from the Jack$_{\alpha}$ measure in \eqref{37}.

\cite{F1} proved an $\bigo(n^{-1/4})$ bound for the error in the Kolmogorov metric for the normal
approximation of $W$, improved by \cite{F4} using martingales to $\bigo(n^{-1/2+\epsilon})$ for any $\epsilon>0$, and by \cite{F3} to $\bigo(n^{-1/2})$ using Bolthausen's inductive
approach and Stein's method, but without an explicit constant. \cite{HO} prove a central limit theorem, with no error bound, for $W_{n,\alpha}$ using quantum probability.

\cite{FuGo11} prove the bound 
\bea \label{40}
d_1(W, Z) \le \sqrt{\frac{2}{n}} \left(2 + \sqrt{2+\frac{\max(\alpha,1/\alpha)}{n-1}}  \right) \qm{for all $n \ge 2, \alpha>0$,}
\ena
in the Wasserstein metric $d_1$,
where $Z$ is a standard normal variable. In addition to providing explicit constants, this bound also highlights the role of $\alpha$. A natural question it brings is whether a bound in the Kolmogorov metric can be shown that has this same dependence on $\alpha$. A few weeks before the current work was posted,
\cite[Theorem~1.1]{ChTh17} proved the bound
\beas
\sup_{x \in \mathbb{R}} |\IP_{n,\alpha}[W \le x] - \IP[Z \le x]| \le 9\left( \frac{1}{\sqrt{n}} \vee \frac{(\sqrt{\alpha} \vee 1/\sqrt{\alpha}) \log n}{n}\right) ,
\enas
which achieves this goal with an explicit constant to within a logarithmic factor. 

Here, given any $\epsilon \in (0,1)$, we show that, in the `large $\alpha$' region $\alpha \ge n^{1+\epsilon}$, this log factor may be removed, resulting in the bound having the same $\alpha$ dependence as \eqref{40}. That is, as $\alpha \ge n$ over the region we consider, the ratio between the right hand sides of \eqref{40} and \eqref{41} is bounded away from zero and infinity. This same result, with an explicit constant, was also achieved by \cite[Proposition~4.1]{ChTh17} by applying a different approach. We do not consider $\epsilon > 1$, as Theorem \ref{thm5} below shows that this case is degenerate. 

\begin{theorem}\label{thm4}
For $W$ as given in \eqref{39} with $\Lambda_n$ sampled according to Jack$_{\alpha}$ measure for some $n \ge 2$, for every $\epsilon \in (0,1)$ there exists a constant $C$ depending only on $\epsilon$ such that

\bea \label{41}
\sup_{z\in \IR}\babs{\IP_{n,\alpha}[W\leq z] -
	\IP[Z\leq z]} \leq \frac{C \sqrt{\alpha}}{n} \qquad \text{for all $n \ge 2$ and $\alpha \ge n^{1+\epsilon}$.}
\ena
\end{theorem}

We remark that by applying the reasoning at the end of the proof of Theorem~4.1 of \cite{FuGo11} the result holds also for $\alpha \le n^{-1-\epsilon}$ when replacing the $\alpha$ on the right hand side by $1/\alpha$. In the computations that follow, $C$ without subscript will denote a universal constant whose value may change from line to line, and for $n$ a non-negative integer, $[n]$ will denote the set $\{1,\ldots,n\}$.

\section{Isolated vertices in the Erd\H{o}s-R\'enyi random graph} \label{sec1}

In this section we prove Theorem~\ref{thm3}. We begin by 
reviewing Construction~2A of \cite{ChRo10} for Stein couplings. Let ${\bf X}=(X_1,\ldots,X_n)$ be a collection
of mean zero random variables, and let $I$ be a random index uniformly distributed over $[n]$, independent of ${\bf X}$. Let $W=
\sum_{i \in [n]} X_i$ and suppose that for each $i=1,\ldots,n$ there exists $W_i'$ such that
\bea \label{42}
\IE\clc{X_i\given W_i'}=0.
\ena
Then, with~$G=-nX_I$, the triple~$(W,W_I',G)$ is a Stein coupling. To verify the claim, first note that
\bes{
	\IE\{Gf(W_I')\}& =-\IE\{n X_I f(W_I')\}=-\IE \sum_{i \in [n]} X_i f(W_i') \\
	& = -\IE \sum_{i \in [n]} \IE\{X_i|W_i'\}f(W_i') =0.
}
On the other hand,
\beas
-\IE\{Gf(W)\}=\IE\{nX_If(W)\}=\IE\sum_{i \in [n]} X_i f(W)=\IE\{Wf(W)\};
\enas
so \eqref{2} holds.

\subsection{Isolated vertices in $\ERRG(n,m)$}\label{sec2}
Consider the \cite{ER60} random graph $\cG\sim \ERRG(n,m)$ on $n$ vertices, having exactly $m$ edges, distributed uniformly at random. Let $d_v$ be the degree of vertex $v \in [n]$, and consider the number of isolated vertices
\be{
	Y= \sum_{v=1}^n \I[d_v = 0].
}
With $N={n\choose 2}$, the mean and variance of $Y$ are given by, respectively,

\be{ 
	\mu_{n,m}
	= n\frac{{N-(n-1)\choose m}}{{N\choose m}} \qmq{and}
	\sigma^2_{n,m} = \mu_{n,m} + n(n-1)\frac{{N-(2n-3)\choose m}}{{N\choose m}} - \mu_{n,m}^2.
}

We remark that though there may be a choice of couplings for a given situation, the coupling we have chosen will work for the more general problem where $Y$ is a sum
\begin{align*}
Y=\sum_{v=1}^n h_v(d_v)
\end{align*}
of functions $h_v$ of the degree $d_v$ of vertex $v$. For instance, the size bias coupling will work, as in \cite{Go13}, for counting the number of vertices having specified degrees, but not in this greater generality.

\begin{proof}[Proof of Theorem~\ref{thm3}] 
The proof consists of the setting up the framework, and then checking that Conditions~\eqref{6}--\eqref{18} hold, with Condition~\eqref{8} requiring the construction of a Stein coupling. First, let $\cE_n$ be the enumeration of all $N$ unordered pairs $\{v,w\}\subset [n]$ with $v\neq w$, given by
\ben{\label{43}
\cE_n=\bclr{\{1,2\},\dots,\{1,n\},\{2,3\},\dots,\{2,n\},\dots,\{n-1,n\}}.
}
Let~$\pi$ be a uniformly chosen random permutation of~$[N]$. We will describe the construction of a graph~$\cG(m,\pi)$,  determined by~$m$ and~$\pi$, that has distribution~$\ERRG(n,m)$. As~$n$ is determined by~$N$, and hence by~$\pi$,~$n$ may be omitted in the notation for the graph; the same principle will be applied without comment for like quantities that appear later. 

We construct~$\cG(m,\pi)$ as follows. For each~$\{v,w\} \subset [n]$ with~$v< w$, connect vertices~$v$ and~$w$ with an edge if and only if 
\ben{\label{44}
\pi^{-1}(i)\leq m,
}
where $i$ is the index in the enumeration \eqref{43} corresponding to the pair $\{v,w\}$. Clearly this construction results in a graph with~$m$ edges, precisely, those with labels~$\{\pi(1),\ldots,\pi(m)\}$. Since~$\pi$ is uniform it is immediate that~$\cG(m,\pi)\sim\ERRG(n,m)$. Let~$d_v(m,\pi)$ be the degree of vertex~$v\in[n]$ in~$\cG(m,\pi)$, let 
\ben{\label{45}
I_v(m,\pi)=\I[d_v(m,\pi)=0]
\qmq{and}
Y(m,\pi) = \sum_{v=1}^n I_v(m,\pi).
}
We now verify the conditions of Theorem \ref{thm1} with~$\Theta$ and~$r_{n,m}$ as given in \eqref{35} and \eqref{36}, respectively. 

\begin{description}[leftmargin=0em,parsep=\parskip,listparindent=\parindent]
\item[Condition~(\ref{6}).] Let $n_0$, $m_0$, $c_0$ and $C_0$ be as in Lemma~\ref{lem7}.
Now obtain $\overline{r}$ in the definition \eqref{7} of $\Goodset$ through Lemma~\ref{lem8} and the choices
\ben{ \label{46}
	\overline{n}=(2n_0)\vee 344, \qquad	\overline{m}=(2m_0)\vee (8C_0)  \vee 28, 
	\qmq{and}
	\overline{c}= 1 \wedge \frac{c_0}{2}\wedge\frac{1}{3C_0^{1/2}}.
}
Since our definition of $r_{n,m}$ in \eqref{36} implies that $r_{n,m} = 0$ whenever $\sigma^2_{n,m}=0$, the condition that $\sigma^2_{n,m}>0$ on \Goodset is satisfied. 
Note that by Lemma~\ref{lem8}
\ben{\label{47}
n\geq \overline{n} \qmq{and}  \overline{m}\leq m\leq \overline{c} n^{3/2}\quad\text{ whenever $(n,m)\in\Goodset$}.
}

\item[Condition~(\ref{8}).] For $(n,m) \in \Goodset$, let 
\ben{\label{48}
	W = \frac{Y(m,\pi)-\mu_{n,m}}{\sigma_{n,m}},
}
and set $W=0$ otherwise. Assume~$(n,m)\in \Goodset$.
Let~$\Sigma=(\sigma_1,\dots,\sigma_n)$ be a collection of uniform random permutations of~$[N]$, with~$\pi,\sigma_1,\ldots,\sigma_n$ mutually independent.
The purpose of the following algorithm is to take the graph~${\cal G}(m,\pi)$ as input and to construct, for each vertex~$v\in[n]$, a graph~$\cG^v(m,\pi,\sigma_v)$ on the~$n-1$ vertices~$[n]\setminus \{v\}$, having distribution~$\ERRG(n-1,m)$, independent of~$d_v(m,\pi)$, and which can be closely coupled to~$\cG(m,\pi)$.

We first describe the algorithm in words: Initialise counters $k$ and $i$ that respectively record the number of edges successfully relocated, and the index of a candidate edge for possible addition to the new graph; for each given vertex~$v \in [n]$, begin with~$\cG(m,\pi)$ and relocate the~$d_v(m,\pi)$ edges incident to~$v$ uniformly by, incrementing $i$ when needed, adding~$\cE_n(\sigma_v(i))$ as a new edge when it connects two vertices, neither of which are incident to~$v$ (Step 6), and which are not already connected (Step 7). The counter~$k$ records the number of edges successfully relocated, and the set~$L^v(m,\pi,\sigma_v)$ holds their locations (that is, indices) in~${\cal E}_n$. At termination, the set $L^v(m,\pi,\sigma_v)$ will have size~$d_v(m,\pi)$.

\medskip
\makeatletter\newcommand\mynobreakpar{\par\nobreak\@afterheading}\makeatother
\noindent\textbf{Algorithm 1.} Fix~$v\in[n]$.\mynobreakpar
\begin{enumerate}
	\item Let~$L^v(m,\pi,\sigma_v)\leftarrow\emptyset$
	\item Let~$\cG'$ be equal to~$\cG(m,\pi)$, but with vertex $v$ and all~$d_v(m,\pi)$ edges incident to~$v$ removed. 
	\item Let~$k\leftarrow 0$ and~$i\leftarrow 0$.
	\item\label{49} If~$k =d_v(m,\pi)$, then denote the resulting graph by~${\cal G}^v(m,\pi,\sigma_v)$, and stop.
	\item\label{50} Let~$i\leftarrow i+1$.
	\item If~$v\in \cE_n(\sigma_v(i))$, then return to Step \ref{50}.
	\item If~$\pi^{-1}(\sigma_v(i))\leq m$, that is, if~$\cE_n(\sigma_v(i))$ is an edge in $\cG(m,\pi)$, then 
	return to Step \ref{50}.
	\item In~$\cG'$ connect the vertices in~$\cE_n(\sigma_v(i))$ by an edge, and let~$L^v(m,\pi,\sigma_v)\leftarrow L^v(m,\pi,\sigma_v) \cup \{\sigma_v(i)\}$.
	\item Let~$k\leftarrow k+1$.
	\item Return to Step \ref{49}.
\end{enumerate}
It is not difficult to see that the algorithm will succeed in redistributing the edges incident on~$v$ if and only if~$m\leq {n-1\choose 2}$, which is guaranteed by our choice of~\Goodset. Note that, given~$m$,~$\pi$ and~$\sigma_v$, the construction of~$\cG^v(m,\pi,\sigma_v)$ from~$\cG(m,\pi)$ is \emph{deterministic} and hence, for given~$m$,~$\pi$ and~$\sigma_v$, will always result in the same graph~$\cG^v(m,\pi,\sigma_v)$. Note also that, although~$\cG^v(m,\pi,\sigma_v)$ has only~$n-1$ vertices, we keep the labeling from the original graph~$\cG(m,\pi)$. Since the order at which potential locations where the~$d_v(m,\pi)$ edges are added are sampled uniformly at random without replacement (via~$\sigma_v$), it is clear that~$\cG^v(m,\pi,\sigma_v)\sim\ERRG(n-1,m)$, up to vertex labeling.

Now, let~$W=W(m,\pi)$ as in \eqref{48}. With~$\Vee$ a uniformly chosen vertex from~$[n]$, independent of~$\pi, \sigma_1,\ldots,\sigma_n$, and recalling the notation in \eqref{45}, let
\ben{\label{51}
	G = -\frac{n}{\sigma_{n,m}}(I_\Vee(m,\pi)-\mu_{n,m}/n).
}
For $w \neq v$, let~$d^{v}_w(m,\pi,\sigma_v)$ be the degree of vertex~$w$ in the graph~$\cG^v(m,\pi,\sigma_v)$, let
\be{
	I^v_w(m,\pi,\sigma_v)=\I[d^{v}_w(m,\pi,\sigma_v)=0],
	\qquad
	Y^{v}(m,\pi,\sigma_v) = \sum_{w\in [n]\setminus v} I^v_w(m,\pi,\sigma_v),
}
and
\ben{\label{52}
	W' = \frac{Y^{\Vee}(m,\pi,\sigma_\Vee)-\mu_{n,m}}{\sigma_{n,m}}, \qmq{and hence,} D=\frac{Y^{\Vee}(m,\pi,\sigma_\Vee)-Y(m,\pi)}{\sigma_{n,m}}.
}
Since the distribution of~$\cG^v(m,\pi,\sigma_v)$ is the same regardless of the value of~$d_v(m,\pi)$, we conclude that 
$I_v(m,\pi)-\mu_{n,m}/n$ and~$Y^v(m,\pi,\sigma_v)$ are independent, so \eqref{42} holds, implying~$(W,W',G)$ is a Stein coupling.

\item[Condition~(\ref{9}).]  In what follows, consider a fixed $(n,m) \in \Goodset$, and drop the subscript~$\theta$ in the expectations that follow.
As $W$ is a function of $(\pi,\Sigma)$, using \eqref{22} we have
\besn{\label{53}
	\IE\babs{\IE(1-GD|W)} 
	& \leq \bclr{\Var\IE(GD|\pi,\Sigma)}^{1/2}.
}
Now, from \eqref{51} and \eqref{52}, we have
\bes{
	\Var\IE(GD|\pi,\Sigma) 
	& = \frac{1}{\sigma^4_{n,m}}\Var\sum_{v\in[n]}(I_v(m,\pi)-\mu_{n,m}/n)(Y(m,\pi)-Y^v(m,\pi,\sigma_v)).
}
Splitting the sum into two and using $\Var(X+Y)\leq 2\Var X+2\Var Y$, we have
\bes{
	\Var\IE(GD|\pi,\Sigma) 
	& \leq \frac{2}{\sigma^4_{n,m}}\Var(f_m(\pi,\Sigma)) + \frac{2\mu_{n,m}^2}{n^2\sigma^4_{n,m}}\Var(g_m(\pi,\Sigma)),
}
where
\be{ \label{54}
  f_{m}(\pi,\Sigma) = \sum_{v\in[n]}I_v(m,\pi)B_v(m,\pi,\sigma_v)
 \qmq{and} 
 g_{m}(\pi,\Sigma) = \sum_{v\in[n]}B_v(m,\pi,\sigma_v)
}
with $B_v(m,\pi,\sigma_v) = Y(m,\pi)-Y^v(m,\pi,\sigma_v)$.
Note that~$f_m(\pi,\Sigma)$ and $g_m(\pi,\Sigma)$ are deterministic functions of~$m$,~$\pi$ and~$\Sigma$. Applying Lemma~\ref{lem1} and using the notation as there, we obtain
\ben{\label{55}
  \Var\IE(GD|\pi,\Sigma) 
  \leq 
  \frac{1}{\sigma^4_{n,m}}\bclr{
    R_{g,1}+R_{f,1}+
  	R_{g,2}+R_{f,2}}
}
where
\ba{
	R_{g,1} & =\frac{\mu_{n,m}^2}{n^2} \sum_{i=1}^n\IE\bclr{g_m(\pi,\Sigma)-g_m(\pi,\Sigma_i')}^2,\\
  R_{f,1} & = \sum_{i=1}^n\IE\bclr{f_m(\pi,\Sigma)-f_m(\pi,\Sigma_i')}^2, \\
    R_{g,2} & =\frac{\mu_{n,m}^2}{n^2}\sum_{j=1}^{N-1} \IE\bclr{g_m(\pi,\Sigma)-g_m(\pi\tau_j,\Sigma)}^2,\\
  R_{f,2} & = \sum_{j=1}^{N-1} \IE\bclr{f_m(\pi,\Sigma)-f_m(\pi\tau_j,\Sigma)}^2.
}

\paragraph{Bounding $\boldsymbol{R_{g,1}}$.}
Note that
\ben{ \label{56}
	g_m(\pi,\Sigma)-g_m(\pi,\Sigma_i')
	=  B_i(m,\pi,\sigma_i)-B_i(m,\pi,\sigma_i'),
}
since all differences arising from the first sum in \eqref{54} cancel except the one with index~$v=i$. Applying the simple bound
\be{ 
	\abs{B_v(m,\pi,\sigma_v)}= \abs{Y(m,\pi)-Y^v(m,\pi,\sigma_v)} \leq 1+2d_i(m,\pi)
}
we obtain
\ben{\label{57}
	\abs{ B_i(m,\pi,\sigma_i)-B_i(m,\pi,\sigma_i')}
	\leq 2+4d_i(m,\pi).
}

Let $\Hyp(N,m,n)$ count the number of white balls among~$m$ draws from an urn with~$N$ balls,~$n$ of which are white and~$N-n$ black.
Note that the marginal distribution of the degree of any vertex in~${\cal G}(m,\pi)$ is~$\Hyp\bclr{N,m,n-1}$, and hence has mean $2m/n$, since the graph's~$m$ edges are uniformly sampled among all~$N$ possibilities, and exactly~$n-1$ of them are associated with a specific vertex. Hence, applying Lemma~\ref{lem2},  \eqref{56} and \eqref{57}, we obtain
\be{
	\IE\bclr{g_m(\pi,\Sigma)-g_m(\pi,\Sigma_i')}^2
	\leq C\left(1+ \frac{m^2}{n^2}\right),
}
where we recall $C$ denotes a universal constant, whose value may change from line to line. 
Thus, as $\mu_{n,m} \le n$,
\ben{\label{58}
	R_{g,1} \leq \frac{C\mu_{n,m}^2}{n}\left(1+ \frac{m^2}{n^2}\right) \leq C\mu_{n,m}\left(1+ \frac{m^2}{n^2}\right).
}

\paragraph{Bounding $\boldsymbol{R_{f,1}}$.}
As for $g_m$, we likewise have
\be{
	f_m(\pi,\Sigma)-f_m(\pi,\Sigma_i')
	= I_i(m,\pi)\bclr{ B_i(m,\pi,\sigma_i)-B_i(m,\pi,\sigma_i')}.
}
Noting that, if $I_i(m,\pi)=1$, we have $d_i(m,\pi)=0$ and hence $B_i(m,\pi,\sigma_i)=B_i(m,\pi,\sigma_i')=1$, it is immediate that
\ben{\label{59}
  R_{f,1}= 0.
}

\paragraph{Bounding $\boldsymbol{R_{g,2}}$.} In order to bound~$R_{g,2}$, with~$\tau_{ij}$ the transposition of~$i$ and~$j$, note first that
\ben{\label{60}
	g_m(\pi\tau_{ij},\Sigma)=g_m(\pi,\Sigma),
	\qquad \text{if~$i,j\leq m$ or~$i,j>m$,}
}
since~$g_m$ is a function of the graph~$\cG(m,\pi)$ and~$\Sigma$, and by \eqref{44}, the graph~$\cG(m,\pi)$ obtained from~$\pi$ does not change when swapping edge with edge or non-edge with non-edge.
Hence, averaging over~$\tau_j$, a transposition of $j$ and a uniformly chosen index in $\{j,\ldots,N\}$, yields
\be{
	R_{g,2} = \frac{\mu_{n,m}^2}{n^2} \sum_{j=1}^{m}\frac{1}{N-j+1}\sum_{i=m+1}^N
	\IE\bclr{g_m(\pi,\Sigma)-g_m(\pi\tau_{ij},\Sigma)}^2.
}
By exchangeability the expectation on the right hand side is constant for~$j \le m$ and~$i \ge m+1$; hence, for such $i$ and $j$,
\be{
	\IE\bclr{g_m(\pi,\Sigma)-g_m(\pi\tau_{ij},\Sigma)}^2  = \IE\bclr{g_m(\pi,\Sigma)-g_m(\pi\tau_{1,m+1},\Sigma)}^2,
} 
so that
\bes{
	R_{g,2} &=\frac{\mu_{n,m}^2}{n^2}\sum_{j=1}^m\frac{N-m}{N-j+1} \IE\bclr{g_m(\pi,\Sigma)-g_m(\pi\tau_{1,m+1},\Sigma)}^2\\
	&\leq \frac{\mu_{n,m}^2m}{n^2}\IE\bclr{g_m(\pi,\Sigma)-g_m(\pi\tau_{1,m+1},\Sigma)}^2.
}
Now,
\bes{
	&\IE\bclr{g_m(\pi,\Sigma)-g_m(\pi\tau_{1,m+1},\Sigma)}^2\\
	&\qquad\leq2\IE\bclr{g_m(\pi,\Sigma)-g_{m+1}(\pi\tau_{1,m+1},\Sigma)}^2\\
	&\qquad\quad+2\IE\bclr{g_{m+1}(\pi\tau_{1,m+1},\Sigma)-g_m(\pi\tau_{1,m+1},\Sigma)}^2\\
	&\qquad= 2\IE\bclr{g_m(\pi,\Sigma)-g_{m+1}(\pi,\Sigma)}^2
	+2\IE\bclr{g_{m+1}(\pi,\Sigma)-g_m(\pi,\Sigma)}^2\\
	&\qquad= 4\IE\bclr{g_m(\pi,\Sigma)-g_{m+1}(\pi,\Sigma)}^2;
}
here, we have first applied the inequality~$(x+y)^2\leq 2x^2+2y^2$, followed by \eqref{60} with~$m$ replaced by~$m+1$ to the first expectation in the expression that results to yield that~$g_{m+1}(\pi\tau_{1,m+1},\Sigma)=g_{m+1}(\pi,\Sigma)$, and~$\pi\tau_{1,m+1} \eqlaw \pi$ to the second expectation, where $\eqlaw$ denotes equality in distribution. Hence,
\besn{\label{61}
	R_{g,2} &\leq \frac{4\mu_{n,m}^2m}{n^2}\IE\bclr{g_m(\pi,\Sigma)-g_{m+1}(\pi,\Sigma)}^2\\
  & = \frac{4\mu_{n,m}^2m}{n^2} \IE\bbbclr{\,\sum_{v\in[n]}\bclr{B_v(m,\pi,\sigma_v)-B_v(m+1,\pi,\sigma_v)}}^2.
}
Now, recalling that $L^v(m,\pi,\sigma_v)$ is the set of indices of edges
to which those edges adjacent to vertex $v$ were relocated, let
\besn{\label{62}
	N^v(m,\pi,\sigma_v) & = \bigcup_{i \in L^v(m,\pi,\sigma_v)} {\cal E}_n(i)\\
  & = \{w\in [n]\colon\text{$\exists i\in L^v(m,\pi,\sigma_v)$ such that $w\in\cE_{n}(i)$}\},
}
the set of vertices that received at least one additional edge when redistributing those edges. Also, let
\ben{\label{63}
	M^v(m,\pi,\sigma_v) = \{w\in[n]\setminus v\,:\, \{w,v\}\in\cG(m,\pi),w\not\in N^v(m,\pi,\sigma_v)\},
}
the neighbours of $v$ that did not receive a new edge when redistributing the edges incident on $v$.

Note that the chosen vertex $v$ will increase the difference $Y(m,\pi)-Y^v(m,\pi,\sigma_v)$ by one if it is isolated in ${\mathcal G}(m,\pi)$. A vertex $w \not = v$, will have this same effect if $w$ is isolated in ${\mathcal G}(m,\pi)$ but then has an edge attached to it in the redistribution of the removed edges of $v$. On the other hand, a vertex $w \not = v$ will decrease this difference by one when $w$ is connected to $v$, and has degree 1 in ${\mathcal G}(m,\pi)$, and does not have such an edge reattached. Hence, this difference is given by 
\ben{\label{64}
	B_v(m,\pi,\sigma_v) 
	= I_v(m,\pi) 
	+ \sum_{\substack{w\in\\ N^v(m,\pi,\sigma_v)}}I_w(m,\pi)
	    - \sum_{\substack{w\in\\ M^v(m,\pi,\sigma_v)}} I_{w,1}(m,\pi),
}
where~$I_{w,1}(m,\pi) = \I[d_w(m,\pi)=1]$.
Letting $\triangle$ denote set difference, we obtain
\besn{ \label{65}
	&\abs{B_v(m,\pi,\sigma_v)-B_v(m+1,\pi,\sigma_v)} \\[2ex]
	&\quad\leq \abs{I_v(m,\pi)-I_v(m+1,\pi)}\\
	&\qquad+ \bbbabs{\sum_{\substack{w\in\\ N^v(m,\pi,\sigma_v)}}I_w(m,\pi)-\sum_{\substack{w\in\\ N^v(m+1,\pi,\sigma_v)}}I_w(m+1,\pi)}\\
	&\qquad + \bbbabs{\sum_{\substack{w\in\\ M^v(m,\pi,\sigma_v)} }I_{w,1}(m,\pi)-\sum_{\substack{w\in\\ M^v(m+1,\pi,\sigma_v)} }I_{w,1}(m+1,\pi)}\\[2ex]
	&\quad\leq \I[v \in {\cal E}_n(\pi(m+1))] \\
	&\qquad + |N^v(m,\pi,\sigma_v) \cap {\mathcal E}_n(\pi(m+1))|
	\\
	&\qquad\qquad\qquad\qquad +\sum_{\substack{w\in\\ N^v(m,\pi,\sigma_v)\triangle N^v(m+1,\pi,\sigma_v)}} I_w(m+1,\pi)\\
	&\qquad+\sum_{\substack{w\in\\ M^v(m,\pi,\sigma_v)} }\abs{I_{w,1}(m,\pi)-I_{w,1}(m+1,\pi)}\\
	&\qquad\qquad\qquad\qquad+\sum_{\substack{w\in\\ M^v(m,\pi,\sigma_v)\triangle M^v(m+1,\pi,\sigma_v)} } I_{w,1}(m+1,\pi).
}
For the first term in \eqref{65}, we have used that for any vertex~$w \in [n]$ we can only have~$I_w(m,\pi)\not =I_w(m+1,\pi)$ when~$w$ is an endpoint of the additional edge determined by~$\pi(m+1)$, that is, when~$w\in\cE_n(\pi(m+1))$. For the second term in \eqref{65} we have used similarly that
\be{ 
	\sum_{\substack{w\in\\ N^v(m,\pi,\sigma_v)}}\abs{I_w(m,\pi)-I_w(m+1,\pi)}
	\leq |N^v(m,\pi,\sigma_v) \cap {\mathcal E}_n(\pi(m+1))|.
}
Moving now to the third term in \eqref{65}, if~$v\notin\cE_n(\pi(m+1))$ and~$\pi(m+1)\not\in L^v(m,\pi,\sigma_v)$, then~$L^v(m+1,\pi,\sigma_v)= L^v(m,\pi,\sigma_v)$; indeed, if $v\notin\cE_n(\pi(m+1))$, vertex $v$ has the same degree in both $\cG(m,\pi)$ and $\cG(m+1,\pi)$, and if also  $\pi(m+1)\not\in L^v(m,\pi,\sigma_v)$, then Algorithm 1 will redistribute the edges adjacent to $v$ to the same available pairs of vertices when $v$ has degree $m$ or $m+1$; indeed, note that between the two cases $m$ and $m+1$, Step~7 changes only if $\sigma_v(i)=\pi(m+1)$ for any of the $i$ tested there, which is equivalent to $\pi(m+1)\in L^v(m,\pi,\sigma_v)$). Therefore, if $L^v(m,\pi,\sigma_v)\neq L^v(m+1,\pi,\sigma_v)$, we must either have~$v\in{}\cE_n(\pi(m+1))$ or $\pi(m+1)\in L^v(m,\pi,\sigma_v)$. Now, if $v \in{}\cE_n(\pi(m+1))$, then the degree of $v$ in $\cG(m+1,\pi)$ is one more than its degree in $\cG(m,\pi)$, so~$L^v(m+1,\pi,\sigma_v)$ will contain one more edge than~$L^v(m,\pi,\sigma_v)$. And if $\pi(m+1)\in L^v(m,\pi,\sigma_v)$, then~$\abs{L^v(m,\pi,\sigma_v)\triangle L^v(m+1,\pi,\sigma_v)} = 2$ since~$\pi(m+1)$ will be found blocked when forming ${\cal G}^v(m+1,\pi,\sigma_v)$ and a new non-edge has to be found. Hence,
\bes{
	&\sum_{\substack{w\in\\ N^v(m+1,\pi,\sigma_v)\triangle N^v(m,\pi,\sigma_v)}} I_w(m+1,\pi) \\
  &\qquad\qquad\qquad\leq \abs{N^v(m+1,\pi,\sigma_v)\triangle N^v(m,\pi,\sigma_v)}\\
	&\qquad\qquad\qquad\leq 2 \I[v \in {\cal E}_n(\pi(m+1))]+4\I[\pi(m+1)\in L^v(m,\pi,\sigma_v)].
}
For the fourth term in \eqref{65} we apply the bound
\bes{
  &\sum_{\substack{w\in\\ M^v(m,\pi,\sigma_v)} }\abs{I_{w,1}(m,\pi)-I_{w,1}(m+1,\pi)} \\
  &\qquad\leq\sum_{\substack{w:\\\{w,v\}\in\cG(m,\pi)} }\abs{I_{w,1}(m,\pi)-I_{w,1}(m+1,\pi)}.
}
Finally, for the last term, similarly as for the third, if both~$v\notin\cE_n(\pi(m+1))$ and~$\pi(m+1)\not\in L^v(m,\pi,\sigma_v)$, it is easy to see that~$M^v(m+1,\pi,\sigma_v)= M^v(m,\pi,\sigma_v)$; indeed, under these conditions, the set of vertices adjacent to~$v$ does not change with the addition of edge $m+1$, and moreover, $L^v(m+1,\pi,\sigma_v)=L^v(m,\pi,\sigma_v)$, which implies $N^v(m+1,\pi,\sigma_v)=N^v(m,\pi,\sigma_v)$, so that $M^v(m+1,\pi,\sigma_v)= M^v(m,\pi,\sigma_v)$. Hence, if $M^v(m,\pi,\sigma_v)\neq M^v(m+1,\pi,\sigma_v)$, we must either have $v\in\cE_n(\pi(m+1))$ or $\pi(m+1)\in L^v(m,\pi,\sigma_v)$. 

If~$v \in\cE_n(\pi(m+1))$, then $v$ has one more neighbour in ${\mathcal G}(m+1,\pi)$ than in ${\mathcal G}(m,\pi)$, and so $L^v(m+1,\pi,\sigma_v)$ will contain one more edge than~$L^v(m,\pi,\sigma_v)$. In this case, $M^v(m,\pi,\sigma_v)$ and $M^v(m+1,\pi,\sigma_v)$ can differ by at most three elements. Indeed, they may only differ by the additional neighbour in ${\mathcal G}(m+1,\pi)$, and by at most two existing neighbours of $v$ in ${\mathcal G}(m,\pi)$ which were not assigned an edge in $L^v(m,\pi,\sigma_v)$, but were so assigned in $L^v(m+1,\pi,\sigma_v)$. 

If~$\pi(m+1)\in L^v(m,\pi,\sigma_v)$, then~$\abs{L^v(m,\pi,\sigma_v)\triangle L^v(m+1,\pi,\sigma_v)} = 2$, so that $M^v(m,\pi,\sigma_v)$ and $M^v(m+1,\pi,\sigma_v)$ can differ by at most four elements; hence
\bes{
  &\sum_{\substack{w\in\\ M^v(m,\pi,\sigma_v)\triangle M^v(m+1,\pi,\sigma_v)} } I_{w,1}(m+1,\pi) \\
  &\qquad\leq \abs{M^v(m,\pi,\sigma_v)\triangle M^v(m+1,\pi,\sigma_v)} \\
  &\qquad\leq 3\I[v \in {\cal E}_n(\pi(m+1))]+4\I[\pi(m+1)\in L^v(m,\pi,\sigma_v)].
}
Now recalling \eqref{61}, summing \eqref{65} over $v\in[n]$ and noting that
\be{
  \sum_{v\in[n]}\I[v\in{\cal E}_n(\pi(m+1))]\leq 2,
} 
we obtain
\besn{\label{66}
	R_{g,2} \leq 
	\frac{C\mu_{n,m}^2m}{n^2}\IE\bclc{1+ R_{g,2,1}^2 + R_{g,2,2}^2 + R_{g,2,3}^2},
}
where
\ba{
	R_{g,2,1} &= \sum_{v\in[n]}|N^v(m,\pi,\sigma_v) \cap {\mathcal E}_n(\pi(m+1))|\\
	 R_{g,2,2}&= \sum_{v\in[n]} \I[\pi(m+1)\in L^v(m,\pi,\sigma_v)]
 \\
 R_{g,2,3} &=\sum_{v\in[n]}\sum_{\substack{w:\{w,v\}\in\cG(m,\pi)} }\abs{I_{w,1}(m,\pi)-I_{w,1}(m+1,\pi)}.
}
For the first term, 
\besn{\label{67}
	\IE R_{g,2,1}^2 &= n\IE|N^1(m,\pi,\sigma_1) \cap {\mathcal E}_n(\pi(m+1))|^2\\
	&\quad+ n(n-1)\IE \{|N^1(m,\pi,\sigma_1) \cap\cE_n(\pi(m+1))|\\ &\kern10em\times|N^2(m,\pi,\sigma_2)\cap\cE_n(\pi(m+1))|\}.
}
Note that each vertex has at most $n-1$ potential edges available where the new edge $\pi(m+1)$ can be placed. Hence, since $N^1(m,\pi,\sigma_1)\leq 2d_1(m,\pi)$, there are at most $2d_1(m,\pi)(n-1)$ potential edges with one end in $N^1(m,\pi,\sigma_1)$, and so
\be{
	\IP[N^1(m,\pi,\sigma_1) \cap {\mathcal E}_n(\pi(m+1))\neq\emptyset\given d_1(m,\pi),N^1(m,\pi,\sigma_1)] 
	\leq\frac{2d_1(m,\pi)(n-1)}{N-m}.
}

Noting that $\abs{N^1(m,\pi,\sigma_v) \cap \cE_n(\pi(m+1))}$ is bounded by $2\I[N^1(m,\pi,\sigma_v) \cap \cE_n(\pi(m+1))\neq \emptyset]$, recalling that $d_1(m,\pi)\sim\Hyp(N,m,n-1)$ and using Lemma~\ref{lem2}, and also \eqref{47} of Condition~\eqref{6}, which gives that $m \le n^{3/2}$ as $\overline{c} \le 1$ by \eqref{46}, we therefore have
\be{
	n\IE|N^1(m,\pi,\sigma_v) \cap {\mathcal E}_n(\pi(m+1))|^2
	\leq C\bbclr{1+\frac{m}{n}}.
}
Moreover, with $\IP_{12}[\cdot] = \IP[\cdot\given d_1(m,\pi),d_2(m,\pi),N^1(m,\pi,\sigma_1),N^2(m,\pi,\sigma_2)]$, 
\bes{
  & \IP_{12}[N^1(m,\pi,\sigma_1) \cap\cE_n(\pi(m+1))\neq\emptyset,N^2(m,\pi,\sigma_2) \cap\cE_n(\pi(m+1))\neq\emptyset] \\
&\qquad \leq\frac{4d_1(m,\pi)d_2(m,\pi)}{N-m},
}
since there are at most $2d_1(m,\pi)\times 2d_2(m,\pi)$ potential edges with one end in $N^1(m,\pi,\sigma_1)$ and the other end in $N^2(m,\pi,\sigma_2)$.
Hence, again using $m \le n^{3/2}$ and Lemma~\ref{lem2}, and also Cauchy-Schwarz, we obtain
\bes{
 	&n(n-1)\IE \{|N^1(m,\pi,\sigma_1) \cap\cE_n(\pi(m+1))||N^2(m,\pi,\sigma_2)\cap\cE_n(\pi(m+1))|\}\\
  &\qquad \leq C\bbclr{1+\frac{m^2}{n^2}},
}
so that \eqref{67} results in the bound
\bea \label{68}
\IE R_{g,2,1}^2 \leq  C\bbclr{1+\frac{m^2}{n^2}}.
\ena
Next, we have
\besn{
	\IE R_{g,2,2}^2
	&= n\IP[\pi(m+1)\in L^1(m,\pi,\sigma_1)] \\
	&\quad+n(n-1)\IP[\pi(m+1)\in L^1(m,\pi,\sigma_1)\cap L^2(m,\pi,\sigma_2)].\label{69}
}
To calculate the first probability, we condition on~$\pi$ and average over~$\sigma_1$. If~$1 \in {\cal E}_n(\pi(m+1))$, then the conditional probability vanishes, as no edge incident on the (removed) vertex $v$ gets redistributed. Hence, take $\pi$ such that~$1 \not \in {\cal E}_n(\pi(m+1))$. To compute~$\IP[\pi(m+1)\in L^1(m,\pi,\sigma_1)|\pi]$, note that there are~$N-m$ non-edges of ${\mathcal G}(m,\pi)$, out of which~$n-1-d_1(m,\pi)$ involve vertex~$1$ and can therefore not be used during the redistribution of the~$d_1(m,\pi)$ edges incident to vertex~$1$, which is to be removed. This leaves~$N-m-n+1+d_1(m,\pi)$ potential edges from which to draw our sample of~$d_1(m,\pi)$ non-edges. By uniformity, the probability that~$\pi(m+1)$ is in this sample is given by
\bea \nonumber
p(\pi):={}&\IP[\pi(m+1)\in L^1(m,\pi,\sigma_1)|\pi] \\
={}&  \frac{d_1(m,\pi)}{N-m-n+1+d_1(m,\pi)} 
\leq \frac{d_1(m,\pi)}{N-m-n} \le\frac{Cd_1(m,\pi)}{N} ,\label{70}
\ena 
as we only ask for the probability that one special object is included in a simple random sample of~$d_1(m,\pi)$ objects from a population of size~$N-m-n+1+d_1(m,\pi)$, and where in the final inequality we have used \eqref{47} of Condition~\eqref{6}.
Averaging over~$\pi$, for the first term in \eqref{69} we obtain the bound
\ben{\label{71}
	n\IE p(\pi) \leq \frac{Cn\IE d_1(m,\pi)}{N} = \frac{ Cn(n-1)m}{N^2} \le \frac{Cm}{n^2}.
}
Next, as the events $\pi(m+1) \in L^1(m,\pi,\sigma_1)$ and $\pi(m+1) \in L^2(m,\pi,\sigma_2)$ are conditionally independent given $\pi$, we may handle the second, off diagonal term of \eqref{69} by using Lemma \ref{lem2} to give that
\bea \label{72}
\IE H^2 \le C\left(1+\frac{m^2}{n^2} \right)  \qmq{when} H \sim \Hyp(N,m,n),
\ena
which, recalling \eqref{70}, results in the bound
\bes{ 
  &n^2 \IE \bclc{\IP\cls{\pi(m+1) \in L^1(m,\pi,\sigma_1)\given\pi} \IP\cls{\pi(m+1) \in L^1(m,\pi,\sigma_2)\given\pi}}\\
&\quad = n^2 \IE p(\pi)^2\le \frac{Cn^2 \IE d_1(m,\pi)^2}{N^2} \le \frac{Cn^2}{N^2} \left(1+\frac{m^2}{n^2} \right).
}
Thus, using \eqref{69}, \eqref{71} and the inequality directly above, we obtain
\bea \label{73} 
\IE R_{g,2,2}^2  \le  C\bbclr{1+\frac{m^2}{n^2}}.
\ena

Finally, in order to bound $R_{g,2,3}$, note that the double sum is simply twice the sum over all the vertices of edges in ${\mathcal G}(m,\pi)$. Note also that, as $w$ must have degree at least one to be included in the sum, $I_{w,1}(m,\pi)\neq I_{w,1}(m+1,\pi)$ only if $w$ has degree 1 in $\cG(n,m)$ and it receives the additional edge $\pi(m+1)$. Thus, since the additional edge has two endpoints, it is immediate that $R_{g,2,3}$ can be no more than $4$, so that
\bea \label{74}
\IE R_{g,2,3}^2 \leq 16.
\ena
Recalling \eqref{66} and applying \eqref{68},  \eqref{73} and \eqref{74} yields
\be{
	 R_{g,2} \leq \frac{C\mu_{n,m}^2m}{n^2}\left( 1+ \frac{m^2}{n^2} \right).
}
Now, by Lemma \ref{lem6}, we have $\mu_{n,m}/n\leq \exp(-2m/n)$, and since $x\exp(-2x)$ remains bounded on the positive real numbers, it follows that $m\mu_{n,m}/n^2$ is bounded; hence,
\ben{ \label{75}
	 R_{g,2} \leq C\mu_{n,m} \left( 1+ \frac{m^2}{n^2} \right).
}

\paragraph{Bounding $\boldsymbol{R_{f,2}}$.} Using the same arguments as those used for $R_{g,2}$ to reach \eqref{61}, we can show that
\be{
  R_{f,2}\leq Cm\IE\bbbclr{\,\sum_{v\in[n]}\bclr{I_v(m,\pi)B_v(m,\pi,\sigma_v)-I_v(m+1,\pi)B_v(m+1,\pi,\sigma_v)}}^2
}
Adding and subtracting $I_v(m+1,\pi)B_v(m,\pi,\sigma_v)$, and splitting the sum, we obtain
\be{
  R_{f,2} \leq Cm\IE\bclc{R_{f,2,1}^2+R_{f,2,2}^2},
}
where
\ba{  
  R_{f,2,1}&=\sum_{v\in[n]}\bclr{I_v(m,\pi)-I_v(m+1,\pi)}B_v(m,\pi,\sigma_v),\\
  R_{f,2,2}&=\sum_{v\in[n]}I_v(m+1,\pi)\bclr{B_v(m,\pi,\sigma_v)-B_v(m+1,\pi,\sigma_v)}.
}

In order to bound $R_{f,2,1}$, note first that $I_v(m,\pi)-I_v(m+1,\pi)$ is non-zero, and in that case equals one, exactly when vertex $v$ is isolated in $\cG(m,\pi)$ and the $(m+1)^{\mathrm{th}}$ added edge is incident on $v$; that is, 
\be{
  I_v(m,\pi)-I_v(m+1,\pi) = I_v(m,\pi)\I[v\in\cE_n(\pi(m+1))].
} 
And since $I_v(m,\pi)=1$ implies $B_v(m,\pi,\sigma_v)=1$, we have
\be{
  R_{f,2,1} = \sum_{v\in[n]} I_v(m,\pi)\I[v\in\cE_n(\pi(m+1))].
}
Squaring, taking expectation and using exchangeability, we obtain 
\bes{
  & \IE R_{f,2,1}^2\\
  &\quad = n\IE\bclc{I_1(m,\pi)\I[1\in\cE_n(\pi(m+1))]}  \\
  &\quad\quad + n(n-1)\IE\bclc{I_1(m,\pi)I_{2}(m,\pi)\I[1\in\cE_n(\pi(m+1))]
    \I[2\in\cE_n(\pi(m+1))]}\\
  & \qquad =: R_{f,2,1,1}+R_{f,2,1,2}  .
}
For the first term, we have
\bes{
  R_{f,2,1,1} &= \mu_{n,m}\IP[1\in\cE_n(\pi(m+1))\given d_1(m,\pi)=0]
  = \frac{\mu_{n,m}(n-1)}{N-m} \leq \frac{C\mu_{n,m}}{n},
}
while for the second term
\bes{
  R_{f,2,1,2} &= n(n-1)\IP[d_1(m,\pi)=d_2(m,\pi)=0]\\
  &\quad\times\IP[\cE_n(\pi(m+1))=\{1,2\}\given d_1(m,\pi)=d_2(m,\pi)=0] \\
  & = \Hyp(N,m,2n-3)(\{0\})\frac{n(n-1)}{N-m}
  \leq \frac{C\mu_{n,m}}{n},
}
where we have used that $\Hyp(N,m,n)(\{0\})$, the probability that a hypergeometric variable with the given parameters takes the value 0, is a decreasing function of the number of special items $n$. Hence, 
\be{
  \IE R_{f,2,1}^2\leq \frac{C\mu_{n,m}}{n}
}
In order to handle $R_{f,2,2}$, note that if $I_v(m+1,\pi)=1$, we necessarily have $I_v(m,\pi)=1$, so that 
$B_v(m,\pi,\sigma_v)=B_v(m+1,\pi,\sigma_v)=1$ whenever $I_v(m+1,\pi)=1$; it follows that 
\be{
R_{f,2,2}=0.
}
Therefore,
\ben{\label{76}
  R_{f,2}\leq C\mu_{n,m}\frac{m}{n} \le C\mu_{n,m}  \left( 1+  \frac{m^2}{n^2}\right).
}
Combining the bounds \eqref{58}, \eqref{59}, \eqref{75} and \eqref{76} as in \eqref{55}, and then recalling \eqref{53}, we obtain
\be{
  \IE\babs{\IE(1-GD|W)} 
  \le \frac{C\sqrt{\mu_{n,m}}}{\sigma^2_{n,m}}\sqrt{1+\frac{m^2}{n^2}}.
}
Recalling \eqref{36} and noting that $\sigma_{n,m}^2\leq \mu_{n,m}$ by Lemma~\ref{lem5}, the first condition in \eqref{10} holds, as
\bes{
	\sup_{(n,m)\in\Goodset} r_{n,m}\,	\IE\babs{\IE(1-GD|W)}  < \infty.
}

Next, it clearly suffices to verify the second condition in \eqref{10} of \eqref{9} with $D$ replaced by its absolute upper bound \begin{align}\label{77}
\overline{D}=\frac{1+2d_\Vee}{\sigma_{n,m}}
\end{align}
obtained in \eqref{57}, and splitting the resulting expression to be bounded into two terms, we have
\begin{align}\label{78}
\IE\bclc{(1+\abs{W})\abs{G}\overline{D}^2} = \IE\bclc{\abs{G}\overline{D}^2}+\IE\bclc{\abs{W}\abs{G}\overline{D}^2}.
\end{align}
Now, let $a\geq 1$. Using the given form \eqref{51} of $G$,
we obtain
\besn{ \label{79}
  \IE\bclc{\abs{G}\overline{D}^a} 
  & = \frac{n}{\sigma_{n,m}} \IE\bclc{|I_\Vee-\frac{\mu_{n,m}}{n}|\overline{D}^a}\\
  & \leq \frac{n}{\sigma_{n,m}} \IE\bclc{I_\Vee\overline{D}^a} 
  + \frac{\mu_{n,m}}{\sigma_{n,m}} \IE \overline{D}^a 
  \le \frac{C\mu_{n,m}}{\sigma_{n,m}^{1+a}}   \left(1 + \left( \frac{m}{n} \right)^a \right),
}
where, for the final inequality, we used that $\overline{D}=1/\sigma_{n,n}$ when  $I_\Vee=1$ and that $\IE I_\Vee=\mu_{n,m}/n$ on the first summand,  and Lemma \ref{lem2} on the second summand. Setting $a=2$ we obtain the bound 
\begin{align} \label{80}
\IE\bclc{|G|\overline{D}^2} \le \frac{C \mu_{n,m}}{\sigma_{n,m}^3}\left( 1+\left(\frac{m}{n}\right)^2  \right)
\end{align}
on the first term of \eqref{78}.

The second term in \eqref{78} likewise leads to two terms, corresponding to the two in the second line of \eqref{79}, but with an additional factor of $|W|$. Now setting $a=2$, for the first we have, by applying Cauchy-Schwarz,
\ben{\label{81}
  \frac{n}{\sigma_{n,m}}\IE\bclc{\abs{W}I_\Vee\overline{D}^2}
  \leq \frac{\mu_{n,m}}{\sigma^{3}_{n,m}}\IE\bclc{\abs{W}\given I_\Vee=1}
  \leq \frac{\mu_{n,m}}{\sigma^{3}_{n,m}}\sqrt{\IE\bclc{W^2\given I_\Vee=1} }.
}
Conditional on vertex $\Vee$ being isolated, the distribution of the number of isolated vertices in the $\ERRG(n,m)$ model is one more than the number of isolated vertices in the $\ERRG(n-1,m)$ model. Hence, writing
\be{
W=\frac{\sigma_{n-1,m}}{\sigma_{n,m}}\times \frac{Y-\mu_{n-1,m}}{\sigma_{n-1,m}} + \frac{\mu_{n-1,m}-\mu_{n,m}}{\sigma_{n,m}},
}
and using $(x+y)^2 \le 2(x^2+y^2)$ twice, we obtain
\bes{
&\IE_{n,m} \bclc{W^2\given I_\Vee=1}  \\
&\qquad\le 2 \left( \frac{\sigma_{n-1,m}}{\sigma_{n,m}} \right)^2
\IE_{n-1,m}
\left(\frac{Y-\mu_{n-1,m}+1}{\sigma_{n-1,m}} \right)^2 + 
2\left(\frac{\mu_{n-1,m}-\mu_{n,m}}{\sigma_{n,m}}\right)^2\\
&\qquad\le  \frac{4\sigma_{n-1,m}^2}{\sigma_{n,m}^2} 
+ \frac{4+2(\mu_{n-1,m}-\mu_{n,m})^2}{\sigma_{n,m}^2}.
}

Lemma \ref{lem9} yields that the first term is bounded by a constant. 
For the second term, 
by removing all edges from the $n^{\mathrm{th}}$ vertex and relocating them among the remaining vertices, we have a coupling of  $\ERRG(n,m)$ and  $\ERRG(n-1,m)$ which yields $|Y_{n-1,m}-Y_{n,m}| \le 1+2d_n$, so that
\begin{align*}
|\mu_{n-1,m}-\mu_{n,m}| \le 1+2\IE d_n \le 1 + \frac{2mn}{N} \leq C\left(1+\frac{m}{n}\right).
\end{align*}
Using that $r_{n,m}$ in \eqref{36} is lower bounded by $\overline{r}$, which is at least 1 by Lemma \ref{lem8}, and that $\mu_{n,m} \ge \sigma_{n,m}^2$ by Lemma \ref{lem5} yields
 $\sigma_{n,m}\geq (1+(m/n)^2)$, and using also \eqref{81}, we conclude that
\begin{align} \label{82}
\IE [W^2|I_\Vee=1] \le C \quad \mbox{and hence} \quad \frac{n}{\sigma_{n,m}}\IE\bclc{\abs{W}I_\Vee\overline{D}^2} \le \frac{\mu_{n,m}}{\sigma_{n,m}^3}.
\end{align}

For the corresponding second term of \eqref{79}, with $a=2$ and the additional factor of $|W|$, using Cauchy-Schwarz and $\IE W^2 =1$, 
\begin{align} \label{83}
\frac{\mu_{n,m}}{\sigma_{n,m}}  
\IE \left[|W|D^2\right]  \le \frac{\mu_{n,m}}{\sigma_{n,m}}   \sqrt{\IE \left[D^4\right]} \le \frac{C\mu_{n,m}}{\sigma_{n,m}^3} \left(1 + \left( \frac{m}{n} \right)^2 \right),
\end{align}
applying Lemma \ref{lem2}. Combining with \eqref{80} and \eqref{82} we see the sum is of the order of \eqref{83}
and it follows that
\be{
	\sup_{(n,m)\in\Goodset} r_{n,m}\,\IE_{n,m}\bclc{(1+\abs{W})\abs{G}D^2} < \infty.
}

\item[Condition~(\ref{11}).] Let~$(n,m)\in\Goodset$, and define
\ben{ \label{def1}
	\cF_{n,m} = \sigma\bclr{
		\Vee, d_\Vee(m,\pi)},
}
the~$\sigma$-algebra generated by the identity of the vertex chosen to be removed in the coupling and its degree.
Letting $\overline{G}=|G|$, and $\overline{D}$ be as in \eqref{77}, we see that both are clearly $\cF_{n,m}$-measurable.

For the first condition in \eqref{12}, let
\bea \label{84}
F_{n,m,1} =\{d_\Vee(m,\pi) \le t(n,m) \} \qmq{where} t(n,m)=\min\{n,m\}/4.
\ena
Recall \eqref{46} and \eqref{47}; in particular, on $\Goodset$, we have $n \ge 344$ and $28 \le m \leq n^{3/2}$. It is straightforward to check that under these conditions,
\bea \label{85}
\underline{t}(n,m):= \frac{4m}{n}+2\log(m\wedge n)  \leq t(n,m)
\quad\text{for all $(n,m)\in\Goodset$.}
\ena
Indeed, if for $m \le n$, the bound follows using that $2 \log m \le (1/4-4/344)m$ for $m \ge 28$, while for $n \le m$ one verifies, for $n \ge 344$, that $4\sqrt{n}+2\log n\leq n/4$.

Now, bounding~$D$ by~$\overline{D}$ as given in \eqref{77}, writing $F$ as short for $F_{n,m,1}$ and using that $\IP[I_\Vee=0]=1-\mu_{n,m}/n$ in the final inequality, we obtain
\bes{
  \IE \bclc{ |G|D^2(1-I_F)} 
  &\le\frac{n}{\sigma_{n,m}} \IE \bclc{ \abs{I_\Vee-{\mu_{n,m}}/{n}}\overline{D}^2(1-I_F)} \\
  &\le\frac{n}{\sigma_{n,m}} \IE \bclc{ I_\Vee\overline{D}^2(1-I_F)} 
  +\frac{\mu_{n,m}}{\sigma_{n,m}} \IE \bclc{\overline{D}^2(1-I_F)}.
}
Since $\Vee$ cannot be both isolated and have positive degree, we have $I_\Vee (1-I_F) = 0$ almsot surely, and so the first term is zero. Applying Cauchy-Schwarz to the second term and then invoking Lemma \ref{lem2}, 
\begin{align}
  \IE \bclc{ |G|D^2(1-I_F)}
  &\leq \frac{\mu_{n,m}}{\sigma_{n,m}}\bclr{\IE \overline{D}^4 \IE(1-I_F)}^{1/2} \nonumber \\
  &\leq \frac{C\mu_{n,m}\bclr{1+\bclr{\frac{m}{n}}^2}}{\sigma_{n,m}^3}\IP[d_\Vee(m,\pi)> t(n,m)]^{1/2}.\label{86}
\end{align}
By Lemma~\ref{lem2} with~$\gamma =2m/n$ being the mean of~$d_1(m,\pi)$, we have for any~$t>\gamma~$ that 
\bes{
	\IP[d_\Vee(m,\pi)> t] 
	& \leq \IP[d_\Vee(m,\pi)> \gamma  + (t-\gamma )]\\
	& \leq \exp\bbbclr{-\frac{(t-\gamma )^2}{t+\gamma }}
	\leq \exp\bbclr{-\frac{t-2\gamma }{2}};
}
trivially, the final expression upper bounds the left hand side for~$t \le \gamma~$ as well and hence holds for all~$t \ge 0$. Hence, with  $\underline{t}(n,m)$ as in \eqref{85}, by \eqref{86} and recalling $r_{n,m}$ in \eqref{36}, we obtain
\besn{ \label{87}
	& r_{n,m}^2\IE \{ |G| D^2 (1-I_F)\}
	  \le \frac{ C\sigma_{n,m}^3}{\mu_{n,m}(1+(\frac{m}{n})^2)} \exp\bbbclr{-\frac{\underline{t}(n,m)-2\gamma }{4}} 
	 \\ &\qquad \le \frac{C(m\wedge n)^{1/2}}{1+(\frac{m}{n})^2}\exp\bbbclr{-\frac{\underline{t}(n,m)-2\gamma }{4}} \\
	&\qquad =   \frac{C(m\wedge n)^{1/2}}{1+\bclr{\frac{m}{n}}^2} \exp\bbbclr{-\frac{1}{2}\log (m \wedge n)} 
	= \frac{C}{1+\bclr{\frac{m}{n}}^2} \le C,
}
where we have used 
that~$\sigma^2_{n,m}\leq \min\{\mu_{n,m},2m\}$ via Lemma~\ref{lem5}, and trivially $\mu_{n,m} \le n$,  for the second inequality, thus showing the first condition in \eqref{12} is satisfied. 

From \eqref{79} with $a=2$ it follows that
\be{
	\sup_{(n,m)\in\Goodset} r_{n,m}\,\IE_{n,m}\bclc{\overline{G}\overline{D}^2} < \infty,
}
thus showing that the second condition in \eqref{12} is also satisfied.

\item[Condition~(\ref{13}).]\def\rec{{\mathrm{emb}}}
Denote by~${\mathcal G}^{{\rm emb},\Vee}$ the ``embedded'' graph obtained by removing vertex~$\Vee$ and all its incident edges; we keep the original vertex labeling. 
As the remaining $m-d_\Vee(m,\pi)$ edges are uniformly distributed over the remaining~$n-1$ vertices, conditional on~${\cal F}_{n,m}$ in \eqref{def1}, the resulting graph has conditional distribution
\ben{\label{88}
	\law({\mathcal G}^{{\rm emb},\Vee}| \cF_{n,m}) \sim \ERRG(n-1,m-d_\Vee(m,\pi))
}
almost surely; this identity is again to be understood up to labeling. In particular, letting
$d^{{\rm emb},\Vee}_w$ be the degree of vertex~$w$ in graph~${\mathcal G}^{{\rm emb},\Vee}$, 
\be{
	V = \sum_{w: w \not = \Vee} \I[d^{{\rm emb},\Vee}_w=0]
}
is the number of isolated vertices of~${\mathcal G}^{{\rm emb},\Vee}$, and \eqref{88} implies
\be{
	\law(V|\cF_{n,m}) = \law_{\Psi}(Y) \qmq{where} \Psi=(n-1,m-d_\Vee(m,\pi)).
}
Clearly~$\Psi$ is~$\cF_{n,m}$-measurable. Now set $F_{n,m,2}=F_{n,m,1}$ as in \eqref{84}, which is also clearly  $\cF_{n,m}$ measurable. Condition \eqref{15} is 
clearly equivalent to the first condition in \eqref{12}, which was verified in \eqref{87}.

\item[Condition~(\ref{16}).]
Let 
\be{
	\overline{B} = \frac{d_\Vee(m,\pi)+1}{\sigma_{n,m}},
}
which is clearly~$\cF_{n,m}$-measurable. Moreover,~$\sigma_{n,m}^{-1}\abs{Y-V}\leq \overline{B}$ since removing any edge connected to vertex~$\Vee$ can make at most one vertex, other than~$\Vee$, isolated; the additional term of one accounts for the case when vertex~$\Vee$  is isolated. 
Since~$\overline{B} \le \overline{D}$, as given in \eqref{77}, by setting $a=3$ in \eqref{79} we obtain
\bes{ 
  r_{n,m}^2\IE\bclc{\overline{G}\overline{D}^2\overline{B}I_{F_{n,m,2}}}  
  & \le 	r_{n,m}^2 \IE\bclc{\overline{G}\overline{D}^2\overline{B}}
	\le r_{n,m}^2 \IE\bclc{\overline{G}\overline{D}^3} \\
	& \le  r_{n,m}^2 \frac{C\mu_{n,m}(1+(\frac{m}{n})^3)}{\sigma_{n,m}^4} = \frac{C\sigma_{n,m}^2}{\mu_{n,m}(1+\frac{m}{n})}.
}
As $\sigma_{n,m}^2 \le \mu_{n,m}$ via Lemma \ref{lem5}, the second bound in \eqref{17} holds.

\item[Condition~(\ref{18}).] We verify the stronger conditions that \eqref{19} and the second bound of \eqref{20} hold when taking the larger supremum obtained when removing the intersection with  $\{\Psi\in\Goodset\}$. This stronger version 
of \eqref{19} is an immediate consequence of  Lemma \ref{lem9}. As this same lemma shows that the ratios in \eqref{20} involving means and variances are bounded by a constant, it is only required to bound the ratios of the remaining factor.  
For~$r_{n,m}/r_{n-1,m-d}$, we have 
\beas
\frac{(1+((m-d)/(n-1))^2}{(1+(m/n)^2)} \le  \frac{1+2(m/n)^2}{1+(m/n)^2} \le 2, 
\enas
and for the reciprocal, using that~$m/n \le 2(m-d)/(n-1)$ for~$d \le m/4$, 
\beas
\frac{(1+(m/n)^2)}{(1+((m-d)/(n-1))^2} \le 4.
\enas
			
\end{description}
Conditions~\eqref{6}--\eqref{18} have been verified, and Theorem \ref{thm3} now follows from Theorem~\ref{thm1}.
\end{proof}

\subsection{Technical results}
 \begin{lemma}[Efron-Stein-type variance bound]\label{lem1} Let~$\pi$ and the components of~$\Sigma=(\sigma_1,\dots,\sigma_n)$ be independent uniform random permutations of~$[N]$, and let~$h(\pi,\Sigma)$ be a real-valued function. Let~$\tau_1,\dots,\tau_{N-1}$ be random transpositions independent of each other and of~$(\pi,\Sigma)$, where~$\tau_j$ transposes~$j$ and a uniformly chosen integer in the set~$\{j,\ldots,N\}$. Let~$\Sigma'=(\sigma_1',\ldots,\sigma_n')$ be an independent copy of~$\Sigma$ and let~$\Sigma'_i = (\sigma_1,\dots,\sigma_{i-1},\sigma_i',\sigma_{i+1},\ldots,\sigma_n)$. 
Then
\be{
	\Var h(\pi,\Sigma) 
	\leq \frac12\sum_{i=1}^n\IE\bclr{h(\pi,\Sigma)-h(\pi,\Sigma_i')}^2+ \frac{1}{2} \sum_{j=1}^{N-1} \IE\bclr{h(\pi,\Sigma)-h(\pi\tau_j,\Sigma)}^2.
}
\end{lemma}
	
\begin{proof} Without loss of generality assume~$\IE h(\pi,\Sigma)=0$. Let~$\pi_0 = \pi$ and~$\Sigma_0=\Sigma$, and let
\be{
	\pi_j = \pi_0 \tau_{j} \cdots \tau_{1}, \qquad 1 \leq j\leq N-1,
}
and
\be{
	\Sigma_i = (\sigma'_1,\dots,\sigma_i',\sigma_{i+1},\dots,\sigma_{n}),
	\qquad 1\leq i \leq n.
}	
		
Let~$B$ be uniform on~$\{0,1\}$, let~$I$ be uniform on~$\{1,\dots,n\}$, let~$J$ be uniform on~$\{1,\dots,N-1\}$, and assume~$B$,~$I$ and~$J$ are mutually independent and independent of all else. Let 
$W=h(\pi_0,\Sigma_0)$, let~$W_{1,i}' = h(\pi_0,\Sigma_i')$, and let~$W_{2,j}' = h(\pi_0 \tau_j,\Sigma_0)$, and~$W' = BW_{1,I}'+(1-B)W_{2,J}'$. Let~$G_{1,i} = n\bclr{h(\pi_{N-1},\Sigma_i) - h(\pi_{N-1},\Sigma_{i-1})}$, let~$G_{2,j} = (N-1)\bclr{h(\pi_j,\Sigma_0) - h(\pi_{j-1},\Sigma_0)}$, and let~$G = BG_{1,I}+(1-B)G_{2,J}$. Let~$g:\IR\to\IR$ be any bounded measurable function. Then, on the one hand,
\bes{
	-\IE\bclc{Gg(W)}
	& = -\frac{1}{2}\sum_{i=1}^{n} \IE\bclc{\bclr{h(\pi_{N-1},\Sigma_i) - h(\pi_{N-1},\Sigma_{i-1})}g(W)} \\
	& \qquad   - \frac{1}{2}\sum_{j=1}^{N-1} \IE\bclc{\bclr{h(\pi_j,\Sigma_0) - h(\pi_{j-1},\Sigma_0)}g(W)} \\
	& = -\frac{1}{2}\IE\bclc{\bclr{h(\pi_{N-1},\Sigma_{n}) - h(\pi_{N-1},\Sigma_0)}g(W)} \\
	&\qquad  - \frac{1}{2}\IE\bclc{\bclr{h(\pi_{N-1},\Sigma_0) - h(\pi_{0},\Sigma_0)}g(W)}\\
	& = -\frac{1}{2}\IE h(\pi_{N-1},\Sigma_{n})\IE g(W) + \frac{1}{2}\IE\bclc{h(\pi_{0},\Sigma_0)g(W)} \\
	& = \frac{1}{2}\IE\clc{Wg(W)},
}
where we used that~$(\pi_{N-1},\Sigma_{n})$ is equal in distribution to and independent of~$(\pi_{0},\Sigma_{0})$; this follows e.g.\ from \emph{Algorithm P} of \cite[p.~147]{Knuth1969} since the distribution of $\pi_{N-1}$ is uniform conditionally on $\pi_{0}$, and therefore independent of $\pi_{0}$. 

On the other hand, for all $i$ we have~$(\Sigma_i,\Sigma_{i-1},\Sigma'_i)\eqlaw(\Sigma_{i-1},\Sigma_{i},\Sigma_0)$ since $\Sigma\eqlaw \Sigma_i'$, and for all $j$ that $(\pi_j,\pi_{j-1},\pi_0 \tau_j)\eqlaw(\pi_{j-1},\pi_j,\pi_0)$, by recalling the definition of~$\pi_j$  and observing that~$\pi_0$ and~$\pi_0\tau_j$ have the same distribution, and that both are independent of~$\tau_{j-1}\cdots\tau_1$, so
\bes{
	\IE\bclc{G g(W')} 
	& = \frac{1}{2}\sum_{i=1}^{n} \IE\bclc{\bclr{h(\pi_{N-1},\Sigma_i) - h(\pi_{N-1},\Sigma_{i-1})}g(h(\pi_0,\Sigma'_i))} \\
	& \qquad +\frac{1}{2}\sum_{j=1}^{N-1} \IE\bclc{\bclr{h(\pi_j,\Sigma_0) - h(\pi_{j-1},\Sigma_0)}g(h(\pi_0\tau_{j},\Sigma_0))} \\
	& = \frac{1}{2}\sum_{i=1}^{n} \IE\bclc{\bclr{h(\pi_{N-1},\Sigma_{i-1}) - h(\pi_{N-1},\Sigma_{i})}g(h(\pi_0,\Sigma_0))} \\
	& \qquad +\frac{1}{2}\sum_{j=1}^{N-1} \IE\bclc{\bclr{h(\pi_{j-1},\Sigma_0) - h(\pi_{j},\Sigma_0)}g(h(\pi_0,\Sigma_0))} \\
	& =  \frac{1}{2}\IE \bclc{(h(\pi_{N-1},\Sigma_0)-h(\pi_{N-1},\Sigma_n))g(h(\pi_0,\Sigma_0))}\\
	& \qquad + \frac{1}{2}\IE \bclc{(h(\pi_{0},\Sigma_0)-h(\pi_{N-1},\Sigma_0))g(h(\pi_0,\Sigma_0))}\\
	& =\frac{1}{2} \IE\bclc{Wg(W)}.
}
Therefore,~$(W,W',G)$ is a Stein coupling and, specializing \eqref{2} to the case~$f(x)=x$ and applying the Cauchy Schwarz inequality and noting that $(\Sigma_i,\Sigma_{i-1})\eqlaw (\Sigma_i',\Sigma)$, we have
\bes{
	\Var W 
	& = \IE\clc{G(W'-W)} \\
	& = \frac12\sum_{i=1}^{n}\IE\bclc{\bclr{h(\pi_{N-1},\Sigma_i)-h(\pi_{N-1},\Sigma_{i-1})}\bclr{h(\pi_0,\Sigma'_i)-h(\pi_0,\Sigma)}} \\
	&\qquad +  \frac12\sum_{j=1}^{N-1}\IE\bclc{\bclr{h(\pi_j,\Sigma_0)-h(\pi_{j-1},\Sigma_0)}\bclr{h(\pi_0\tau_j,\Sigma_0)-h(\pi_0,\Sigma_0)}}\\
	& \leq \frac12\sum_{i=1}^{n}\bbclr{\IE\bclr{h(\pi_{N-1},\Sigma_i)-h(\pi_{N-1},\Sigma_{i-1})}^2\IE\bclr{h(\pi_0,\Sigma'_i)-h(\pi_0,\Sigma)}^2}^{1/2} \\
	&\qquad + \frac12\sum_{j=1}^{N-1}\bbclr{\IE\bclr{h(\pi_j,\Sigma)-h(\pi_{j-1},\Sigma)}^2\IE\bclr{h(\pi_0\tau_j,\Sigma)-h(\pi_0,\Sigma)}^2}^{1/2},
}
from which the claim follows.
\end{proof}

\begin{lemma}[Tail and moment bounds for the hypergeometric distribution]\label{lem2} Let~$ H$ have the hypergeometric distribution\/~$\Hyp(N,m,n)$ counting the number of white balls among~$m$ draws from an urn with~$N$ balls,~$n$ of which are white and~$N-n$ black. Let~$\gamma =\IE  H = nm/N$. Then, for any~$t>0$,
\ben{\label{89}
	\IP[ H\geq \gamma  + t] \leq \exp\bbbclr{\frac{-t^2}{2\gamma +t}}
}
Moreover, for any~$k\geq 1$, there is a constant~$C_k$ independent of~$\gamma$ such that
\be{
	\IE  H^k \leq C_k(\gamma ^k+1).
}
\end{lemma}
\begin{proof}To construct a bounded size bias coupling, index the white balls by~$[n]$, and write~$ H=\sum_{i=1}^n I_i$ where
$I_i$ is the indicator that the~$i^{\mathrm{th}}$ white ball is sampled. Construct~$ H^s$ with the $ H$-size biased distribution by uniformly sampling a random index~$J$ from~$1$ to~$n$ independently of~$I_1,\ldots,I_n$; if~$I_J=1$, set~$ H^s =  H$, otherwise independently and uniformly select a ball from the sample and swap it with the~$J^{\mathrm{th}}$ white ball. It is easy to see that~$ H^s$ has the size-bias distribution, see for instance, Lemma 2.1 of \cite{Goldstein96}. Moreover,~$ H^s =  H+1$ if a sampled black ball was swapped with the~$J^{\mathrm{th}}$ white ball, and~$ H^s =  H$ otherwise. Hence,~$\abs{ H^s- H}\leq 1$, and the tail-bound \eqref{89} follows readily from Theorem 1.1 of \cite{Ghosh2011}. 
		
Now, it is straightforward to check that~$t^2/(2\gamma +t)\geq (t-1)/(\gamma +1)$ whenever~$t\geq 1$ and~$\gamma >0$, so that
\be{
	\IP[ H \geq \gamma  + t] \leq \exp\bbbclr{\frac{-(t-1)}{\gamma +1}} \qm{for all~$t \ge 1$.}
}
Hence, $ H-\gamma -1$ is stochastically dominated by an exponential random variable~$X$ with mean~$1/(\gamma +1)$, and in particular
\bes{
	\IE  H^k 
	\leq \IE(X+\gamma +1)^k 
	\leq 3^{k-1}(\IE X^k+\gamma ^k+1)
	= 3^{k-1}(k!(\gamma +1)^k+\gamma ^k+1),
}
from which the second claim easily follows. 
\end{proof}

A bound similar to \eqref{89} can be obtained from \cite[Corollary~1]{Greene2017} with better constants, but under additional conditions on the parameters of the hypergeometric distribution
	
\begin{lemma}\label{lem3} If~$H \sim\Hyp(N,m,n)$, then
\be{
	\frac{mn}{N}-\frac{m^2n^2}{2N^2}\enskip\leq\enskip 1-e^{-mn/N} \enskip\leq\enskip \IP[ H>0] \enskip\leq\enskip \frac{mn}{N}
}
and
\be{
	e^{-mn/(N-m-n+1)} \enskip\leq\enskip \IP[ H=0] \enskip\leq\enskip 
	e^{-mn/N}, 
}
where the lower bound on~$\IP[ H=0]$ is valid whenever~$m+n-1<N$.
\end{lemma}
\begin{proof} Since~$\IP[ H>0]\leq \IE  H$, the upper bound on~$\IP[ H>0]$ immediately follows. Using the usual exponential upper bound for the final inequality,
\besn{ \label{90}
\bbclr{1-\frac{n}{N-m+1}}^m 	\le \IP\cls{ H=0} 
	& = \bbclr{1-\frac{n}{N}}\cdots\bbclr{1-\frac{n}{N-m+1}}\\
	&\leq\bbclr{1-\frac{n}{N}}^m 
	\leq e^{-mn/N},
}
from which the upper bound on~$\IP[ H=0]$ and first lower bound on~$\IP[ H>0]$ follow. The second lower bound on~$\IP[ H>0]$ follows from the first lower bound and the inequality~$e^{-x}\leq 1-x+x^2/2$ when~$x\geq 0$. The lower bound on~$\IP[ H=0]$ follows from the inequality~$\log(1+x) \ge x/(1+x)$ for~$x>-1$ and the lower bound in \eqref{90}, which together yield
\be{
\IP[ H=0]  \ge \exp\bbclr{-\frac{mn}{(N-m+1)\bclr{1-\frac{n}{N-m+1}}}}
 = \exp\bbclr{-\frac{mn}{N-m-n+1}}.\qedhere
}	
\end{proof}

\begin{lemma}\label{lem4} For any~$x\geq 0$
\be{
  \frac{\min\{x^2,1\}}4\leq 1-e^{-x}(1+x)\leq \frac{\min\{x^2,2\}}{2}.
}
\end{lemma}

\begin{proof}
The upper and lower bounds hold trivially at~$x=0$. With~$\psi(x)=1-e^{-x}(1+x)$, by Talyor's expansion around zero, for all~$x > 0$ there exists~$\xi_x \in (0,x)$ such that
\bg{
\psi(x)=\psi(0)+x\psi'(0)+\frac{x^2}{2}\psi''(\xi_x)=\frac{x^2}{2}\psi''(\xi_x), \\
\qmq{where}\psi'(x)=xe^{-x} \qmq{and} \psi''(x)=e^{-x}(1-x).
}
For~$y \in [0,2]$ we have~$|\psi''(y)| \le |1-y|\le 1$, thus proving the upper bound $x^2/2$ over this interval. As~$\psi'''(y)=e^{-y}(y-2) \ge 0$ for all~$y \ge 2$, the function~$\psi''(y)$ is non-decreasing for~$y \ge 2$. As~$\psi''(2)=-e^{-2} \in (-1,0)$, and~$\lim_{y \rightarrow \infty} \psi''(y)=0$, we have~$\psi''(y) \in (-1,0)$ for all~$y \ge 2$, thus proving the upper bound $x^2/2$ on $(2,\infty)$. As $\psi'(x) \ge 0$ for all $x \ge 0$, the function is non-decreasing on $[0,\infty)$, and as $\psi(x) \rightarrow 1$ as $x \rightarrow \infty$, we have $\psi(x) \le 1$ for all $x \ge 0$.

For the lower bound, for~$x>0$ letting 
\beas
q(x)=\frac{1-e^{-x}(1+x)}{x^2}, 
\qmq{
we have
}
q'(x)=\frac{e^{-x}(x^2+2x+2)-2}{x^3}.
\enas
With~$p(x)=e^{-x}(x^2+2x+2)-2$ we have~$p'(x)=-x^2e^{-x} \le 0$, so~$q(x)$ is decreasing for~$x>0$. In particular,~$q(x) \ge q(1)=1-2e^{-1} \ge 1/4$ for~$x \in [0,1]$. As~$\psi'(x) = xe^{-x}$, the function~$\psi(x)$ is non-decreasing, and hence for~$x \ge 1$ we have~$\psi(x) \ge \psi(1) = 1-2e^{-1} \ge 1/4$, completing the proof of the lower bound. 
\end{proof}

\begin{lemma}\label{lem5} For all $(n,m)\in\Theta$ and distinct vertices $v$ and $w$, the indicators $\I[d_w=0]$ and $\I[d_v=0]$ that $v$ and $w$ are isolated are negatively correlated, that is, 
\beas
\IP[d_v=0,d_w=0] \le \IP[d_v=0]\IP[d_w=0], \qmq{and} \sigma^2_{n,m} \leq \min\{\mu_{n,m},2m\}.
\enas 
\end{lemma}
\begin{proof} Vertex $v$ is isolated if and only if none of the $n-1$ edges that connect $v$ to another vertex is included in the set of $m$ edges selected. Likewise, distinct vertices $v$ and $w$ are both isolated if and only if none of a particular set of $(n-2)+(n-2)+1$ edges is selected. Hence, the first claim is equivalent to
\beas 
\frac{{N-2n+3 \choose m}}{{N \choose m}} \le  \frac{{N-n+1 \choose m}^2}{{N \choose m}^2}
\qmq{or}  
 {N \choose m} {N-2n+3 \choose m}\le {N-n+1 \choose m}^2.
\enas 
Expanding the binomial coefficients and canceling common factors yields the equivalent form
\beas 
(N)_m (N-2n+3)_m  \le (N-n+1)_m^2,
\enas 
where $(n)_k=n(n-1)\cdots(n-k+1)$,
and pairing up the~$k^{\mathrm{th}}$ factors of the falling factorials we obtain
\beas
\prod_{k=0}^{m-1} (N-k)(N-2n+3-k) \le \prod_{k=0}^{m-1} (N- n+1- k)^2.
\enas
It suffices to show the inequality holds termwise. 
Expanding both sides of the~$k^{\mathrm{th}}$ term of each side and simplifying yields
\beas  
N + 2n \le n^2+1+k.
\enas
The case~$k=0$ implies all others, and reduces to
$
0 \le n^2-3n+2 = (n-2)(n-1),
$
and so holds for all~$n \ge 2$, thus proving the first claim.

Since the indicators of vertices being isolated are negatively correlated, we have
\be{
    \sigma^2_{n,m}\leq n\IP[d_v=0]\IP[d_v>0] \leq \min\{n\IP[d_v=0],n\IP[d_v>0] \},
}  
from which~$\sigma_{n,m}^2\leq \mu_{n,m}$ is immediate. As~$d_v\sim\Hyp(N,m,n-1)$ for~$N=n(n-1)/2$, using Lemma~\ref{lem3} we have
\be{
    \sigma_{n,m}^2 \le n\IP[d_v>0] \leq n\frac{m(n-1)}{N}=2m,
}
as claimed.
\end{proof}

\begin{lemma}\label{lem6} For~$n\geq 6$ and~$0\leq m\leq n^2/4-3n/2$, we have
\besn{\label{91} 
 \exp\bbclr{-\frac{2m}{n}-\frac{8m(m+n)}{n^3}}\leq \frac{\mu_{n,m}}{n}\leq \exp\bbclr{-\frac{2m}{n}},
}
and
\besn{ \label{92}
	&\mu_{n,m}\bbbcls{1-\frac{\mu_{n,m}}{n}\bbbclr{1+\frac{2m}{n}+\frac{78m(m+n)}{n^3}}}\\
	&\qquad\Leq \sigma_{n,m}^2 
    \Leq \mu_{n,m}\bbbcls{1-\frac{\mu_{n,m}}{n}\bbbclr{1+\frac{2m}{n}-\frac{48m(m+n)}{n^3}}}.
}
\end{lemma}
\begin{proof} Since the distribution of each individual degree is~$\Hyp(N,m,n-1)$, and as the hypothesis of  Lemma~\ref{lem3} holds due to the restriction assumed on $m$, 
it follows from that lemma that
\bes{
 \exp\bbclr{-\frac{m(n-1)}{N-m-n+2}}\leq \frac{\mu_{n,m}}{n}\leq \exp\bbclr{-\frac{2m}{n}},
}
yielding the upper bound in \eqref{91}. Since under the assertions on~$m$ and~$n$ we have~
\ben{\label{93}
  n^2-2m-3n+4 \geq n^2/2,
}
it follows that
\be{
  \frac{m(n-1)}{N-m-n+2} = \frac{2m}{n}+\frac{4m(m+n-2)}{n(n^2-2m-3n+4)}
  \leq \frac{2m}{n}+\frac{8m(m+n)}{n^3},
}
from which we obtain the lower bound in \eqref{91}.

In order to prove the upper and lower bounds on the variance, we use the fact that~$\Var(W)=\IE\clc{G(W'-W)}$ when~$(W,W',G)$ is a Stein coupling for a mean zero random variable~$W$; this identity follows immediately upon setting~$f(x)=x$ in \eqref{2}. Now recall \eqref{51}, \eqref{52} and \eqref{64}, and that~$N^v(m,\pi,\sigma_v)$ in \eqref{62} is the set of vertices that receive at least one edge when forming~$\cG^v(m,\pi,\sigma_v)$, and that $M^v(m,\pi,\sigma_v)$ in \eqref{63} is the set of all vertices $w \not =v$ such that $\{v,w\}$ is an edge in ${\mathcal G}(m,\pi)$, and does not receive a redistributed edge. As when $I_v(m,\pi)=1$ the sets $N^v(m,\pi,\sigma_v)$ and $M^v(m,\pi,\sigma_v)$ are empty, and recalling that $I_{w,1}(m,\pi)=\I[d_w(m,\pi)=1]$, we have
\besn{\label{94}
	&\sigma_{n,m}^2 \\
	&\quad = \IE\sum_{v\in[n]}\bbclr{I_v(m,\pi)-\frac{\mu_{n,m}}{n}}\\ &\quad\qquad\qquad\times\bbbclr{I_v(m,\pi)+\sum_{\substack{w\in\\ N^v(m,\pi,\sigma_v)}}I_w(m,\pi)
		-\sum_{\substack{w \in \\M^v(m,\pi,\sigma_v)}}I_{w,1}(m,\pi)}\\
	&\quad = n\IE\bbbclc{\bbclr{I_1(m,\pi)-\frac{\mu_{n,m}}{n}}\\ &\quad\qquad\qquad\times\bbbclr{I_1(m,\pi)+\sum_{\substack{w\in\\ N^1(m,\pi,\sigma_1)}}I_w(m,\pi)-\sum_{\substack{w \in\\ M^1(m,\pi,\sigma_1)}}I_{w,1}(m,\pi)}}\\
	&\quad = n\IE\bbbclc{I_1(m,\pi)\bbclr{1-\frac{\mu_{n,m}}{n}} -\frac{\mu_{n,m}}{n}\sum_{\substack{w\in\\ N^1(m,\pi,\sigma_1)}}I_w(m,\pi)\\
  &\kern20em+\frac{\mu_{n,m}}{n}\sum_{\substack{w \in\\ M^1(m,\pi,\sigma_1)}}I_{w,1}(m,\pi)}\\
	&\quad = \mu_{n,m}\bbbclr{1-\frac{\mu_{n,m}}{n} - \IE\sum_{\substack{w\in\\ N^1(m,\pi,\sigma_1)}}I_w(m,\pi)
		+ \IE\sum_{\substack{w \in\\ M^1(m,\pi,\sigma_1)}}I_{w,1}(m,\pi)}.
}
Now consider the first sum in \eqref{94}. Note that when  $d_1(m,\pi)=k$, of the potential $N$ edges, $n-1$ have vertex 1 as an endpoint, and an additional $m-k$ edges remain in ${\mathcal G}(m,\pi)$ and are not redistributed. Hence, 
\besn{ \label{95}
	&\IE\sum_{w\in N^1(m,\pi,\sigma_1)}I_w(m,\pi) \\
	&\quad = (n-1)\IP[2\in N^1(m,\pi,\sigma_1),d_2(m,\pi)=0] \\
	&\quad = \mu_{n,m}\frac{n-1}{n}\IP[2\in N^1(m,\pi,\sigma_1)|d_2(m,\pi)=0] \\
	&\quad = \mu_{n,m}\frac{n-1}{n}\sum_{k=0}^{n-2}\IP[2\in N^1(m,\pi,\sigma_1)|d_1(m,\pi)=k,d_2(m,\pi)=0]\\[-1.5ex]  &\kern16.5em\times\IP[d_1(m,\pi)=k|d_2(m,\pi)=0] \\
	&\quad = \mu_{n,m}\frac{n-1}{n}\sum_{k=0}^{n-2}
	\IP\bcls{\Hyp\bclr{N-(n-1)-(m-k),k,n-2}>0}\\[-1.5ex]
  &\kern16.5em\times\IP[d_1(m,\pi)=k|d_2(m,\pi)=0].
}
To arrive at the hypergeometric expression in the sum in the last equality from the conditional probability that vertex $2$ is incident on any of the $k$ redistributed edges that were removed from vertex $1$ when making the new graph, note that the total number of edges available is reduced from $N$ first by $n-1$, as vertex $1$ has been removed, and also due to the $m-k$ edges that were part of the original graph that are not changed. Of these remaining edges, $n-2$ are incident on vertex $2$, which is one fewer than their original number of $n-1$, due to the removal of vertex $1$.

Using Lemma~\ref{lem3},
\besn{\label{96}
	&\frac{k(n-2)}{N-n-m +k+1}
	-\frac{k^2(n-2)^2}{2(N-n-m +k+1)^2}\\
	&\qquad\qquad\leq
	\IP\bcls{\Hyp\bclr{N-(n-1)-(m -k),k,n-2}>0}\\
  &\kern18em       \leq \frac{k(n-2)}{N-n-m+k+1},
}
from which we obtain the upper bound
\bes{
	&\IE\sum_{w\in N^1(m,\pi,\sigma_1)}I_w(m,\pi) \\
	&\qquad \leq \frac{\mu_{n,m}}{n}
	\frac{(n-1)(n-2)}{(N-n-m+2)} \sum_{k =1}^{n-1} k \IP[d_1(m,\pi)=k|d_2(m,\pi)=0]\\
	&\qquad =  \frac{\mu_{n,m}}{n}
	\frac{(n-1)(n-2)}{(N-n-m+2)}\IE\clc{d_1(m,\pi)\given d_2(m,\pi)=0}.
}
Given~$d_2(m,\pi)=0$, we have $d_1(m,\pi)\sim\Hyp(N-(n-1),m,n-2)$, hence
\be{
	\IE\clc{d_1(m,\pi)\given d_2(m,\pi)=0} = \frac{m(n-2)}{N- n+1},
}
and so,
\besn{\label{97} 
	\IE\sum_{w\in N^1(m,\pi,\sigma_1)}I_w(m,\pi) 
	&\leq \frac{\mu_{n,m}}{n}
	\frac{m(n-1)(n-2)^2}{(N-n-m+2)(N-n+1)}\\
	&= \frac{\mu_{n,m}}{n}\bbcls{\frac{4m}{n}
	+ \frac{4m(2m+n-4)}{n(n^2-2m-3n+4)}}\\
	&\leq \frac{\mu_{n,m}}{n}\bbcls{\frac{4m}{n}
	+ \frac{16m(m+n)}{n^3}}.
}
Similarly, using the second moment expression from \eqref{72}
\be{
	\IE\bclc{d_1(m,\pi)^2\big| d_2(m,\pi)=0} = 
  \frac{m (n-2) (N+m n-2 n-3 m+3)}{(N-n) (N-n+1)},
}
and so from \eqref{96} we obtain the lower bound
\bes{
	&\IE\sum_{w\in N^1(m,\pi,\sigma_1)}I_w(m,\pi) \\
	&\qquad \geq \mu_{n,m} \frac{n-1}{n}
	\sum_{k=1}^{n-2} \bbclr{\frac{k(n-2)}{N-n-m +k+1}
		-\frac{k^2(n-2)^2}{2(N-n-m +k+1)^2}}\\
  &\kern20em\times \IP[d_1(m,\pi)=k|d_2(m,\pi)=0]\\
	&\qquad \geq \mu_{n,m} \frac{n-1}{n}
	\sum_{k=1}^{n-2} \bbclr{\frac{k(n-2)}{N-m -1}
		-\frac{k^2(n-2)^2}{2(N-n-m +2)^2}}\\
  &\kern20em\times\IP[d_1(m,\pi)=k|d_2(m,\pi)=0]\\
	&\qquad = \frac{\mu_{n,m}}{n} \bbbclr{\frac{(n-1)(n-2)}{N-m-1}
		\IE\bclc{d_1(m,\pi)\big| d_2(m,\pi)=0}\\
  &\kern8em
		-\frac{(n-1)(n-2)^2}{2(N-n-m+2)^2}\IE\bclc{d_1(m,\pi)^2\big| d_2(m,\pi)=0}}\\
	&\qquad = \frac{\mu_{n,m}}{n} \bbbclr{\frac{(n-1)(n-2)}{N-m-1}
		\frac{m(n-2)}{N-n+1}\\
	&\kern8em	-\frac{(n-1)(n-2)^2}{2(N-n-m+2)^2}\frac{m (n-2) (N+m n-2 n-3 m+3)}{(N-n) (N-n+1)}}.
}
Now, for the first term in the brackets we have
\bes{
  \frac{m(n-1)(n-2)^2}{(N-m-1)(N-n+1)}
  & =\frac{4m}{n}+\frac{4m(2m-n+2)}{n(n^2-2m-n-2)}  \\
  &  \geq \frac{4m}{n}-\frac{4mn}{n(n^2-2m-n-2)} 
   \geq\frac{4m}{n}-\frac{8m}{n^2},
}
where we have used \eqref{93} for the last inequality.
For the second term in the brackets,
\bes{
  &\frac{m(n-1)(n-2)^3 (N+m n-2 n-3 m+3)}{2(N-n) (N-n+1)(N-n-m+2)^2}
  = \frac{4 m (n-2)^2 (2 m+n-2)}{n \left(n^2-2 m-3 n+4\right)^2} \\
  &\qquad\leq\frac{8 m n^2 (m+n)}{n \left(n^2-2 m-3 n+4\right)^2} \leq \frac{32 m (m+n)}{n^3},
}
where again we have used \eqref{93} for the last inequality.
Hence, together with the upper bound \eqref{97}, we arrive at
\besn{ \label{98}
  &\frac{\mu_{n,m}}{n}\bbcls{\frac{4m}{n}-\frac{40m(m+n)}{n^3}}\\
  &\qquad\qquad\leq\IE\sum_{w\in N^1(m,\pi,\sigma_1)}I_w(m,\pi)
  \leq\frac{\mu_{n,m}}{n}\bbcls{\frac{4m}{n}+\frac{16m(m+n)}{n^3}}.
}
Now considering the second sum in \eqref{94}, we can write
\ben{\label{99}
  \sum_{\substack{w \in\\ M^1(m,\pi,\sigma_1)}}I_{w,1}(m,\pi)
  = \sum_{\substack{w:\{w,1\}\\\in\cG(m,\pi)}}I_{w,1}(m,\pi) - 
  \sum_{\substack{w \in\\ M^{c,1}(m,\pi,\sigma_1)}}I_{w,1}(m,\pi),
}
where $M^{c,1}(m,\pi,\sigma_1) = \{w:\{w,1\} \in {\mathcal G}(m,\pi), w \in N^1(m,\pi,\sigma_1)\}$.
Taking expectation of the first sum on the right hand side of \eqref{99} and noting that the distributions of the degrees in the graph are hypergeometric, we obtain that
\bes{
  &\IE\sum_{\substack{w:\{w,1\}\\\in\cG(m,\pi)}}I_{w,1}(m,\pi)\\	&\qquad=(n-1)\IP[\{1,2\}\in\cG(m,\pi),d_2(m,\pi)=1] \\
	&\qquad=(n-1)\IP[d_2(m,\pi)=1|\{1,2\}\in\cG(m,\pi)]\,\IP[\{1,2\}\in\cG(m,\pi)] \\
	&\qquad  =(n-1)\IP[\Hyp(N-1,m-1,n-2)=0]\frac{m}{N}\\
	&\qquad  =\frac{\mu_{n,m}}{n}\frac{N(n-1)}{N-m-n+2}\frac{m}{N}\\
	&\qquad=\frac{\mu_{n,m}}{n}\bbbcls{\frac{2m}{n}+\frac{4m(m+n-2)}{n(n^2-2m-3n+4)}}.
}
From this equality and using the assertions on~$m$ and~$n$, we obtain
\ben{\label{100}
	\frac{\mu_{n,m}}{n}\bbbcls{\frac{2m}{n}+\frac{2m(m+n)}{n^3}}\Leq\IE\sum_{\substack{w:\{w,1\}\\\in\cG(m,\pi)}}I_{w,1}(m,\pi)
	\Leq\frac{\mu_{n,m}}{n}\bbbcls{\frac{2m}{n}+\frac{8m(m+n)}{n^3}}.
}
Now taking expectation of the second sum of \eqref{99},  
\bes{
  &\IE\sum_{\substack{w\in\\M^{c,1}(m,\pi,\sigma_1)}}I_{w,1}(m,\pi) \\
  &\qquad = (n-1)\IP[\{1,2\}\in\cG(m,\pi),2\in N^1(m,\pi,\sigma_1),d_2(m,\pi)=1]\\
  &\qquad = (n-1)\sum_{k=1}^{n-1}\IP[\{1,2\}\in\cG(m,\pi),2\in N^1(m,\pi,\sigma_1),d_2(m,\pi)=1,d_1(m,\pi)=k]\\
  &\qquad = \frac{(n-1)m}{N}\sum_{k=1}^{n-1}\IP[2\in N^1(m,\pi,\sigma_1)\given d_2(m,\pi)=1,d_1(m,\pi)=k,\{1,2\}\in\cG(m,\pi)]\\
  &\kern10em\times \IP[d_1(m,\pi)=k\given d_2(m,\pi)=1,\{1,2\}\in\cG(m,\pi)]] \\[1ex]
  &\kern10em\times \IP[d_2(m,\pi)=1\given\{1,2\}\in\cG(m,\pi)]] \\
  &\qquad = \frac{(n-1)m}{N}\sum_{k=1}^{n-1}
  \IP[\Hyp(N-(n-1)-(m-k),k,n-2)>0]\\
  &\kern10em\times\IP[\Hyp(N-(n-1),m-1,n-2)=k-1]\\[1ex]
  &\kern10em\times\IP[\Hyp(N-1,m-1,n-2)=0].
}
We arrive at the first Hypergeomtric expression in the sum in the last equality by the same reasoning as that given following \eqref{95}; the remaining two expressions in the sum follow by similar, and simpler, means.

Now, for the first and last terms, using Lemma \ref{lem3} for the upper bound, we have
\bg{
	\IP[\Hyp(N-(n-1)-(m-k),k,n-2)>0]  \leq \frac{k(n-2)}{N-n-m+k+1}
	\\
	\IP[\Hyp(N-1,m-1,n-2)=0]  = \frac{\mu_{n,m}}{n}\frac{N}{N-m-n+2},
}
and thus, using in the final inequality that $n^2 \le 4(N-m-n+2)$, which holds via 
the assumption that $m\leq n^2/4-3n/2$, and that $n^2 \le 4(N-n)$, which holds as $n \ge 6$, true by assumption, we obtain 
\besn{\label{101}
  &\IE\sum_{\substack{w:\\M^{c,1}(m,\pi,\sigma_1)}}I_{w,1}(m,\pi) \\
  &\qquad \leq \frac{\mu_{n,m}}{n}\frac{(n-1)^2m}{(N-m-n+2)^2}\\
  &\kern7em\times \sum_{k=0}^{n-2}
  (k+1) \IP[\Hyp(N-(n-1),m-1,n-2)=k]\\
  &\qquad = \frac{\mu_{n,m}(n-1)^2m}{n(N-m-n+2)^2}\bbbclr{\frac{(m-1)(n-2)}{(N-(n-1))}+1}
  \leq \frac{\mu_{n,m}}{n}\frac{16m}{n^2}\bbbclr{\frac{4m}{n}+1}\\
  &\qquad 
  \leq \frac{\mu_{n,m}}{n}\frac{16m}{n^2}\bbbclr{\frac{4m}{n}+\frac{4n}{n}} \leq \frac{\mu_{n,m}}{n}\frac{64(m+n)}{n^3}.
}
Using the estimates from \eqref{100}
and \eqref{101} in the difference \eqref{99}, and then applying that result and \eqref{98} in 
\eqref{94} yields the claim. 
\end{proof}

\begin{lemma}\label{lem7} 
There exist universal integers~$m_0$ and~$n_0$, and positive constants~$C_0$ and~$c_0$ such that, whenever
\ben{\label{102}
	n\geq n_0 \quad\text{and}\quad
	m_0 \leq m\leq c_0 n^{3/2},
}
we have
\ben{\label{103}
	\bbabs{\frac{\mu_{n,m}}{ne^{-2m/n}} - 1} \leq C_0\bbclr{\frac{m}{n^2} + \frac{m^2}{n^3}}
}
and 
\ben{\label{104}
	\bbabs{\frac{\sigma_{n,m}^2}{n\phi(2m/n)} -1 }\leq C_0\bbclr{\frac{1}{m}+\frac{m^2}{n^3}},
}
where
\be{ 
  \phi(x) = e^{-x}(1-e^{-x}(1+x)).
}
\end{lemma}

\begin{proof} It is easy to verify that 
\bea \label{105}
n^{3/2} \le \frac{n^2}{4}-\frac{3n}{2} \qmq{for all $n \ge 27$.}
\ena
Hence, with the first inequality in \eqref{102} holding with $n_0$ replaced by 27, and taking $c_0 \le 1$, Lemma~\ref{lem6} can be invoked to yield
\bes{
	\bbabs{\frac{\mu_{n,m}}{ne^{-2m/n}}-1 }
	& \leq1-\exp\bbbclr{-\frac{8m(m+n)}{n^3}}\leq \frac{8m(m+n)}{n^3},
}
from which \eqref{103} now follows for any $C_0 \ge 8$.

Turning to \eqref{104}, we first show that the lower bound in \eqref{92} is positive whenever $n \ge 78$ and $m \ge 78$. Indeed, that lower bound is positive whenever
\be{ 
	\frac{\mu_{n,m}}{n}\bbbclr{1+\frac{2m}{n}+\frac{78m(m+n)}{n^3}} < 1,
}
which, recalling the upper bound \eqref{91}, is implied whenever
\ben{\label{106}
	e^{-x}(1+x+y) < 1
}
with
\ben{\label{107}
	x = \frac{2m}{n}\qquad \text{and}\qquad y=\frac{78m(m+n)}{n^3}.
}
Since \eqref{106} is equivalent to the inequality~$y < e^x - x - 1$, which in turn is satisfied if~$y\leq x^2/2$, since~$x^2/2< e^x - x - 1$, we arrive at the sufficient condition 
\be{ 
	\frac{78m(m+n)}{n^3} \leq \frac{2m^2}{n^2},
}
which is equivalent to~$39\leq m(2-39/n)$. This inequality holds whenever both $n \ge 78$ and $m \ge 78$. 

We now proceed to bound the ratio between the upper and lower bounds, say $\overline{\sigma}^2_{n,m}$ and $\underline{\sigma}^2_{n,m}$, respectively, of \eqref{92}. Using the identity~$(1-a)/(1-b)=1+(b-a)/(1-b)$, we have
\ben{\label{108}
\frac{\overline{\sigma}^2_{n,m}}{\underline{\sigma}^2_{n,m}}=	\frac{1-\frac{\mu_{n,m}}{n}\bbclr{1+\frac{2m}{n}-\frac{48m(m+n)}{n^3}}}{1-\frac{\mu_{n,m}}{n}\bbclr{1+\frac{2m}{n}+\frac{78m(m+n)}{n^3}}}
	 = 1+     \frac{\frac{\mu_{n,m}}{n}{\frac{126m(m+n)}{n^3}}}{1-\frac{\mu_{n,m}}{n}\bbclr{1+\frac{2m}{n}+\frac{78m(m+n)}{n^3}}}.
}
We proceed to lower bound the denominator in \eqref{108}. Letting~$x$ and $y$ be as in \eqref{107}, and applying the upper bound in \eqref{91}, we may write
\bes{
	& 1-\frac{\mu_{n,m}}{n}\bbclr{1+\frac{2m}{n}+\frac{78m(m+n)}{n^3}}\\
	&\qquad\geq
	1-e^{-2m/n}\bbclr{1+\frac{2m}{n}+\frac{78m(m+n)}{n^3}}\\
	& \qquad= 1-e^{-x}(1+x+y) \geq 1-e^{-x}(1+x) -y.
}
If~$2m/n\leq 1$, we have~$0\leq x\leq 1$ and thus~$1-e^{-x}(1+x) \ge x^2/4$ from Lemma~\ref{lem4}, so that
\be{
1-e^{-x}(1+x) -y \ge	\frac{x^2}{4} - y = \frac{m^2}{n^2}\bbclr{1-\frac{78}{n}-\frac{78}{m}}
	\geq \frac{1}{8}\bbclr{\frac{2m}{n}}^2
}
when~$\min(n,m) \ge 312$. If~$2m/n > 1$ and so~$x>1$, we simply use the lower bound
\be{
	1-e^{-x}(1+x) \geq \frac{1}{4},
}
and for any positive~$c_0$ we can take $n_0$ large enough so that
\be{
	1-e^{-x}(1+x)-y\geq \frac{1}{4} - y = \frac{1}{4} - \frac{78m^2}{n^3} - \frac{78m}{n^2}\geq \frac{1}{8}.
}
Hence, writing $\bigo(\cdot)$ with the understanding that the implied bound holds with universal constants, recalling \eqref{108}, and using Lemma~\ref{lem6} to bound $\mu_{n,m}/n$ in its numerator, we have
\ben{\label{109}
\frac{\overline{\sigma}^2_{n,m}}{\underline{\sigma}^2_{n,m}} \le	1+ \frac{8e^{-2m/n}}{\bclr{\frac{2m}{n}\wedge 1}^2}
	\bbclr{\frac{126m^2}{n^3}+\frac{126m}{n^2}} 
  = \begin{cases}
	1+ \displaystyle\bigo\bbclr{\frac{1}{m}} & \text{if~$2m/n\le 1$}\\[3ex]
	1+	\displaystyle\bigo\bbclr{\frac{m^2}{n^3}} & \text{if~$2m/n>1~$,}
	\end{cases}
}
where both the $\bigo(\cdot)$ terms are non-negative.

Next, with $\phi(x) = e^{-x}(1-e^{-x}(1+x))$, we show that
\besn{ \label{110}
\frac{\overline{\sigma}^2_{n,m}}{n\phi(2m/n)}
	& = \begin{cases}
		\displaystyle1+\bigo\bbclr{ \frac{1}{m}}&\text{if~$2m/n \le 1$}\\[3ex]
		\displaystyle1+\bigo\bbclr{\frac{m^2}{n^3}}&\text{if~$2m/n > 1$.} 
	\end{cases}
}
Using \eqref{103} for the second equality, \eqref{91} for the third, then \eqref{103} again and the lower bound of Lemma \ref{lem4} for the fourth, we obtain
\bes{
	&\frac{\overline{\sigma}^2_{n,m}}{n\phi(2m/n)}\\
	&\quad=\frac{\mu_{n,m}}{n e^{-2m/n}} \times
	\left[ \frac{{\displaystyle 1-\frac{\mu_{n,m}}{n}\bbclr{1+\frac{2m}{n}+\bigo\bbclr{\frac{m}{n^2}+\frac{m^2}{n^3}}}}}{{\displaystyle 1-e^{-2m/n}\bbclr{1+\frac{2m}{n}}}} \right]\\
	&\quad= \bbclr{1+\bigo\bbclr{\frac{m}{n^2}+\frac{m^2}{n^3}}} \\
	&\qquad\times\bbbbcls{1-\frac{\displaystyle\frac{\mu_{n,m}}{n}\bbclr{1+\frac{2m}{n}+\bigo\bbclr{\frac{m}{n^2}+\frac{m^2}{n^3}}}-e^{-2m/n}\bbclr{1+\frac{2m}{n}}}{\displaystyle1-e^{-2m/n}\bbclr{1+\frac{2m}{n}}}}\\
	&\quad= \bbclr{1+\bigo\bbclr{\frac{m}{n^2}+\frac{m^2}{n^3}}} \\
	&\qquad\times\bbbbcls{1
		-\frac{\displaystyle\bbclr{\frac{\mu_{n,m}}{n}-e^{-2m/n}}\bbclr{1+\frac{2m}{n}}}{\displaystyle1-e^{-2m/n}\bbclr{1+\frac{2m}{n}}}
		+\bigo\bbbbclr{\frac{\displaystyle e^{-2m/n}\bbclr{\frac{m}{n^2}+\frac{m^2}{n^3}}}{1- e^{-2m/n}(1+\frac{2m}{n})}}}\\ 
	&\quad= \bbclr{1+\bigo\bbclr{\frac{m}{n^2}+\frac{m^2}{n^3}}} \\
	&\qquad\times\bbbbcls{1
		+\bigo\bbbbclr{\frac{\displaystyle e^{-2m/n}\bbclr{\frac{m}{n^2}+\frac{m^2}{n^3}}\bbclr{1+\frac{2m}{n}}}{\bclr{\frac{2m}{n}\wedge 1}^2}}
		+\bigo\bbbbclr{\frac{\displaystyle e^{-2m/n}\bbclr{\frac{m}{n^2}+\frac{m^2}{n^3}}}{\bclr{\frac{2m}{n}\wedge 1}^2}}
	}\\
    &\quad =:(1+R_1)(1+R_2+R_3)=\bigo((1+R_1)(1+R_2)),
}
as $R_3 = \bigo(R_2)$. In the case $2m/n\leq1$, we have
\bes{
R_1&=\bigo \left( \frac{m}{n^2}+\frac{m^2}{n^3}\right) = \bigo\left( \frac{1}{m}\bbclr{\frac{m^2}{n^2}+\frac{m^3}{n^3}}\right) = \bigo\bbclr{\frac{1}{m}} \qm{and}\\
R_2&=\bigo\left(\frac{{m}/{n^2}+{m^2}/{n^3}}{\left( {2m}/{n} \right)^2} \right)=\bigo\bbclr{ \frac{1}{m}+\frac{1}{n}}=\bigo\bbclr{\frac{1}{m}},
}
showing the first bound in \eqref{110}. In the case $2m/n>1$,
\bes{
R_1=\bigo\left( \frac{m}{n^2}+\frac{m^2}{n^3}\right) = \bigo\bbclr{\frac{m^2}{n^3}},
}
and using that $x\exp(-x)$ is bounded over $[0,\infty)$, 
\bes{
	R_2 & =\bigo\bbbclr{\displaystyle e^{-2m/n}\bbclr{\frac{m}{n^2}+\frac{m^2}{n^3}}\bbclr{1+\frac{2m}{n}}}
	=\bigo\bbbclr{\displaystyle \bbclr{\frac{m^2}{n^3}}e^{-2m/n}\bbclr{\frac{2m}{n}}}=\bigo\bbclr{\frac{m^2}{n^3}}.
}
Applying \eqref{105}, the second bound in \eqref{110} is shown. Now, using that $\overline{\sigma}^2_{n,m}\ge \underline{\sigma}^2_{n,m}$, and writing
\be{
    \frac{\overline{\sigma}^2_{n,m}}{\underline{\sigma}^2_{n,m}} = 1+a,
    \qquad
    \frac{\overline{\sigma}^2_{n,m}}{n\phi(2m/n)} = 1+b
}
and observing that, because the implicit constants in the bounds \eqref{109} and \eqref{110} are universal, and using that the $\bigo(\cdot)$ terms in  \eqref{109} are non-negative, we can choose $c_0$ small enough and $m_0$ large enough to guarantee that $0\leq a<1$ and $-1<b<1$, and hence obtain the upper and lower bounds
	\bes{    
	&(1-a)(1+b)\leq \bbclr{1-\frac{a}{1+a}}(1+b)=\frac{\underline{\sigma}^2_{n,m}}{\overline{\sigma}^2_{n,m}} \frac{\overline{\sigma}^2_{n,m}}{n\phi(2m/n)} = \frac{\underline{\sigma}^2_{n,m}}{n\phi(2m/n)}\\
	&\qquad \leq \frac{\sigma^2_{n,m}}{n\phi(2m/n)}  \leq \frac{\overline{\sigma}^2_{n,m}}{n\phi(2m/n)} = 1+b,
}
from which the estimate \eqref{104} follows.

\end{proof}

\begin{lemma} \label{lem8} 
	Let~$r_{n,m}$ be defined as in \eqref{36}. For any integers~$\overline{n}$ and~$\overline{m}$ and any positive constant~$\overline{c}>0$, there exists~$\overline{r} \ge 1$ such that~$r_{n,m}>\overline{r}$ implies
\ben{\label{111}
n\geq\overline{n}  \quad\text{and}\quad \overline{m}\leq m\leq \overline{c}n^{3/2}.
}
\end{lemma} 
\begin{proof} We will show that $r_{n,m} \le \overline{r}$ for $\overline{r} = \max\bclc{\overline{n}^{1/2}, (2\overline{m})^{3/2}, 1/\overline{c}^2,1}$ if \eqref{111} is violated.
Indeed, if~$n < \overline{n}$, we have by Lemma~\ref{lem5}, and that $\mu_{n,m} \le n$, then 
\be{
 r_{n,m} =  \frac{\sigma_{n,m}^3}{\mu_{n,m}(1+(m/n)^2)} \leq \sigma_{n,m} \le  \min\{\overline{n}^{1/2},(2 \overline{m})^{1/2}\}.
}
Finally, if~$m> \overline{c}n^{3/2}$, then similarly
\bes{
  r_{n,m} = \frac{\sigma_{n,m}^3}{\mu_{n,m} \left(1+(m/n)^2 \right) }
  \leq \frac{\sqrt{\mu_{n,m}}}{1+(m/n)^2} \\
    \leq \frac{\sqrt{n}}{1+(\overline{c}n^{3/2}/n)^2}
  = \frac{\sqrt{n}}{1+(\overline{c}\sqrt{n})^2} \leq \frac{1}{\overline{c}^2}.\qedhere
}
\end{proof}

\begin{lemma} \label{lem9}
Letting~\Goodset be as in Condition~\eqref{6}, it holds that
\beas 
	\sup_{\substack{(n,m) \in \Goodset\\0\leq d\leq \min\{n,m\}/4}} \left( \frac{\mu_{n,m}^2}{\mu_{n-1,m-d}^2}
	\vee
	\frac{\mu_{n-1,m-d}^2}{\mu_{n,m}^2}
	\right)
	< \infty,
\enas
and
\bea \label{112}
		\sup_{\substack{(n,m) \in \Goodset\\0\leq d\leq \min\{n,m\}/4}} \left( \frac{\sigma_{n,m}^2}{\sigma_{n-1,m-d}^2}
		\vee
		\frac{\sigma_{n-1,m-d}^2}{\sigma_{n,m}^2}
		\right)
		< \infty.
\ena
\end{lemma}

\begin{proof} First, note that if~$(n,m)\in\Goodset$, then from \eqref{46} and \eqref{47} the conclusion of Lemma \ref{lem7} holds. For the ratio of means, from Lemma  \ref{lem6}, to upper bound $\mu_{n,m}/\mu_{n-1,m-d}$ it suffices to upper bound the ratio
\begin{align*}
&-\frac{2m}{n}+\frac{2(m-d)}{n-1}+\frac{8(m-d)(m-d+n-1)}{(n-1)^3}\\ &\qquad=\frac{2m-2nd}{n(n-1)}+\frac{8(m-d)(m-d+n-1)}{(n-1)^3}\le 8 \left( \frac{2m}{n^2}+\frac{8m(m+n)}{n^3}\right), 
\end{align*}
which is bounded by a constant via $m \le c_0 n^{3/2}$, as in  \eqref{102}. Similarly,  to upper bound $\mu_{n-1,m-d}/\mu_{n,m}$ it suffices to upper bound the ratio
\begin{align*}
&-\frac{2(m-d)}{n-1}+\frac{2m}{n}+\frac{8m(m+n)}{n^3} \\ &\qquad=\frac{2nd-2m}{n(n-1)}+\frac{8m(m+n)}{n^3} \le 2\left(\frac{2d}{n}+\frac{8m(m+n)}{n^3}  \right) , 
\end{align*}
which, here using that $d \le n/4$, we see is also so bounded.

For the ratios of variances, for $0 \le d \le \min \{n,m\}/4$ let 
\be{
  x = \frac{2m}{n},\qquad y = \frac{2m}{n} - \frac{2(m-d)}{n-1} = \frac{2(nd-m)}{n(n-1)},
}
let~$\phi(x) = e^{-x}(1-e^{-x}(1+x))$, and write
\be{
		\frac{\sigma_{n,m}^2}{\sigma_{n-1,m-d}^2} = \frac{\sigma_{n,m}^2}{n\phi(x)}\times\frac{(n-1)\phi(x-y)}{\sigma_{n-1,m-d}^2}
		\times\frac{n}{n-1}\times\frac{\phi(x)}{\phi(x-y)} =: R_1\times R_2\times R_3\times R_4.
}
We show that these four terms, and their reciprocals, can be uniformly bounded over the range of the supremum in \eqref{112}.
Since \eqref{102} holds for~$n,$ and $m$, we can apply Lemma~\ref{lem7}, and also \eqref{46} for the first and final bounds, and obtain
\ben{\label{113}
		\frac{1}{2}\leq 1-C_0\bbclr{\frac{1}{m}+\frac{m^2}{n^3}} \le  \frac{\sigma_{n,m}^2}{n\phi(x)} \le  1+C_0\bbclr{\frac{1}{m}+\frac{m^2}{n^3}}
    \leq \frac{3}{2}.
}

Next, since~$m\geq 2m_0$ by \eqref{46}
and~$d\leq m/4$, we have that~$m-d\geq 3m/4\geq 3m_0/2\geq m_0$. Since~$n\geq 2n_0$, again by \eqref{46}, we have that~$n-1\geq n_0$, and since~$m\leq (c_0/2) n^{3/2}$  by \eqref{46} and~$(n/(n-1))^{3/2}\leq 2$ for~$n\geq 3$, we have that~$m-d\leq m\leq c_0 (n-1)^{3/2}$. It follows that~$(m-d,n-1)$ also satisfies the hypotheses of Lemma \ref{lem7}. Using the lower bound on $\overline{n}$ from  \eqref{46}, we have $1/(n-1)^3\leq 2/n^3$, and also from \eqref{47} that $C_0 (2/m+2m^2/n^3) \le 1/2$, so also using $d \le m/4$ for the second and second to last inequality,
\besn{ \label{114}
		\frac{1}{2}&\leq 1-C_0\bbclr{\frac{2}{m}+\frac{2m^2}{n^3}}\leq 1-C_0\bbclr{\frac{1}{m-d}+\frac{(m-d)^2}{(n-1)^3}} \\
        &\kern1em\le  \frac{\sigma_{n-1,m-d}^2}{(n-1)\phi(x-y)}\\
        &\kern4em\leq 1+C_0\bbclr{\frac{1}{m-d}+\frac{(m-d)^2}{(n-1)^3}} \le  1+C_0\bbclr{\frac{2}{m}+\frac{2m^2}{n^3}}
    \leq \frac{3}{2}.
}
Hence, \eqref{113} and \eqref{114} imply that
\be{ 
  \frac{1}{2}\leq R_1\leq \frac{3}{2} \qmq{and}
  \frac{2}{3}\leq R_2\leq 2.
}
Clearly,~$1 \le R_3 \le 2$ for~$n \ge 2$. Lastly, 
\beas
R_4 =\frac{e^{-x}(1-e^{-x}(1+x))}{e^{-x+y}(1-e^{-x+y}(1+x-y))} =\frac{1-e^{-x}(1+x)}{e^y(1-e^{-x+y}(1+x-y))}.
\enas
Note that by \eqref{47}, and by \eqref{46} that gives that $\overline{c} \le 1$, and also using $d \le n/4$,
\begin{align}\label{115}
 - \frac{4}{n_0^{1/2}} \le  - \frac{4}{n^{1/2}}\le - \frac{2m}{n(n-1)} \le  y = \frac{2(nd-m)}{n(n-1)}\leq\frac{2d}{n-1} \le 1.
\end{align}

It follows that~$1/e^y$ remains bounded on~\Goodset, and therefore, to show $R_4$ is bounded it suffices to show that
\beas
\frac{1-e^{-x}(1+x)}{1-e^{-x+y}(1+x-y)}
\enas
remains bounded. Using Lemma~\ref{lem4},
\beas
\frac{1-e^{-x}(1+x)}{1-e^{-x+y}(1+x-y)} 
\le \frac{2\min\{x^2,2\}}{\min\{(x-y)^2,1\}}
=\frac{2\min\{x^2,2\}}{\min\{x^2(1-y/x)^2,1\}}.
\enas
But this ratio remains bounded from above, away from 1, as  $d  \le m/4$ implies
\be{
  \frac{y}{x}
  =\frac{nd-m}{m(n-1)} \leq \frac{2d}{m}-\frac{1}{n-1} \le \frac{1}{2}.
}
	
The reciprocal~$1/R_4$ is bounded similarly,  using that \eqref{115} shows that $e^y$ is bounded. 
\end{proof}

\section{Jack Measure on Tableaux}\label{sec3} 
We now turn to the study of the distribution of the standardized sum of the $\alpha$-contents over all boxes in a tableaux whose shape is determined by the partition~$\Lambda_n$ of $n$, that is, to
\bea \label{def2}
W = \frac{Y}{\sqrt{\alpha {n \choose 2}}}, \qmq{where} Y = \sum_{x \in
	\Lambda_n} c_{\alpha}(x),
\ena
where
\[ c_{\alpha}(x) = \alpha(\mbox{column number of $x - 1$}) - (\mbox{row number of $x - 1$}),\]
and where the partition $\Lambda_n$ is sampled from the Jack$_{\alpha}$ measure in \eqref{37}, as described in detail in the introduction; see  \eqref{38} for an illustration of $c_\alpha(x)$, where $x \in \Lambda_7$.

Our bound is based on the zero bias construction in \cite{FuGo11}, which itself depends on an exchangeable pair constructed using Kerov's growth process, a sequential procedure for growing a random partition distributed according to Jack$_\alpha$ measure.

The state of Kerov's growth process at times~$n=1,2,\ldots$ is a partition of~$n$, starting at time 1 with the unique partition (1) of 1. To describe its transition rule from time~$n-1$ to~$n$ for~$n \ge 2$, given a box~$x$ in the diagram  of a partition~$\Lambda_n$ of~$n$,
let $a(x)$ denote the number of boxes in the same
row of $x$ and to the right of $x$ (the ``arm'' of $x$), and let $l(x)$
denote the number of boxes in the same column of $x$ and below $x$
(the ``leg'' of $x$),
as in \eqref{37}. Now set
\beas c_{\Lambda}(\alpha) =
\prod_{x \in \Lambda} (\alpha a(x) + l(x) +1), \quad
c_{\Lambda}'(\alpha) = \prod_{x \in \Lambda} (\alpha a(x) +
l(x) + \alpha)
\enas
and, for~$\Lambda_{n-1}$ a partition of~$n-1$ obtained from~$\Lambda_n$ by removing a single corner box, let
\bes{
  & \psi_{\Lambda_n/\Lambda_{n-1}}'(\alpha) \\
  & \quad= \prod_{x \in
	C_{\Lambda_n/\Lambda_{n-1}}-R_{\Lambda_n/\Lambda_{n-1}}} \frac{(\alpha a_{\lambda}(x) +
	l_{\lambda}(x)+1)}{(\alpha a_{\lambda}(x) + l_{\lambda}(x)+\alpha)}
\frac{(\alpha a_{\Lambda_{n-1}}(x) + l_{\Lambda_{n-1}}(x)+\alpha)}{(\alpha a_{\Lambda_{n-1}}(x) +
	l_{\Lambda_{n-1}}(x)+1)},
}
where~$C_{\Lambda_n/\Lambda_{n-1}}$ is the union of columns of~$\Lambda_n$ that
intersect~$\Lambda_n-\Lambda_{n-1}$ and~$R_{\Lambda_n/\Lambda_{n-1}}$ is the union of rows of~$\Lambda_n$ that intersect~$\Lambda_n-\Lambda_{n-1}$.
If at stage~$n-1$ the state of the process is the partition~$\Lambda_{n-1}$, a transition to the partition~$\Lambda_n$ occurs with probability
\[ \frac{c_{\Lambda_{n-1}}(\alpha)}{c_{\Lambda_n}(\alpha)} \psi_{\Lambda_n/\Lambda_{n-1}}'(\alpha).\] It is shown in \cite{Ke69}, see also \cite{F4}, that if~$\Lambda_{n-1}$ is distributed according to Jack$_{\alpha}$ measure on partitions of~$n-1$, then the partition~$\Lambda_n$ obtained by this process at time~$n$ has the Jack$_{\alpha}$ distribution.

In the proof of Theorem 3.1 of \cite{FuGo11}, a variable having the zero bias distribution of $W$ was constructed as follows. Fix $n$ and $\alpha$ and let~$\Lambda_k$ be the state of Kerov's growth process at time~$k$, and set
\ben{\label{116}
  V=\sum_{x\in\Lambda_{n-1}}c_\alpha(x).
}
Denoting by~$c_\alpha(x_n)$ the content of the box $x_n$ added at time~$n$ to form~$\Lambda_n$, we can now write
\bea \label{117}
W = \frac{V}{\sqrt{\alpha {n \choose 2}}} + T, \qmq{where} T= \frac{c_\alpha(x_n)}{\sqrt{\alpha {n \choose 2}}}.
\ena
With~$dF(t|\Lambda_{n-1})$ the conditional distribution of~$T$ given~$\Lambda_{n-1}$, constructing the pair
\bea \label{118}
(T^\dagger, T^\ddagger) \sim (t''-t')^2 dF(t'|\Lambda_{n-1})dF(t''|\Lambda_{n-1})
\ena
on the same space as~$\Lambda_{n-1}$, and letting $U \sim \mathcal{U}[0,1]$ be independent of $V, T^\dagger$ and $T^\ddagger$,
the variable
\bea \label{119}
W^*=\frac{V}{\sqrt{\alpha {n \choose 2}}}  +T^* \qm{with~$T^*=UT^\dagger+(1-U)T^\ddagger$}
\ena
has the~$W$-zero bias distribution. In fact, the joint distribution on the right hand side of \eqref{118} can be achieved by 
running Kerov's growth process twice, conditionally independent on~$\Lambda_{n-1}$. As shown in \cite{FuGo11}, the resulting variables, say~$T'$ and $T''$, yield the crucial exchangeable Stein pair in \eqref{3} via \eqref{117}. Again by \cite{FuGo11}, both the conditional mean and variance of~$T$ given~$\Lambda_{n-1}$ do not depend on~$\Lambda_{n-1}$; specifically, 
\bea \label{120}
  \IE\clc{T\given\Lambda_{n-1}}=0 \qmq{and} 
  \IE \clc{T^2\given\Lambda_{n-1}}=\frac{2}{n}.
\ena
It is essentially for this reason that we may construct~$W^*$ as in \eqref{119}, using~$V$; for details, see \cite{FuGo11}.
\begin{proof}[Proof of Theorem~\ref{thm4}]
We verify the conditions of Theorem \ref{thm2}.

\begin{description}[leftmargin=0em,parsep=\parskip,listparindent=\parindent]
\item[Condition~(\ref{23}).]
Fix an~$\epsilon \in (0,1)$, suppressed in the notation, and let
\bea \label{121}
\Theta=\{(n,\alpha): \alpha > n^{1+\epsilon}, n \ge 2\} \qmq{and} \Goodset=\{(n,\alpha) \in \Theta:r_{n,\alpha} > 2^{1/2-\epsilon/2}\},
\ena
where
\bea \label{122}
r_{n,\alpha} = \frac{n}{\sqrt{\alpha}},
\ena
which is positive and measurable. Note that
\ben{\label{123}
    (n,\alpha)\in\Goodset\quad\iff\quad n^{1+\eps}<\alpha<\frac{n^2}{2^{1-\eps}},
}
which implies in particular that $n\geq 3$ if $(n,\alpha)\in\Goodset$.

From \cite{F1}, the mean and variance of the content~$Y$ of a tableaux of a partition of~$n$ under Jack$_{\rm \alpha}$ measure is given, respectively, by
\bea \label{124}
\mu_{n,\alpha} =0 \qmq{and} \sigma_{n,\alpha}^2 = \alpha {n \choose 2} \qmq{for all~$({n,\alpha}) \in \Theta$.}
\ena
In particular we have that $\Var_{n,\alpha} Y>0$ for all $(n,\alpha) \in \Goodset$.

\item[Condition~(\ref{24}).]  The variable $Y$, given in \eqref{def2} is easily seen to satisfy the needed conditions, and the construction of the zero bias variable~$W^*$ is outlined above in \eqref{117}, \eqref{118} and \eqref{119}.

\item[Condition~(\ref{25}).] From \eqref{117} and \eqref{119} we see that 
\beas 
D=T^*-T.
\enas 
For each~$({n,\alpha}) = \Goodset$ let~${\cal F}_{n,\alpha}$ be the trivial~$\sigma$-algebra~$\{\emptyset, \Omega\}$, let
\bea \label{125}
\overline{D}=\frac{10\sqrt{\alpha}}{n \epsilon} \qmq{and let} F_{n,\alpha,1} = \left\{\lambda_1 \le \frac{2}{\epsilon} \right\},
\ena
where~$\lambda_1$ and $\lambda_1'$ respectively denote the length of the first row and first column of the tableaux~$\Lambda_{n-1}$ produced by Kerov's growth process at time~$n-1$.
Clearly $\overline{D}$ is $\mathcal{F}_{n,\alpha}$ measurable.

We next argue that~$|D| \le \overline{D}$ on~$F_{{n,\alpha},1}$ as follows. With $c_\alpha(x_n), c_\alpha(x_n')$ and $c_\alpha(x_n'')$ the contents of the boxes added to $\Lambda_{n-1}$ by Kerov's growth process, all conditionally independent given $\Lambda_{n-1}$, with probability one,  
\beas 
\{c_\alpha(x_n), c_\alpha(x_n'), c_\alpha(x_n'')\} \subset [-(\lambda_1'+1),\alpha (\lambda_1+1)],
\enas
as the extreme values~$\alpha (\lambda_1+1)$ and~$-(\lambda_1'+1)$ are achieved, respectively, by adding a box at the end of first row, and at bottom of the first column. Scaling by $\sigma_{n,\alpha}$ in \eqref{124} to obtain $T$, $T'$ and $T''$, respectively, with probability one 
\bea \label{126}
\{T, T', T''\} \subset [-(\lambda_1'+1)/\sigma_{n,\alpha},\alpha (\lambda_1+1)/\sigma_{n,\alpha}].
\ena
Now note that by \eqref{118} the distribution of $(T^\dagger, T^\ddagger)$ is absolutely continuous with respect to that of $(T',T'')$, and hence with probability one
\beas
\{T, T^\dagger, T^\ddagger\} \subset [-(\lambda_1'+1)/\sigma_{n,\alpha},\alpha (\lambda_1+1)/\sigma_{n,\alpha}].
\enas

As $T^*$ is the convex combination $UT^\dagger + (1-U)T^\ddagger$ of $T^\dagger, T^\ddagger$, it too must lie in this same interval, and hence, as the length of the first column of $\Lambda_{n-1}$ can be no more than $n$, we obtain
\besn{ \label{127}
  |D| & = |T^*-T| \le \frac{\alpha \lambda_1  + \lambda_1'+\alpha +1}{\sigma_{n,\alpha}}\\
  \le & \frac{2\alpha/\epsilon  + n+\alpha +1}{\sigma_{n,\alpha}}
\le \frac{5 \alpha}{\epsilon \sigma_{n,\alpha}} \le 
	\overline{D} \qm{on~$F_{{n,\alpha},1}$ for all $(n,\alpha) \in \Goodset$.}
}

In what follows, we think of $(n,\alpha)\in\Goodset$ as fixed and suppress the subscript in $\IE_{n,\alpha}$. Turning to the moment conditions, we claim that
\bea \label{128}
\sqrt{\IE D^2}  \le C\left(\frac{1}{\sqrt{n}} + \frac{\sqrt{\alpha}}{n} \right) \le C\frac{\sqrt{\alpha}}{n}.
\ena
 Now, 
\besn{\label{129}
	\sqrt{\IE D^2} & = \sqrt{\IE(T^*-T)^2} \\ 
  & \le \sqrt{2(\IE\{(T^*)^2\}+\IE T^2)}
	 \le \sqrt{2}\left(\sqrt{\IE\{(T^*)^2\}}+\sqrt{\IE T^2} \right).
}
To bound the second moment of $T^*$, by the zero bias formula \eqref{4} with~$f(x)=x^3/3$, and the proof of Theorem 4.1 in \cite{FuGo11}, we obtain
\beas
3\Var(T)\IE\{(T^*)^2\} = \IE T^4\le \frac{8}{n^2}+ \frac{4\alpha}{n^2(n-1)}.
\enas
Hence, by \eqref{120},
\be{
  \sqrt{\IE\{(T^*)^2\}} = \sqrt{\frac{n}{6}\left( \frac{8}{n^2}+ \frac{4\alpha}{n^2(n-1)}\right)}\le C \sqrt{\frac{1}{n}+\frac{\alpha}{n^2}} 
\le C \left(\frac{1}{\sqrt{n}}+ \frac{\sqrt{\alpha}}{n}\right).
}
 For the second term of \eqref{129}, by \eqref{120}, we obtain~$\sqrt{\IE T^2}=\sqrt{2/n}$, thus showing first inequality in \eqref{128}. The final inequality in \eqref{128} holds as $(n,\alpha) \in \Goodset$ implies $\alpha \ge n$.

To verify the first condition in \eqref{26}, apply the Cauchy Schwarz inequality,  \eqref{122} and \eqref{128} to obtain
\bea \label{130} r^2_{n,\alpha}\,\IE\bclc{|D|(1-I_{F_{n,\alpha,1}})} \le r_{n,\alpha}^2 \sqrt{\IE D^2  \IP[F_{n,\alpha,1}^c]} \le \frac{C n}{\sqrt{\alpha}}\sqrt{\IP[F_{n,\alpha,1}^c]}.
\ena
To control $\IP[F_{n,\alpha,1}^c]$, with $m=n-1$, we apply 
the inequality 
\beas 
\IP[\lambda_1 = l] \le \left( \frac{m}{\alpha} \right)^l \frac{\alpha l}{l!^2}
\enas
from the proof of Lemma 6.6 in \cite{F1}. Using that~$\alpha \ge  n^{1+\epsilon} \ge m^{1+\epsilon}$ in the third inequality below we obtain 
\bes{
  \IP[F_{n,\alpha,1}^c] & \le	\IP[\lambda_1 \ge 2/\epsilon] \le \sum_{l \ge 2/\epsilon}\left( \frac{m}{\alpha} \right)^l \frac{\alpha l}{l!^2} = \frac{\alpha}{m^2}  \sum_{l \ge 2/\epsilon} \frac{m^{l+2}}{\alpha^l} \frac{l}{l!^2} \\ 
  & \le \frac{\alpha}{m^2}  \sum_{l \ge 2/\epsilon} m^{2-l \epsilon} \frac{l}{l!^2}
	\le \frac{\alpha}{m^2}  \sum_{l \ge 2/\epsilon} \frac{l}{l!^2} \le \frac{\alpha}{m^2}\sum_{l \ge 0} \frac{l}{l!^2} \le \frac{e\alpha}{m^2}\le \frac{4e \alpha}{n^2}.
}
Substitution into \eqref{130} now verifies the first condition in \eqref{26}.

For the second condition in \eqref{26}, 
using \eqref{122}, the Cauchy Schwarz inequality,  that~$\IE W^2=1$,  \eqref{128} and  \eqref{125} we obtain
\beas 
\sup_{(n,\alpha) \in \Goodset} r_{n,\alpha} \IE \bclc{|DW|+\overline{D}} \le \sup_{n,\alpha \in \Goodset} \frac{n}{\sqrt{\alpha}} \left( \sqrt{\IE D^2} + \IE |\overline{D}| \right) < \infty.
\enas

\item[Condition~(\ref{27}).] For~$(n,\alpha ) \in \Goodset$, let
\bea \label{131}
\Psi(n,\alpha)=(n-1,\alpha),
\ena
which is~${\cal F}_{n,\alpha}$ measurable, let~$F_{n,\alpha,2}=\Omega$, and let $V$ be as in \eqref{116}. The conditional distribution condition \eqref{28} is satisfied for~$V$ with~$\theta=(n,\alpha)$ by the properties of Kerov's growth process. Clearly the set~$F_{n,\alpha,2}$ is measurable with respect to~${\cal F}_{n,\alpha}$. The moment condition \eqref{29} is trivially satisfied, as~$1-1_{F_{n,\alpha,2}}=0$ almost surely. 

\item[Condition~(\ref{30}).] By \eqref{def2} and \eqref{116} we have that~$(Y-V)/\sigma_{n,\alpha}=T$ as in \eqref{117}, the scaled content~$c_{\alpha}(x_n)$ of the box~$x_n$ added at time~$n$ in Kerov's growth process. Hence, the first part of Condition \eqref{31} holds with~$\overline{B}=\overline{D}$ in \eqref{125}, as by \eqref{126}, and arguing as in \eqref{127}, we have
\beas
\frac{|Y-V|}{\sigma_{n,\alpha}}=|T| \le \frac{\alpha \lambda_1  + \lambda_1'+\alpha +1}{\sigma_{n,\alpha}}
\le 
\overline{D}\qmq{on~$F_{n,\alpha,1}$ for $(n,\alpha) \in \Goodset$.}
\enas

The second part of this condition holds easily, as 
\beas
r_{n,\alpha}^2 (\overline{D}(\overline{B}+\overline{D})) = 2r_{n,\alpha}^2 \overline{D}^2 = 200/\epsilon^2 \qm{almost surely.}
\enas

\item[Condition~(\ref{18}).] To verify the variance ratio condition \eqref{19}, recalling~$\sigma_{n,\alpha}^2$ from
\eqref{124} and $\Psi(\alpha,n)$ from 
\eqref{131}, we have
\beas
\frac{\sigma_{\alpha,n}^2}{\sigma_{\Psi(\alpha,n)}^2} = \frac{\alpha {n \choose 2}}{\alpha {n-1 \choose 2}} = \frac{n}{n-2} \le 3 \qm{for all~$(n,\alpha) \in \Goodset$,}
\enas
as~$n \ge 3$ for all~$(n,\alpha) \in \ \Goodset$ by the comment after \eqref{123}. For this same reason 
condition \eqref{20} holds, as
\beas
\frac{r_{\alpha,n}}{r_{\Psi(\alpha,n)}} = \frac{r_{\alpha,n}}{r_{\alpha,n-1}} = \frac{n}{n-1}\in [1,3/2]. 
\enas

\end{description}

Conditions~\eqref{23}--\eqref{30} and \eqref{18} have been verified, and Theorem \ref{thm4} now follows from Theorem~\ref{thm2}.
\end{proof}

The next result shows that the case when~$\alpha$ is taken larger than that in Theorem~\ref{thm4} is degenerate; the boundary case~$\epsilon=1$ is left unresolved.

\begin{theorem} \label{thm5}
For all~$\epsilon>1$, along any sequence~$\{(n,\alpha_n), n  \ge 1\}$ for which~$\alpha_n \ge n^{1+\epsilon}$, 
\beas
\lim_{n \rightarrow \infty} \IP_{n,\alpha_n}[\lambda_1'=n] =1.
\enas
\end{theorem}

\begin{proof}
Note that for all boxes~$x$ in the Tableaux with~$\lambda_1'=n$ we have~$a(x)=0$ and~$l(x)$ takes all values between~$0$ and~$n-1$. Hence, 
from the Jack$_{\alpha}$ measure distribution as given in \eqref{37},
\beas
\frac{1}{\IP_{n,\alpha_n}[\lambda_1'=n]} & =\frac{\prod_{l=0}^{n-1}(l+1)(l+\alpha_n)}{\alpha_n^n n!}
=\frac{\prod_{l=0}^{n-1}(l+\alpha_n)}{\alpha_n^n}\\
&= \prod_{l=0}^{n-1} \left(1+\frac{l}{\alpha_n} \right) \le \prod_{l=0}^{n-1} \exp \left( \frac{l}{\alpha_n}\right) \le \exp\left( \frac{n^2}{\alpha_n} \right).
\enas
Substituting the lower bound on~$\alpha_n$ into this inequality yields
\be{
\IP_{n,\alpha_n}[\lambda_1'=n] \ge \exp(-n^{1-\epsilon}) \rightarrow 1 \qm{as~$n \rightarrow \infty$.}\qedhere
}
\end{proof}

\begin{remark}  
The Wasserstein bound in \eqref{40} suggests that a bound in the Kolmogorov metric should hold with rate function 
\bea \label{132}
r_{n,\alpha}=\left(\frac{1}{\sqrt{n}} + \frac{\sqrt{\alpha}}{n} \right)^{-1} \qmq{for all $n \ge 2$ and $\alpha>0$.}
\ena
This rate function is equivalent to the one we take in \eqref{122} for the `large $\alpha$' parameter set \eqref{121}, as there $n \le \alpha$ and $1/\sqrt{n}$ is dominated by $\sqrt{\alpha}/{n}$. 
Directly extending the arguments used here to cover the `small' alpha regime requires that \eqref{130} hold for some choice of $F_{{n,\alpha},1}$. In particular, \eqref{128} shows that $\IE_{n,\alpha} D^2 \le C/r_{n,\alpha}^2$, with $r_{n,\alpha}$ as in \eqref{132}. Hence, taking this route, one needs to specify $F_{n,\alpha,1}$ as an appropriate restriction on $\Lambda_{n-1}$ that satisfies $\IP_{n,\alpha} [F_{n,\alpha,1}^c] < C/r_{n,\alpha}^2$, and which gives rise to a bounding $\overline{D}$ of the right order. If in this case $\overline{B}$ may be taken to be $\overline{D}$ as in (Z5) above, then $\overline{D}$ needs to be of order $1/r_{n,\alpha}$. 
\end{remark}

\section{Proof of Theorems \ref{thm1} and \ref{thm2}}
\label{sec4}
The proofs of Theorems~\ref{thm1} and~\ref{thm2} ultimately rely on obtaining information about the solution to a certain recursive inequality. In its simplest form, and closely related to the argument in \cite{Bolthausen84}, this inequality becomes
\ben{\label{133}
   a_n\leq q a_{n-1}+c \qm {for $n\geq 2$ and $a_1=1$}
}
for some $0<q<1$ and $c>0$. In this simple case, it is not difficult to solve the corresponding equality explicitly to yield
\be{
  a_n = q^{n - 1} + c\frac{ 1 - q^{n - 1}}{1 - q} \qm{for $n \ge 1$.}
}
What is important here is not the exact form of the solution but rather that $a_n$ is uniformly bounded over $n \ge 1$. We show below that this property holds in greater generality when we replace $n$ on the left hand side of \eqref{133} by a generic parameter $\theta\in\Theta$, and average the right hand side over a randomly chosen parameter $Y\in\Theta$, rather than evaluate at $n-1$. Although, in the general case, there may exist additional solutions to the inequality that are unbounded, it turns out that these solutions must grow exponentially fast along some sequence, which is a behavior that can be excluded in our applications.

\begin{lemma}\label{lem10} Let~$(\Theta,\cT)$ and~$(\Omega,\cF)$ be measurable spaces. For each~$\theta\in\Theta$, let~$\IP_{\theta}[\cdot]$ be a probability measure on~$\Omega$. Let~$X:\Theta\times\Omega\to [0,\infty)$ and~$\Psi:\Theta\times\Omega\to \Theta$ be such that, for each~$\theta\in \Theta$, both~$X(\theta,\cdot)$ and~$\Psi(\theta,\cdot)$ are measurable functions. Assume there are constants~$0<q<1$ and~$c>0$, measurable functions~$a:\Theta\to [0,\infty)$ and $r:\Theta\to [0,\infty)$, and a measurable set~$\Goodset\subset\Theta$ such that
	\bg{
		(\mathrm{A}1)\enskip
		\text{$\IE_\theta X =1$ for all~$\theta\in\Goodset$,}
		\qquad
		(\mathrm{A}2)\enskip\text{$\IE_\theta X = 0$ for all~$\theta\in\Theta\setminus\Goodset$,}\\[1ex]
		(\mathrm{A}3)\enskip
		\text{$a(\theta) \leq q \IE_\theta \clc{Xa(\Psi)} + c < \infty$ for all~$\theta\in\Theta$}, \\
		(\mathrm{A}4)\enskip \text{$a(\theta) \le r(\theta)$ and $\IP_{\theta}-\esup_{\{X>0\}} r(\Psi)  \le \frac{r(\theta)}{2q}$ for all $\theta \in \Goodset.$}
	}
	Then
	\begin{align*}
		\sup_{\theta\in \Theta} a(\theta) \leq \frac{c}{1-q}.
	\end{align*}
\end{lemma}

\begin{proof} Note that, for $\theta \in \Theta\setminus\Goodset$, the variable $X$ must be zero $\IP_\theta$-almost surely by~(A2), and so (A3) yields that 
\ben{\label{134}
	a(\theta)\leq c\qquad \text{for $\theta\in\Theta\setminus\Goodset$.}
}
We may therefore assume that~$\Goodset$ is non-empty, else the claim in trivial. We argue by contradiction; so assume  Conditions (A1)--(A4) are satisfied and that the opposite of the conclusion is true.
For every~$\theta\in\Goodset$, we can use (A1) and consider the probability measure~$\IP^X_\theta$ specified by its Radon-Nikodym derivative
\be{
	\frac{d \IP^X_\theta}{d\IP_\theta}=X, \qmq{so that}	\IE_\theta\clc{X a(\Psi)} = \IE^X_\theta \clc{a(\Psi)},
}
where~$\IE^X_\theta$ denotes expectation with respect to~$\IP^X_\theta$. 
We argue by contradiction, assuming that when 
	\begin{align} \label{135}
		\sup_{\theta \in \Theta} a(\theta) > \frac{c}{1-q}
	\end{align}
	and Conditions (A1)--(A4) hold, there exists a sequence~$\{\theta_n\}_{n \ge 0} \subset \Goodset$ and a constant $C$ such that, for all~$n\geq 0$,
\begin{align} \label{136}
	\frac{c}{1-q}+\frac{c\delta}{q^n} \leq  a(\theta_n) \leq r(\theta_n) \leq C/(2q)^n,
\end{align}
which is clearly impossible.

We proceed by induction. For the base case $n=0$, we note that since $a(\cdot)$ is bounded by $c$ on~$\Theta\setminus\Goodset$ by \eqref{134},  from \eqref{135} that there is~$\theta_0\in \Goodset$ such that~$a(\theta_0)=c/(1-q)+c\delta$, for some $\delta>0$; taking also $C=r(\theta_0)$, \eqref{136} is satisfied.

For the induction step, assume that the lower bound in~\eqref{136} is true for $n -1 \ge 0$. As~$\theta_{n-1} \in \Goodset$, Condition (A3) yields that $\IE_{\theta_{n-1}}^X a(\Psi)\geq (a(\theta_{n-1})-c)/q$, and so the integrand must be at least this lower bound on a set of positive $\IP^X_{\theta_{n-1}}$--measure; that is, 
\bm{
	A_{n-1} := \bclc{\omega\in\Omega\,:\, a\bclr{\Psi(\theta_{n-1},\omega)} \geq (a(\theta_{n-1})-c)/ q}\\
	\text{satisfies}
	\quad
	\IP^X_{\theta_{n-1}}\cls{A_{n-1}} > 0.
}
Moreover, by the definition of essential supremum,
\bm{
	B_{n-1} := \bclc{\omega\in\Omega\,:\,r\bclr{\Psi(\theta_{n-1},\omega)} \leq \IP_{\theta_{n-1}}\text{-}\esup r\bclr{\Psi}}
	\\  \text{satisfies} \quad\IP^X_{\theta_{n-1}}[B_{n-1}] = 1.
}
Hence~$\IP_{\theta_{n-1}}^X[A_{n-1}\cap B_{n-1}] = \IP_{\theta_{n-1}}^X[A_{n-1}]>0$, and we can find~$\theta_n\in \Theta$ satisfying
\ben{\label{137}
 a (\theta_n) \ge  \frac{a(\theta_{n-1})-c}{q}
	\quad\text{and}\quad
	r(\theta_n) \leq \IP^X_{\theta_{n-1}}\text{-}\esup r\bclr{\Psi(\theta_{n-1},\cdot)}.
}
Since $a(\cdot) \le c$ on~$\Theta\setminus\Goodset$ we conclude that~$\theta_n \in \Goodset$ in view of the first inequality of \eqref{137}, which also completes the induction for the lower bound in \eqref{135}.  Applying~\eqref{137} and (A4) yields 
\begin{align*}
	r(\theta_n) \leq \IP^X_{\theta_{n-1}}\text{-}\esup r\bclr{\Psi(\theta_{n-1},\cdot)} \le \frac{1}{2q}r(\theta_{n-1}),
\end{align*}
yielding the upper bound in \eqref{136}, and concluding the induction.
\end{proof}

\begin{proof}[Proof of Theorem~\ref{thm1}] Throughout the proof,~$C$ denotes a constant that does not depend on~$\theta$ and can change from formula to formula.  Note first that by Condition~\eqref{6} the bound \eqref{21} trivially holds for every~$\theta\in\Theta\setminus\Goodset$ by taking $C = \overline{r}$. Therefore we need only show that \eqref{21} holds for all~$\theta\in\Goodset$. Let
\ben{\label{138}
	\delta(\theta) = \left\{
	\begin{array}{cl}
		\sup_{z\in\IR}\abs{\IP_\theta[W\leq z] - \IP[Z\leq z]} & \theta \in \Goodset \\
		1 & \theta \in \Theta \setminus \Goodset.
	\end{array}
	\right.
}
Fix~$\eps>0$, whose exact value is to be chosen later, and for $z \in \mathbb{R}$ define
\be{
	h_{z,\eps} (x)= \begin{cases}
		1 & \text{if~$x\leq z$,}\\
		1 + (z-x)/\eps & \text{if~$z<x\leq z+\eps$,}\\
		0 & \text{if~$z+\eps<x$.}
	\end{cases}
}
Let~$f_{z,\eps}$ be the unique bounded solution to the Stein equation
\be{
	f_{z,\eps}'(x) - xf_{z,\eps}(x) = h_{z,\eps}(x) - \IE h_{z,\eps}(Z).
}
Using a standard smoothing inequality, see e.g. the proof of Theorem 5.1 in \cite{Chen10}, we have
\ben{\label{139}
	\delta(\theta) \leq \sup_{z\in\IR}\abs{\IE_\theta \clc{f_{z,\eps}'(W) - Wf_{z,\eps}(W)}} + \frac{\eps}{\sqrt{2\pi}}.
}
For ease of notation, we drop the indices~$z$ and~$\eps$ from~$f$.
\paragraph{Bound on~$\boldsymbol{\abs{\IE_\theta \clc{f'(W) - Wf(W)}}}$.}  Taking an arbitrary $\theta \in \Goodset$ and using the definition \eqref{2} of a Stein coupling in the second line below, we have
\bes{	
  &\abs{\IE_\theta \clc{f'(W) - Wf(W)}} \\
	&\qquad = \abs{\IE_\theta \clc{(1-GD)f'(W) - G\int_0^D (f'(W+t)-f'(W))dt}}\\
	&\qquad\leq \babs{\IE_\theta\bclc{f'(W)\IE_\theta \cls{ 1-GD \given W}}}
	+ \bbbabs{\IE_\theta\bbbclc{G\int_0^D\bclr{f'(W+t)-f'(W)}dt} }\\
	&\qquad =: R_1 + R_2.
}
From (4.6) and (4.7) of \cite{Chen04} we have, respectively, that~$\norm{f'}\leq 1$ and
\ben{\label{140}
	\abs{f'(x+t) - f'(x)} \leq \abs{t}\bbclr{1+\abs{x}+\frac{1}{\eps}\int_0^1 \I[z<x+ut\leq z+\eps]du},
}
implying,  by
the first condition in \eqref{10}, that
\ben{\label{141}
	R_1\leq \frac{C}{r_\theta} \qquad\text{for all~$\theta\in \Goodset$,}
}
and that
\bes{
	R_2 & \leq \IE_\theta\bbbclc{\abs{G}(1+\abs{W})\int_{0\wedge D}^{0\vee D}\abs{t}dt}\\
	&\qquad + \frac{1}{\eps}\IE_\theta\bbbclc{\abs{G}\int_{0\wedge D}^{0\vee D}\int_0^1 \abs{t}\I[z<W+ut\leq z+\eps]dudt} 
	 =: R_{2,1} + R_{2,2}.
}
Using the second condition in \eqref{10}, and that $|t| \le |D|$ in the integral, we have
\ben{\label{142}
	R_{2,1} \le \IE_\theta\bclc{(1+\abs{W})\abs{G}D^2} \leq \frac{C}{r_\theta} \qquad\text{for all~$\theta\in \Goodset$.}
}

Let~$F_\theta=F_{\theta,1} \cap F_{\theta,2}$. To handle the indicator in~$R_{2,2}$, write
\besn{ \label{143}
	& \I[z<W+ut\leq z+\eps]\\
	& \qquad \leq (1-I_{F_{\theta,1}}) +(1-I_{F_{\theta,2}}) +  I_{F_{\theta}}\I[z<W+ut\leq z+\eps].
}

Using \eqref{143}, and again that~$\abs{t}\leq\abs{D}$, we have
\ban{
	R_{2,2} & \leq \frac{1}{\eps}\IE_\theta\bbbclc{\abs{G}\int_{0\wedge D}^{0\vee D}\int_0^1 \abs{t}(1-I_{F_{\theta,1}})dudt +
		\abs{G}\int_{0\wedge D}^{0\vee D}\int_0^1 \abs{t}(1-I_{F_{\theta,2}})dudt} \notag \\
	&\quad+  \frac{1}{\eps}\IE_\theta\bbbclc{\abs{G}\int_{0\wedge D}^{0\vee D}\int_0^1 \abs{t}I_{F_{\theta}}\I[z<W+ut\leq z+\eps]dudt} \notag \\
	& \leq \frac{1}{\eps}\IE_\theta\bbbclc{\abs{G}D^2 (1-I_{F_{\theta,1}}) + \abs{G}D^2(1-I_{F_{\theta,2}})} \notag \\ &
	\quad+ \frac{1}{\eps}\IE_\theta\bbbclc{\abs{GD}\int_{0\wedge D}^{0\vee D}\int_0^1 I_{F_{\theta}\cap F_\circ}\I[z<W+ut\leq z+\eps]dudt} \notag \\ &
	\quad+  \frac{1}{\eps}\IE_\theta\bbbclc{\abs{G}D^2I_{F_{\theta}\cap F_\circ^c}} =: R_{2,2,1} + R_{2,2,2} + R_{2,2,3},\label{thm6}
}
where~$F_\circ = \{\Psi(\theta,\cdot)\in\Goodset\}$. Now, by \eqref{15} and the first condition of \eqref{12}
\ben{\label{144}
	R_{2,2,1} \leq \frac{C}{\eps r_\theta^2}.
}
Since~$F_{\theta}\cap F_\circ$ is contained in $F_\circ$ and~$\sigma_{\theta}>0$ for~$\theta\in\Goodset$, on this intersection we may define
\be{
	\~W = \frac{V-\mu_{\Psi}}{\sigma_{\Psi}},
}
and thus write
\be{
	W = \frac{\sigma_{\Psi}}{\sigma_\theta} \~W + \frac{Y-V}{\sigma_\theta} - \frac{\mu_\theta - \mu_\Psi}{\sigma_\theta} =: \rho \~ W + T_1 - T_2,
}
where~$\rho$,~$T_1$ and~$T_2$ are to be understood as random variables on~$\Goodset \times\Omega$. By the first condition in \eqref{17}, we have~$\abs{T_1}\leq \overline{B}$ on $F_\theta \cap F_\circ$. Hence,
\ban{ \label{145}
	& I_{F_\theta\cap F_\circ}I[z<W+ut\leq z+\eps] \\
	& \quad = I_{F_\theta\cap F_\circ} I \bbbcls{\frac{z-T_1+T_2-ut}{\rho} < \~W \leq \frac{z-T_1+T_2-ut + \eps}{\rho}} \notag \\
	& \quad \leq I_{F_\theta\cap F_\circ} I\bbbcls{\frac{z-{\overline B}+T_2-ut}{\rho} < \~W \leq \frac{z+{\overline B} +T_2-ut + \eps}{\rho} }\notag \\
	& \quad = I_{F_\theta\cap F_\circ} I \bbbcls{Q_{z,ut} - \frac{\overline B}{\rho}< \~W \leq Q_{z,ut} + \frac{\overline B+\eps}{\rho}}\notag
}
where
\be{
	Q_{z,y} = \frac{z+T_2-y}{\rho}
}
is~${\cal F}_\theta$ measurable by Condition~\eqref{13}. 

Note that~$F_\circ \in {\cal F}_\theta$ since~\Goodset, given in Condition~\eqref{6}, is in ${\cal T}$ and  $\Psi(\theta,\cdot)$ is ${\cal F}_\theta$-measurable by Condition~\eqref{13} for~$\theta \in \Goodset$.
Now using Condition~\eqref{11} to bound~$|D|$ by~$\overline{D}$ on~$F_{\theta,1}$, and applying the measurability of~$\overline{G}, \overline{D}$ and~$F_{\theta,2}$ with respect to~$\cF_\theta$ by Conditions~\eqref{11} and \eqref{13}, we obtain 
\besn{  \label{146}
	& R_{2,2,2} \\
	& \quad\le   \frac{1}{\eps}\IE_\theta\bbbclc{\overline{G}\,\overline{D}\int_{-\overline D}^{\overline D}\int_0^1 I_{F_{\theta}\cap F_\circ} \I\left[{\textstyle
		Q_{z,ut} - \frac{\overline B}{\rho}< \~W \leq Q_{z,ut} + \frac{\overline B+\eps}{\rho}}\right] dudt}  \\
	& \quad\le   \frac{1}{\eps}\IE_\theta\bbbclc{\overline{G}\,\overline{D}\int_{-\overline D}^{\overline D}\int_0^1 I_{F_{\theta,2}\cap F_\circ} \IP_\theta\bbcls{{\textstyle
		Q_{z,ut} - \frac{\overline B}{\rho}< \~W \leq Q_{z,ut} + \frac{\overline B+\eps}{\rho} \given {\cal F}_\theta
		}}dudt}.
}

Using \eqref{138} and \eqref{14} we obtain
\beas
\sup_{x \in \mathbb{R}}|\IP_\theta [\widetilde{W} \le x|{\cal F}_\theta] - \IP[Z \le x]|\le \delta(\Psi),
\enas
and as the normal density is bounded by~$1/\sqrt{2 \pi}$, using \eqref{19} we see that the integrand in \eqref{146} can be no more than 
\be{
	I_{F_{\theta,2}\cap F_\circ}\bbbclr{2\delta(\Psi) + \frac{2\overline B+ \eps}{\rho \sqrt{2\pi}}} 
	\leq C\,I_{F_{\theta,2}\cap F_\circ}\bbbclr{\delta(\Psi) + \overline B
		+ \eps}.
}
Therefore, using the second condition in \eqref{12} and the second inequality in \eqref{17} for the fourth inequality below, and then the first condition in \eqref{20} for the last, we obtain
\ban{
	R_{2,2,2} & \leq \frac{C}{\eps}\IE_\theta\bbbclc{\overline{G}\,\overline{D}^2 I_{F_{\theta,2}\cap F_\circ}\bbbclr{\delta(\Psi) + \overline B + \eps}} \notag\\
	& \leq \frac{C}{\eps}\IE_\theta\bclc{\overline{G}\,\overline{D}^2 I_{F_{\theta,2}\cap F_\circ}\delta(\Psi)} 
	+ \frac{C}{\eps} \IE_\theta\bclc{\overline{G}\,\overline{D}^2 \overline{B}I_{F_{\theta,2}}}
	+ C \IE_\theta\bclc{\overline{G}\,\overline{D}^2} \notag\\
	& \leq \frac{C\IE_\theta\bclc{\overline{G}\,\overline{D}^2}}{\eps}\IE_\theta\bbbclc{\frac{\overline{G}\,\overline{D}^2 I_{F_{\theta,2}\cap F_\circ}}{\IE_\theta \bclc{\overline{G}\,\overline{D}^2 I_{F_{\theta,2}}}}\delta(\Psi)} 
	+ \frac{C}{\eps} \IE_\theta\bclc{\overline{G}\,\overline{D}^2 
		\overline{B} I_{F_{\theta,2}}}
	+ C \IE_\theta\bclc{\overline{G}\,\overline{D}^2} \notag\\
	& \leq \frac{C}{\eps r_\theta}\IE_\theta\bbbclc{\frac{\overline{G}\,\overline{D}^2I_{F_{\theta,2}\cap F_\circ}}{\IE_\theta \bclc{\overline{G}\,\overline{D}^2 I_{F_{\theta,2}}}} \delta(\Psi)}
	+ \frac{C}{\eps r_\theta^2}
	+ \frac{C}{r_\theta} \notag\\
	& \leq \frac{C}{\eps r_\theta^2}\IE_\theta\bbbclc{\frac{\overline{G}\,\overline{D}^2I_{F_{\theta,2}\cap F_\circ}}{\IE_\theta \bclc{\overline{G}\,\overline{D}^2 I_{F_{\theta,2}}}} \delta(\Psi)r_\Psi}
	+ \frac{C}{\eps r_\theta^2}
	+ \frac{C}{r_\theta},  \label{147}}
where $R_{2,2,2}=0$ in the case~$\IE\bclc{\overline{G}\,\overline{D}^2 I_{F_{\theta,2}}}=0$, by the first line of the display above. 

In order to bound~$R_{2,2,3}$, using that~$\delta(\theta)=1$ for~$\theta \in \Theta \setminus \Goodset$ by \eqref{138} for the second equality, that~$F_\theta \subset F_{\theta,2}$ for the first inequality,  the first condition in \eqref{20} for the second inequality, and the second condition in \eqref{12} for the last, we have
\ban{
	R_{2,2,3} &= \frac{1}{\eps} \IE_\theta \bclc{\overline{G}\,\overline{D}^2 I_{F_{\theta}\cap F_\circ^c}} 
	 = \frac{1}{\eps} \IE_\theta\bclc{\overline{G}\,\overline{D}^2 I_{F_{\theta}\cap F_\circ^c}\delta(\Psi)}\\
	&\le \frac{1}{\eps} \IE_\theta\bclc{\overline{G}\,\overline{D}^2 I_{F_{\theta,2}\cap F_\circ^c}\delta(\Psi)}
	 \leq \frac{C}{\eps r_\theta} \IE_\theta\bclc{\overline{G}\,\overline{D}^2 I_{F_{\theta,2}\cap F_\circ^c}\delta(\Psi)r_{\Psi}} \notag \\
	& \leq \frac{C\IE_\theta\bclc{\overline{G}\,\overline{D}^2}}{\eps r_\theta} \IE_\theta\bbbclc{\frac{\overline{G}\,\overline{D}^2 I_{F_{\theta,2}\cap F_\circ^c}}{\IE_\theta\bclc{\overline{G}\,\overline{D}^2 I_{F_{\theta,2}}}}\delta(\Psi)r_{\Psi}} \notag \\
	& \leq \frac{C}{\eps r_\theta^2} \IE_\theta\bbbclc{\frac{\overline{G}\,\overline{D}^2 I_{F_{\theta,2}\cap F_\circ^c}}{\IE_\theta\bclc{\overline{G}\,\overline{D}^2I_{F_{\theta,2}} }}\delta(\Psi)r_{\Psi}}, \label{148}
}
where $R_{2,2,3}=0$ when~$\IE_\theta\bclc{\overline{G}\,\overline{D}^2 I_{F_{\theta,2}}}=0$, by the first line of the display.

Collecting the bounds \eqref{141}, \eqref{142}, \eqref{144}, \eqref{147} and \eqref{148} and using \eqref{139} we arrive at
\besn{\label{149}
	\delta(\theta) & \leq R_1 + R_{2,1} + R_{2,2,1} + R_{2,2,2} + R_{2,2,3} +\frac{\eps}{\sqrt{2\pi}} \\
	& \leq \frac{C}{\eps r_\theta^2}\IE_\theta\bbbclc{\frac{\overline{G}\,\overline{D}^2 I_{F_{\theta,2}}}{\IE_\theta\clc{\overline{G}\,\overline{D}^2 I_{F_{\theta,2}}}} \delta(\Psi)r_\Psi} + \frac{C}{\eps r_\theta^2} + \frac{C}{r_\theta} + C\eps.
}

Since Condition~\eqref{6} implies that~$\overline{r}$ is an upper bound on~$r_\theta$ for~$\theta \in \Theta \setminus \Goodset$, and a lower bound on~$r_\theta$ for~$\theta \in \Goodset$, we conclude that
\be{
    \sup_{\theta\in\Goodset}\IP_{\theta}\text{-}\esup_{\omega\in F_{\theta,2}\cap\{\Psi\in\Theta\setminus\Goodset\}}\frac{r_{\Psi(\theta,\omega)}}{r_\theta} < \infty. 
}
Hence, by the second condition in \eqref{20},
\ben{\label{150}
	q = \frac{1}{2}\left( 1\vee  \sup_{\theta\in\Goodset}\IP_{\theta}\text{-}\esup_{\omega\in F_{\theta,2}}\frac{r_{\Psi(\theta,\omega)}}{r_\theta} \right)^{-1}  \in (0,1).
}

Choosing~$\eps = C  / r_\theta q$ with~$C$ as in \eqref{149} and multiplying that inequality by~$r_\theta$ on both sides and then setting $a(\theta)=\delta(\theta)r(\theta)$ we obtain, for some possibly different constant~$c>0$, which does not depend on $\theta$ but may depend on~$q$,
\be{
	a(\theta) \leq q \IE_\theta\bbbclc{\frac{\overline{G}\,\overline{D}^2 I_{F_{\theta,2}}}{\IE_\theta\bclc{\overline{G}\,\overline{D}^2 I_{F_{\theta,2}}}} a(\Psi)} + c \qquad\text{for all~$\theta\in\Goodset$.}
}

We now verify the hypotheses of Lemma~\ref{lem10}, with the additional identification
\begin{align} \label{151}
X = \frac{\overline{G}\,\overline{D}^2 I_{F_{\theta,2}}}{\IE_\theta\clc{\overline{G}\,\overline{D}^2 I_{F_{\theta,2}}}}\I[\theta\in\Goodset].
\end{align}
Conditions (A1) and (A2) follow directly from the definition of $X$, while (A3) on $\Goodset$ is \eqref{151}, and is satisfied on $\Theta \setminus \Goodset$ as $\delta(\theta) \le 1$, and we may replace $c$ by $\max\{\overline{r},c\}$. Condition (A4) follows from \eqref{150}.
The conclusion of Lemma~\ref{lem10} now implies  that~$\delta(\theta)\leq C/r_\theta$ for all~$\theta\in\Theta$.
\end{proof}

\begin{proof}[Proof of Theorem~\ref{thm2}]
The proof for zero biasing is quite similar, but simpler, than the proof of Theorem \ref{thm1}; we only highlight the important differences. 

Recalling~$D=W^*-W$, applying the bound  \eqref{140}, and the zero bias characterization \eqref{4}, we obtain
\begin{multline}
	\left| \IE_\theta \left(f'(W)-W f(W) \right)\right|
	= \left| \IE_\theta \left(f'(W+D)-f'(W)\right) \right|\\
	\le \IE_\theta \left( |D| \left( 1+|W|+\frac{1}{\eps}\int_0^1 1_{[z,z+\epsilon]}(W+uD)du\right) \right).
	\label{152}
\end{multline}
Using \eqref{25}, noting in particular that $|D| \le |\overline{D}|$ on $F_{\theta,1}$, and the fact that $r_\theta > \overline{r}$ for $\theta \in \Goodset$ yields $1/r_\theta^2 \le C/r_\theta$, for the first two terms in \eqref{152},  we have
\beas 
\IE_\theta \clc{|D| +|DW|} \le \IE_\theta \clc{|D|(1-I_{F_{\theta,1}})} + \IE_\theta\clc{\overline{D}+|DW|}\le \frac{C}{r_\theta}.
\enas

Following the reasoning in \eqref{thm6} and labeling the corresponding terms that arise here in the same manner, for $R_{2,2}$, the only remaining term,   by the first condition in \eqref{26}, and \eqref{29}, we obtain the bound
\beas 
R_{2,2,1} \le \frac{1}{\eps} \IE_\theta \left( |D| (1-I_{F_{\theta,1}}) + |D|  (1-I_{F_{\theta,2}}) \right)\le \frac{C}{\eps r_\theta^2}.
\enas

For~$R_{2,2,2}$, as~$ut$ in \eqref{145} is replaced by~$uD$, separating the term that arises from $uD$ out of $Q_{z,y}$ as defined there, here we obtain
\begin{align*}
 I_{F_\theta\cap F_\circ}I[z<W+uD\leq z+\eps]  \le  I_{F_\theta\cap F_\circ} I \bbbcls{Q_z - \frac{{\overline B}+\overline{D}}{\rho}< \~W \leq Q_z + \frac{{\overline B}+\overline{D}+\eps}{\rho}},
\end{align*}
where~$Q_z = (z+T_2)/\rho$ is~${\cal F}_\theta$ measurable. Now arguing as in \eqref{147} we obtain
\beas
R_{2,2,2} & \leq \frac{C}{\eps}\IE_\theta\bbbclc{\overline{D} I_{F_{\theta,2}\cap F_\circ}\bbbclr{\delta(\Psi) + \overline B + \overline{D}+ \eps}}\\
& \leq \frac{C}{\eps r_\theta^2}\IE_\theta\bbbclc{\frac{\overline{D}I_{F_{\theta,2}\cap F_\circ}}{\IE\bclc{\overline{D} I_{F_{\theta,2}}}} \delta(\Psi)r_\Psi}
+ \frac{C}{\eps r_\theta^2}
+ \frac{C}{r_\theta}
\enas
using the second condition of \eqref{26} and the first one of \eqref{20} for the first term, and the second conditions of \eqref{31} and \eqref{26}, respectively, to obtain the last two terms in the bound.

As in \eqref{148}, using the first condition of \eqref{20} and the second condition of \eqref{26}, we obtain
\beas
R_{2,2,3} = \frac{1}{\eps}\IE_\theta \left\{
	\overline{D} I_{F_{\theta,2}\cap F_\circ^c}
	\right\}\le \frac{C}{\eps r_\theta^2} \IE_\theta \bbbclc{\frac{\overline{D} I_{F_{\theta,2}\cap F_\circ^c}}{\IE_\theta \bclc{\overline{D} I_{F_{\theta,2}} }}\delta(\Psi)r_{\Psi}}.
\enas
Combining terms as in \eqref{149} yields
\beas \delta(\theta) \leq \frac{C}{\eps r_\theta^2}\IE_\theta\bbbclc{\frac{\overline{D} I_{F_{\theta,2}}}{\IE_\theta\clc{\overline{D} I_{F_{\theta,2}}}} \delta(\Psi)r_\Psi} + \frac{C}{\eps r_\theta^2} + \frac{C}{r_\theta} + C\eps.
\enas
The proof can now be concluded as for Theorem \ref{thm1}.
\end{proof}

\section{Appendix} \label{sec5}
We illustrate two instances where the conditions in the General Framework of the Introduction are implicitly invoked. First we show that random version of the random variable~$Y$ at the (random) `smaller' parameter value is a random variable. The maps
\beas
(\theta,\omega) \rightarrow (\Psi(\theta,\omega),\omega) \qmq{and} (\theta,\omega) 
\rightarrow Y\bclr{\Psi(\theta,\omega),\omega}
\enas
are measurable, the first as each component is measurable, and the second being a composition of measurable maps.

Next, we show that if~$f(\theta,\omega)$ is measurable and $\IP_\theta$-integrable for all $\theta \in \Theta$, then 
\beas
\theta \rightarrow \int_{\Omega} f(\theta,\omega)dP_\theta(\omega)
\enas
is a measurable function of~$\theta$. Indeed, the collection ${\mathcal M}$ of subsets $E$ of $\Theta \times \Omega$ for which the integral of~$f(\omega,\theta)=I_E(\omega,\theta)$ is measurable with respect to $\IP_\theta$ is a monotone class. The class ${\mathcal M}$ contains the rectangles which are products of measurable sets $A$ and $B$, as their indicator
\beas
f(\theta,\omega) =\I[\theta \in A]\I[\omega \in B] 
\qmq{has integral} \int_{\Omega} f(\theta,\omega)d\IP_\theta(\omega)=\I[\theta \in A]\IP_\theta[B],
\enas
which is a product of measurable functions of~$\theta$. Hence ${\mathcal M}$ contains the algebra of all finite disjoint unions of such rectangles, and hence, by the Monotone Class theorem, the sigma-algebra these rectangle generate, that is, the product sigma-algebra. Given a non-negative integrable function $f(\theta,\omega)$, standard arguments using an approximating sequence of simple functions from below in concert with the Monotone Convergence Theorem yields the measurability of the integral of $f(\theta,\omega)$, and then for real valued functions by breaking up of any given integrable function into positive and negative parts.

\section*{Acknowledgements}

We are grateful to the referees for their detailed comments and references. This work was partially supported by the Singapore Ministry of Education AcRF Tier~1 Grants R-146-000-230-114 and R-155-000-167-112 through the National University of Singapore. The second author thanks the Department of Statistics and Applied Probability, National University of Singapore, for their kind hospitality.

\setlength{\bibsep}{0.5ex}
\def\bibfont{\small}

\end{document}